\documentclass{amsart}
\usepackage{verbatim}
\usepackage{amssymb}
\usepackage{amsmath}
\usepackage[usenames]{color}
\usepackage{graphicx}
\usepackage{amscd}
\usepackage[T1]{fontenc}  
\usepackage[ngerman]{babel}
\usepackage[numbers,sort,square]{natbib}
\usepackage[T1]{fontenc}

\usepackage{enumerate}

\theoremstyle{plain}
\newtheorem{theorem}{Theorem}

\newtheorem{proposition}[theorem]{Proposition}
\newtheorem{lemma}[theorem]{Lemma}
\newtheorem{corollary}[theorem]{Corollary}

\theoremstyle{definition}

\newtheorem{remark}[theorem]{Remark}

\numberwithin{equation}{section}
\numberwithin{theorem}{section}

\allowdisplaybreaks

\usepackage[colorinlistoftodos,prependcaption,color=yellow,textsize=tiny]{todonotes}

\def\R{{\mathbb R}}

\newcommand{\eps}{\varepsilon}

\definecolor{mycolour}{RGB}{70,70,255}
\newcommand{\dd}[0]{\mathrm{d}}
\newcommand{\ud}[0]{\,\mathrm{d}}

\newcommand{\vertiii}[1]{{\left\vert\kern-0.25ex\left\vert\kern-0.25ex\left\vert #1
    \right\vert\kern-0.25ex\right\vert\kern-0.25ex\right\vert}}

\begin{document}

\title[Stabilization by noise and enhanced dissipation]
{Stabilization by transport noise and enhanced dissipation in the Kraichnan model}

\author{Benjamin Gess}
\address{Max Planck Institute for Mathematics in the Sciences\\
Inselstra{\ss}e 22\\
04103 Leipzig
\&
Faculty of Mathematics\\
University of Bielefeld\\
Universit\"atsstra{\ss}e 25\\
33615 Bielefeld\\
Germany}
\email{benjamin.gess@mis.mpg.de}

\author{Ivan Yaroslavtsev}
\address{Department of Mathematics\\
University of Hamburg\\
Bundesstra{\ss}e 55\\
20146 Hamburg
\&
Max Planck Institute for Mathematics in the Sciences\\
Inselstra{\ss}e 22\\
04103 Leipzig\\
Germany}
\email{yaroslavtsev.i.s@yandex.ru}

\date{\today}

\begin{abstract}
Stabilization and sufficient conditions for mixing by stochastic transport are shown. More precisely, given a second order linear operator with possibly unstable eigenvalues on a smooth compact Riemannian manifold, it is shown that the inclusion of transport noise can imply global asymptotic stability. Moreover, it is shown that an arbitrary large exponential rate of convergence can be reached, implying enhanced dissipation. The sufficient conditions are shown to be satisfied by the so-called Kraichnan model for stochastic transport of passive scalars in turbulent fluids. In addition, an example is given showing that it can be sufficient to force four modes in order to induce stabilization. 
\end{abstract}



\maketitle


\section{Introduction}\label{sec:intro}

Let $d\geq 2$ and let $\mathcal M$ be a $d$-dimensional $C^{\infty}$-smooth connected compact Riemannian manifold. In this work, we consider the following linear second order stochastic PDE with transport noise on  $\mathbb R_+ \times \mathcal M$,
\begin{equation}\label{eq:mainstochdiffeq}
\begin{cases}
 \ud u_t + A\sum_{k\geq 1} \langle \sigma_k, \nabla u_t\rangle_{T\mathcal M} \circ \ud W^k_t & = (T u_t + Cu_t) \ud t,\\
 u_0 &= u\in L^2(\mathcal M),
\end{cases}
\end{equation}
where 
$(W^k)_{k\geq 1}$ are independent real-valued Brownian motions, $A \geq 0$, $C \in \mathbb R$, 
$(\sigma_k)_{k\geq 1}$ are divergence-free $C^{\infty}$-smooth vector fields and
\begin{equation}\label{eq:INTROTda}
 T f:= \sum_{m=1}^n \Bigl\langle\chi_m, \nabla \bigl \langle \chi_m, \nabla f\bigr \rangle \Bigr\rangle,\;\;\; f\in C^{\infty}(\mathcal M),
\end{equation}
is a second-order strictly elliptic operator with $(\chi_m)_{m=1}^n$ divergence-free, smooth tangent vector fields on $\mathcal M$. 
 A model case is given by $T = \Delta$, and $\mathcal M = \mathbb T^d$. 

Stochastic equations of the form \eqref{eq:mainstochdiffeq} describe the evolution of a passive scalar in a turbulent fluid, see, for example \cite{MK99,Fl11}. In particular, for a certain choice of $(\sigma_k)_{k\geq 1}$ equation \eqref{eq:mainstochdiffeq} corresponds to the {\em Kraichnan model} from turbulence theory, see Section \ref{sec:Kraichnan} below. Another example of \eqref{eq:mainstochdiffeq} arises in magnetohydrodynamics, for example studied by Baxendale and Rozovskii in \cite{BR92}, see Example \ref{ex:BRMHDmodel} below. In the latter case, the corresponding sum over $k$ in \eqref{eq:mainstochdiffeq} is finite.

The motivation of the present work is twofold: The first aim is to prove the possibility of stabilization by noise for stochastic PDE with transport noise. The second aim is the derivation of sufficient conditions on the coefficients $\sigma_k$ in \eqref{eq:mainstochdiffeq} to  imply a.s.\ exponential mixing for \eqref{eq:mainstochdiffeq}. We will next describe each of these aspects in some more detail.

The possibility of stabilizing linear SDE by noise was analyzed by Arnold, Crauel, and Wihstutz in \cite[Theorem 2.1]{ACW83} (see also \cite{Ar90}). More precisely, let $n\geq 2$ and let $T \in \mathcal L(\mathbb R^n)$ be a self-adjoint linear operator with a negative trace, $\text{tr}\; T<0$. Then, if $T$ has unstable eigenvalues, the solution to 
$ \ud x_t = Tx_t \ud  t$ blows up exponentially in time. The results of \cite[Theorem 2.1]{ACW83} imply that the inclusion of linear noise can induce global asymptotic stability, in the sense that there exist skew-symmetric matrices $D_1, \ldots, D_{n-1}\in \mathcal L(\mathbb R^n)$ and $A>0$ such that the solution $x_t$ to
\begin{equation}\label{eq:INFTOewithSTA}
 \ud x_t + A\sum_{k=1}^{n-1}D_kx_t\circ dW^k_t = Tx_t  \ud  t,\;\;\; t\geq 0,
\end{equation}
a.s.\ exponentially decays with an exponential rate  arbitrarily close to $\text{tr}\; T$ for $A$ large enough. The natural question of the possibility of an extension to infinite dimensions was posed by Capi\'nski in the late 80s. Choosing $\mathcal M=\mathbb T^d$, $T= \Delta+C$, in \eqref{eq:mainstochdiffeq} 
we have, informally, $\text{tr}\; T=-\infty$. Consequently, one may conjecture that a choice of diffusion coefficients $\sigma_k$ in \eqref{eq:mainstochdiffeq} is possible, implying an exponential rate of decay for the solution to \eqref{eq:mainstochdiffeq} with exponential rate arbitrarily close to $-\infty$. Following Flandoli and Luo \cite{FL19} this problem is called {\em Capi\'nski's conjecture}. 

The first result of this work demonstrates that, if the sum of the coefficients $\sigma_k$ in \eqref{eq:mainstochdiffeq} is smooth,  this conjecture can be deduced by combining results by Dolgopyat, Kaloshin and Koralov \cite{DKK04} with a recent argument by Bedrossian, Blumenthal and Punshon-Smith \cite{BBP-S} and a stopping time argument (see Section \ref{sec:Proofofmainthm} below). With an eye on the Kraichnan model considered below, we include the case of irregular coefficients $\sigma_k$, thereby identifying and including the optimal range of regularities for which these arguments can be used (cf.\ also p.4 below). For irregular coefficients the results of \cite{DKK04} are not directly applicable, but have to be carefully modified and generalized. 

More precisely, we deduce that the following sufficient conditions on the coefficients $(\sigma_k)_{k\geq 1}$, inspired by the ones used in \cite{DKK04,BS88},  imply stabilization by noise. In an informal form, the assumptions read

\begin{enumerate}
 \item[\textnormal {\bf (a)}] The coefficients $(\sigma_k(x_1), \sigma_k(x_2))_{k\geq 1}$ are {strictly elliptic} for any off-diagonal point $(x_1, x_2)\in \mathcal M^2$. This can be relaxed to the H\"ormander condition if the sequence $(\sigma_k)$ is finite, see Section \ref{subsec:Hormtwopint} below.
 \item[\textnormal {\bf (b)}] The coefficients $(\tilde {\sigma}_k)_{k\geq 1}$ of the {normalized tangent flow}, see \eqref{eq:defofotildesigmakdacq} below, are strictly elliptic.
 \item[\textnormal {\bf (c)}] Summability and H\"older continuity of the series $\sum_{k\ge1}\sigma_k$ and their first and second derivatives. 
\end{enumerate}
The rigorous form of the conditions
 {\textnormal {\bf (a)--(c)}} are given in full detail in assumptions {\textnormal {\bf (A)--(C)}} below.
\begin{theorem}\label{thm:proofofCadpconj}
 Let $(\sigma_k)_{k\geq 1}$ satisfy the conditions {\textnormal {\bf (A)--(C)}} and let $T$ be a second-order strictly elliptic operator defined by \eqref{eq:INTROTda}. Then there exist $A_0>0$ and $\gamma_0>0$ such that for any $\gamma\in(0, \gamma_0]$ and $A\geq A_0$ there is $D^{A,\gamma}:\Omega\to \mathbb R_+$ with  $\sup_{A\geq A_0}\mathbb E( D^{A,\gamma})^p <\infty$ for any $1\leq p<\tfrac{9d \gamma_0}{4\gamma}$ such that the solution $(u_t)$ of \eqref{eq:mainstochdiffeq} satisfies a.s.
\begin{equation}\label{eq:INTROmainthmwithCiqo}
 \|u_t\|_{L^2} \leq A^2 D^{A,\gamma} e^{(C-\gamma A^2)t}\|u\|_{L^2},\;\;\; t>0.
\end{equation}
\end{theorem}

The proof of Theorem \ref{thm:proofofCadpconj} is given in Section \ref{sec:Proofofmainthm} below.

As addressed in detail below, we verify the conditions  {\textnormal {\bf (A)--(C)}} in two particular examples: First, for the \textit{Kraichnan model} arising in the analysis of turbulent fluids, for which the series in \eqref{eq:mainstochdiffeq} is infinite, causing limited regularity of the stochastic flow. Second, for a two-dimensional example due to Baxendale and Rozovskii. In this example, only four modes have to be forced, that is $\sigma_k \equiv 0$ for $k\ge5$, to induce stabilization by noise.

We note the relation to the enhancement of diffusive mixing studied in the deterministic setting by Constantin, Kiselev, Ryzhik, and Zlato\v{z} in \cite{CKRZ}, where necessary and sufficient conditions on the coefficient $\sigma$ are derived, implying that the solution $u_t$ of the PDE
\begin{equation}\label{eq:introCKRZcp}
   \ud u_t = \Delta u_t\ud t- A\langle \sigma, \nabla u_t\rangle_{T\mathcal M}\ud t,\;\;\; t\geq 0,
 \end{equation}
converges to zero in $L^2(\mathcal M)$ {\em for any fixed $t>0$} as $A\to \infty$. 

\smallskip

We next address the second main aspect of this work: stochastic and turbulent mixing of passive scalars. Let $u_0:\mathcal M \to \R_+$ be the density of a solute in a solvent with dynamics on $\mathcal M$ governed by a divergence-free vector field $\sigma: [0,T]\times\mathcal M \to T\mathcal M$. Then, the density $u_t$ of the solute at time $t\geq 0$ is given as the solution to the transport equation
\begin{equation}\label{eq:INTROpuretranseq}
\ud u_t +\langle  \sigma(t), \nabla u_t\rangle_{T\mathcal M}\ud t = 0.
\end{equation}
The concept of mixing measures how ``well-spread'' the density $u_t$ is. 
Mixing of passive scalars can be understood as weak-$L^2$ convergence of the density $u_t$ to its mean $\bar u = \int_{\mathcal M}u_0$. For the relation to notions of mixing in dynamical systems theory see e.g.\ \cite{Th12}. This weak-$L^2$ convergence is equivalent to convergence of $u_t$ to $\bar u$ in the Sobolev space $H^{-s}$ for any $s>0$ (see e.g.\ \cite{LTD11,Th12}). In the deterministic setting, the analysis of necessary and sufficient conditions on the drift field $\sigma$ implying mixing and rates of mixing for \eqref{eq:INTROpuretranseq} has attained considerable interest in the literature, see, for example 
Alberti, Crippa, and Mazzucato \cite{ACM19} and the references therein. 

More recently, the case of a solvent in a stochastic fluid has been considered by 
Bedrossian, Blumenthal, and Punshon-Smith in \cite{BBP-S,BBP-S1911a,BBP-S18}. In this case, the flow $\sigma$ in \eqref{eq:INTROpuretranseq} is chosen as the solution to the two-dimensional stochastic Navier-Stokes equation with random excitation of small modes, and it is shown that for a.e.~realization \eqref{eq:INTROpuretranseq} becomes mixing. With an eye on mixing in turbulent fluids, in the present work, the solution to the stochastic Navier-Stokes equation is replaced by the so-called Kraichnan model. 
In a turbulent regime, based on the statistics provided by the Kolmogorov turbulence theory (see, for example, \cite[Section 4]{MK99}), the Kraichnan model (see \cite{Kr68} and Section \ref{sec:Kraichnan} below) describes turbulent fluids by means of a random field, rapidly decorrelating in time and decorrelating in space in accordance to the Kolmogorov statistics. More precisely, in this case 
 we have $\sigma(t) := \sum_{k\geq 1} \sigma_k \dot{W}^k_t$ for a certain choice of coefficients $(\sigma_k)_{k\geq 1}$, determining the spatial decorrelation. Hence, \eqref{eq:INTROpuretranseq} becomes a stochastic transport equation
 \begin{equation}\label{eq:INTROstochtranseq}
  \ud u_t + \sum_{k\geq 1}
  \langle  \sigma_k(t), \nabla u_t\rangle_{T\mathcal M} \circ \ud {W}^k_t = 0.
  \end{equation}

 The second result of this work proves a.s.~exponential mixing for \eqref{eq:INTROstochtranseq} under the same assumptions {\textnormal {\bf (a)--(c)}}. 
 
 These conditions are shown to be satisfied by the Kraichnan model in regimes corresponding to non-smooth coefficients in \eqref{eq:INTROstochtranseq}, while still offering sufficient regularity to be accessible by an analysis of the Lyapunov exponents and stochastic flows: More precisely, we require, roughly speaking, $\sum_{k\geq 1} (\sigma_k,\sigma_k):\mathcal M \times \mathcal M \to T(\mathcal M\times \mathcal M)$  to be $C^{2+\beta}$ for some $\beta>0$. \label{pg:whyonly2+beta} This is the optimal range for the argument used in the present work, since for $\beta <0$ the Lagrangian dynamics lack uniqueness and branching of particles occurs (see \cite[Theorem 9.4 and 10.1]{LJR02}). In addition, \cite[Corollary 3.3 and p.\ 331]{Bax86} implies that the top Lyapunov exponent diverges to infinity for $\beta \searrow 0$. We note that the Kraichnan model corresponding to Kolmogorov turbulence leads to  $\sum_{k\geq 1} (\sigma_k,\sigma_k)$ being at most in $C^{4/3}$, i.e.\ $\beta=-\frac{2}{3}$ above. 

The results of the present work also includes the viscous case, for $\kappa \ge 0$,
\begin{equation}\label{eq:introkappaisnmwq}
  \begin{cases}
  \ud u_t + \sum_{k\geq 1} \langle\sigma_k , \nabla u_t \rangle_{T\mathcal M}\circ \ud W^k_t  \!\!\!& = \kappa  T u_t \ud t,\\
 u_0 &= u.
\end{cases}
\end{equation}

\begin{theorem}\label{thm:H-sconvergensofHsinvar}
 Let $(\sigma_k)_{k\geq 1}$ satisfy the conditions {\textnormal {\bf (A)--(C)}} and let $(u_t)_{t\geq 0}$ be the solution of \eqref{eq:introkappaisnmwq} with $u$ being  mean-zero in $H^s(\mathcal M)\cap H^1(\mathcal M)$ for some $s\in(0,1+\beta/2)$. Then there exist $\kappa_0>0$ and $\gamma_0>0$ such that for any $\kappa\in[0, \kappa_0]$ and $\gamma\in(0, \gamma_0)$ there exists $D_{\kappa,\gamma}:\Omega \to [1, \infty)$ so that a.s.\
 \begin{equation}\label{eq:INTROH-sHsineqsfa}
  \|u_t\|_{H^{-s}} \leq D_{\kappa,\gamma} e^{-\gamma st} \|u\|_{H^s},\;\;\; t> 0.
 \end{equation}
 Moreover, $\sup_{\kappa\in[0, \kappa_0]}\mathbb E |D_{\kappa,\gamma}|^p<\infty$ for any $1\leq p<\frac{9d\gamma_0}{2s\gamma}$.
\end{theorem}

\medskip

We next sketch some aspects of the proof of Theorem \ref{thm:proofofCadpconj}. By rescaling, \eqref{eq:mainstochdiffeq} is equivalent to \eqref{eq:introkappaisnmwq}, see Subsection \ref{sec:Proofofmainthm} below, with $A\to \infty$ corresponding to $\kappa\to 0$. Next, we consider the Lagrangian level, that is the characteristics of \eqref{eq:introkappaisnmwq} and show that the corresponding two-point motion Markov semigroup $(P^{(2), \kappa}_t)_{t\geq 0}$ is ergodic. More precisely, we show that there exists $p>0$ and $\alpha>0$ such that for any bounded mean-zero continuous $\psi:\mathcal M\times \mathcal M\to \mathbb R$ we have that 
\begin{equation}\label{eq:introtwopergfq}
 |P^{(2), \kappa}_t\psi(x, y)| \leq C e^{-\alpha t}d(x, y)^{-p}\|\psi\|_{\infty},\;\;\; t\geq 0.
\end{equation}
 Inequality \eqref{eq:introtwopergfq} will be shown by means of a quantified form of Harris' ergodic theorem, see Theorem \ref{thm:QuaeHarwqthnms} below. The verification of the assumptions of this results require results on stochastic flows provided by Kunita \cite{KSF90} and Baxendale and Stroock \cite{BS88}, and a simple but important trick from \cite[p.\ 9]{DKK04}. Finally, using \eqref{eq:introtwopergfq} we show mixing for \eqref{eq:introkappaisnmwq} in the form of \eqref{eq:INTROH-sHsineqsfa},
by using the methods developed by Bedrossian, Blumenthal, and Punshon-Smith in \cite{BBP-S,BBP-S1911a}. A simple PDE argument then allows to convert $H^{-s}$-decay of the solution of \eqref{eq:introkappaisnmwq} to the desired Theorem \ref{thm:proofofCadpconj} (see Subsection \ref{sec:Proofofmainthm}). In addition, in Section \ref{sec:Proofofmainthm} and Theorem \ref{thm:H-sconvergensofHsinvar} $L^p$-estimates for $D^{A,\gamma}$ in Theorem \ref{thm:proofofCadpconj} are provided.

\subsection{Further comments on the literature}\label{sec:literature}

We collect here further references to the literature on synchronization/stabilization by noise, mixing of passive scalars, and turbulent mixing of passive scalars. Since the literature is extensive, we restrict to literature particularly relevant to this work, referring to the references in the cited works for a more extensive and complete overview.

Stabilization and synchronization by noise for stochastic PDE has been considered, for example in \cite{FGS17-2,GT20,G13,BBP-S,R19}, with related work on the level of SDE in \cite{FGS17,SV18,V18,N18,ACW83,Ar90} and for iterated function systems in \cite{H18} and references therein. For applications to sampling of invariant measures we refer to \cite{LPP15}.

Lyapunov exponents of stochastic flows and their relation to asymptotic properties of the two-point motion have been addressed in depth in the works \cite{DKK04,BS88,Bax86,Bax92,Bax91,ArnBook98,LyapExProc86,Khas12} and references therein.

In parallel to the completion of this work, mixing properties for stochastic PDE falling into the general class of \eqref{eq:mainstochdiffeq} have been analyzed in  \cite{FGL21,FGL21-2} by entirely different methods.   
In comparison to the results obtained in the present work, the methods of \cite{FGL21,FGL21-2} require not only to increase the amplitude of the noise $A$ in \eqref{eq:mainstochdiffeq} but also to choose a particular scaling of the noise coefficients $\sigma_k$ in \eqref{eq:mainstochdiffeq} to obtain an arbitrarily high exponential rate of mixing. In other words, given an exponential rate of mixing $\lambda$, the choice of the noise coefficients $\sigma_k$ in \cite{FGL21,FGL21-2} will depend on this rate $\lambda$, while the present work shows that $\sigma_k$ can be chosen fixed and only the amplitude $A$ has to be chosen large. In this sense the results of the present work are closer to the way problems are formulated in \cite{CKRZ} for the deterministic case \eqref{eq:introCKRZcp}, identifying sufficient conditions for diffusion coefficients $\sigma_k$ to imply arbitrarily fast mixing. 
 On the other hand, the results of \cite{FGL21} allow to include the Kraichnan model for  general $\alpha>0$  due to \cite[Remark 1.2]{FGL21}, which does not follow from Theorem \ref{thm:proofofCadpconj}. A second difference is that by the arguments of \cite{FGL21} no exponential mixing for large times as in Theorem \ref{thm:H-sconvergensofHsinvar} can be obtained, but only weaker statements on finite time horizons, see \cite[Theorem 1.1]{FGL21}. 

Mixing of passive scalars driven by solutions to the stochastic Navier-Stokes equation has been shown in \cite{BBP-S18}. An essential difference to the present contribution is that the passive scalar equation in \cite{BBP-S18} takes the form of a random PDE, while \eqref{eq:mainstochdiffeq} is a stochastic PDE. As a consequence, the Langragian system corresponding to  \eqref{eq:mainstochdiffeq} is Markovian, while the one in \cite{BBP-S18} only becomes Markovian when coupled with the active stochastic PDE (see e.g.\ \cite[Lemma 7.3]{BBP-S18}). Therefore, in the present work we can rely on Harnack inequalities for Kolmogorov equations in finite dimensions, whereas \cite{BBP-S18} has to rely on more specific arguments to take care of the infinite dimensionality of the system. This makes possible covering less regular coefficients in the present work. A second difference to the line of arguments developed in \cite{BBP-S18} is the construction of a $\kappa$-independent Lyapunov function together with the continuity of the two-point motion in $\kappa$ (see Subsections \ref{subsec:constrofVkappa} and \ref{subsec:ergoftwopmot}), replacing the $\kappa$-dependent Lyapunov function in \cite{BBP-S18}. At the same time, the stochastic PDE nature of \eqref{eq:mainstochdiffeq} also causes additional challenges: While  \cite{BBP-S18} may prove the time and spatial regularity of solutions to the passive scalar equation in a pathwise manner (see e.g.\ \cite[proof of Lemma 6.11]{BBP-S}), the present work has to rely on probabilistic calculus in form of Kunita's theory (see \cite{KSF90} and Lemma \ref{lem:phitkat64wppacontinxandkapps}). Finally, the present paper extends these considerations to the manifold setting.

In comparison with Baxendale and Stroock \cite{BS88}, the key difference is that in the present work we include irregular coefficients in the sense that the coefficients of the Fokker-Planck equation corresponding to \eqref{eq:introkappaisnmwq} are assumed only to have  $C^{2+\beta}$ regularity, see Proposition \ref{prop:Lyapconstexistsandpos}).

The existence and uniqueness of solutions to the Kraichnan model in the irregular regime has been analyzed in the works \cite{LR04,LJR02,EVE00}. For related work on well-posedness by noise we refer to \cite{Fl11,FGP10}.

For physical background on mixing by stochastic transport and the Kraichnan model we refer to \cite{CSS94,MK99,GV00,EVE00,ChDG} and the references therein.

There is a large body of works addressing the mixing of passive scalars by deterministic transport. We refer to \cite{CKRZ,LTD11,ACM19,MMGVP,LLNMD12,YZ17} and the references therein. In \cite{MD18} the effect of diffusion on mixing is analyzed. It is observed that in certain cases, a too strong diffusion decreases the mixing effect of transport. This is in line with the results of this work which apply for diffusion intensity $\kappa$ in \eqref{eq:introkappaisnmwq} small enough.

\section{Preliminaries}\label{sec:prelim}

In this section we introduce notations and conventions used throughout the paper. 

For $a, b\in \mathbb R$ we write $a \lesssim_A b$ if there exists a constant $c$ depending only on some quantity $A$ such that
$a \leq cb$, $\gtrsim_A$ is defined analogously, and we write $a \eqsim_A b$ if both $a \lesssim_A b$ and $a \gtrsim_A b$ hold simultaneously. 

 We set $\mathbb R_+ := [0, \infty)$, $\mathbb T = \{z\in\mathbb C : |z|=1\}$ to be the torus and for any $n\geq 1$ we let $|\cdot |$ be the vector norm in $\mathbb R^n$, $\mathbb S^{d-1}$ to be the unit sphere in $\mathbb R^d$ and $\mathbf 1$ to be a function identically equal to $1$.
 
 Further, we set $\mathcal L(X, Y)$ to be the Banach space of bounded linear operators acting from a Banach space $X$ to a Banach space $Y$ endowed with the operator norm $\|\cdot\|$. In particular, $\|\cdot\|$ denotes the operator norm in $\mathcal L(\mathbb R^m, \mathbb R^n)$.

 We let $\mathbb E_{A}\xi_t$ be the expectation of $\xi_t$, where $\xi$ is a Markov process with a given initial value $A$. $\mathbb Z^d_0$ denotes $\mathbb Z^d \setminus \{0\}$.

Let $d\geq 2$ and let $\mathcal M$ be a $d$-dimensional compact $C^{\infty}$-smooth connected  Riemannian manifold with the Riemannian metric tensor $(g_{ij})$ (see e.g.\ \cite[Chapter 3]{LeeIRM}). We will additionally assume that $\mathcal M$ is a regular submanifold of $\mathbb R^{\ell}$ for some $\ell\geq d$ (see \cite[pp.\ 15--16]{LeeIRM}) so that the tangent bundle $T\mathcal M$ and the metric $(g_{ij})$ are generated by the Euclidean structure of $\mathbb R^{\ell}$. Due to the Nash embedding theorem (see \cite[Theorem 3.1.3]{Hs02} and \cite{LeeIRM,Na56}) this does not pose an additional restriction. Let $\mu$ be the Riemannian volume measure and let $d:\mathcal M\times \mathcal M \to \mathbb R_+$ be the distance function on $\mathcal M$ generated by the manifold's metric (see e.g.\ \cite[Sections V and VI]{Boo86}). For any $1\leq p\leq \infty$ we set $L^p(\mathcal M)$ denotes $L^p(\mathcal M,\mu)$. The tangent bundle is $T\mathcal M = (T_x \mathcal M)_{x\in \mathcal M}$ of $\mathcal M$, i.e.\  the set of all tangent spaces $T_{x}\mathcal M$ over $x\in \mathcal M$. For each $x\in\mathcal M$ and $u, v\in T_x\mathcal M$ we denote the inner product of $u$ and $v$ generated by the Riemannian metric by $\langle u, v\rangle_{T_x\mathcal M}$. For simplicity we will write $\langle u, v\rangle_{x}$. Note that since $\mathcal M$ is a regular submanifold of $\mathbb R^{\ell}$, $\langle u, v\rangle_{x} = \langle u, v\rangle$, where $\langle \cdot, \cdot\rangle$ is the standard Euclidean inner product in $\mathbb R^{\ell}$. In particular, we set $|v|:= \sqrt{\langle v,v\rangle_x}=\sqrt{\langle v,v\rangle}$.  For any $x\in \mathcal M$ and $u,v\in T_x\mathcal M$ we denote $\langle u, v\rangle_{T_x\mathcal M}$ by $\langle u, v\rangle_{T\mathcal M}$ if this causes no confusion (as it was done e.g.\ in \eqref{eq:mainstochdiffeq}). Further, for any $x, y\in \mathcal M$, $u\in T_x\mathcal M$, and $v\in T_y\mathcal M$ we set $\langle u, v\rangle$ to be the inner product in $\mathbb R^{\ell}$ containing $\mathcal M$ (this notation is consistent with the one above).

For each $x\in \mathcal M$ we set $S_x\mathcal M$ to be the unit sphere of $T_x\mathcal M$. We also define $S\mathcal M := (S_x\mathcal M)_{x\in \mathcal M}$. Note that $T\mathcal M$ can be endowed with a Riemannian metric called the {\em Sasaki metric}. Then $S\mathcal M\subset T\mathcal M$ can be considered as a submanifold of $T\mathcal M$ with the same metric restricted to $S\mathcal M$.

Analogously to \cite[Section 3]{BS88} define $\Phi:T \mathcal M \to \mathcal M\times \mathcal M$ by
\begin{equation}\label{eq:expmapPhi}
\Phi(x, v)=(x, \exp_x(v)),\;\;\; x\in \mathcal M,\;\;\; v\in T_x \mathcal M,
\end{equation}
where $(x,v)\mapsto \exp_x(v)\in \mathcal M$ is the exponential map (see e.g.\ \cite[p.\ 58]{Hs02}) which is uniquely defined for any $x$ and $v$ as $\mathcal M$ is a $C^{\infty}$-compact. 

 \begin{remark}\label{rem:PhiPhi-1onDcdleta}
 As $\mathcal M$ is compact, it has a positive {\em injectivity radius} $\delta_0$, i.e.\ maximal $\delta_0>0$ such that $\Phi$ is a bijective map from $\{(x, v):x\in \mathcal M, |v|<\delta_0\}$ onto $\mathcal D^c_{\delta_0}:=\{(x, y)\in  \mathcal M \times \mathcal M:0<d(x, y)< \delta_0\}$. Thus for all $(x, y)\in \mathcal D^c_{\delta_0}$ there exists unique $w(x, y)\in T_x\mathcal M$ such that $|w(x, y)| = d(x, y)$ and $\Phi(x, w(x, y)) = (x, y)$.
 \end{remark}

For any smooth $f:\mathcal M \to \mathbb R$, a vector filed $\nabla f:\mathcal M\to T\mathcal M$ denotes the gradient corresponding to the {\em L\`evy-Cevita connection}, i.e.\ a  tangent vector  field of the form $\nabla f=g^{ij} \frac{\partial f}{\partial x^i}\frac{\partial }{\partial x^j}$ where $(x^i)_{i=1}^d$ is a local chart, $\frac{\partial}{\partial x^i}$ is the corresponding basis of  the tangent space, and where $(g_{ij})$ is the Riemannian metric with its inverse $(g^{ij})$, e.g.\ in any orthonormal local coordinates one has that $\nabla f= \frac{\partial f}{\partial x^i}\frac{\partial }{\partial x^i}$. In particular, as any smooth $f:\mathcal M \to \mathbb R$ can be smoothly extended to the whole $\mathbb R^{\ell}\supset\mathcal M$ (see \cite[Exercise 2.3]{LeeIRM}), one might think of $\nabla f$ as of a vector field in $\mathbb R^{\ell}$ so that the projection of $\nabla f(x)$ onto $T_x\mathcal M$ depends only on $f|_{\mathcal M}$ and does not depend on the way $f$ was extended for any $x\in \mathcal M$.

 Now let $X, Y, Z, U$ be Banach spaces, and let $Y\otimes U$ be a normed space of all bilinear forms on $Y^{*}\times U^{*}$ with a finite norm $\|\cdot\|_{Y\otimes U}$ defined by $\|x\|_{Y\otimes U} := \sup\{x(y^*, u^*):y^{*}\in Y^{*}, u^*\in U^*, \|y^*\|=\|u^*\|=1\}$. Then for two linear operators $A\in \mathcal L(X, Y)$ and $B\in \mathcal L(Z, U)$ we let $A\otimes B$ be the bilinear operator $A\otimes B:X\times Z \to Y\otimes U$ defined by $A\otimes B(x, y)(y^*, u^*) := \langle Ax, y^* \rangle\langle Bz, u^* \rangle$. 
 
 If $Y$ and $U$ are finite dimensional, then $Y\otimes U$ is a span of $(y_i\otimes u_j)_{ij}$, where $(y_i)_i$ and $(u_j)_j$ are corresponding bases of $Y$ and $U$, and $y_i\otimes u_j(y^*, u^*) = \langle y_i,y^* \rangle \langle u_i,u^* \rangle$ for any $y^*\in Y^*$ and $u^*\in U^*$. As similarly $(y_i^*\otimes u_j^*)_{ij}$ is a basis of $Y^*\otimes U^*$, where  $(y_i^*)_i$ and $(u_j^*)_j$ are corresponding dual bases of $Y^*$ and $U^*$, i.e.\ $\langle y_i,y_j^* \rangle = \delta_{ij}$ and $\langle u_i,u_j^* \rangle = \delta_{ij}$, one has that $(Y\otimes U)^*  = Y^*\otimes U^*$. In particular, for any $C^2$ function $f:\mathcal M \to \mathbb R$, for any $x\in \mathcal M$, and for any fixed local coordinates $(x^i)_{i=1}^d$ one can define 
 $$
 D^2f(x)(v, w):= v^iw^j \frac{\partial^2 f}{\partial x^i\partial x^j},\;\;\;\; v = v^i \frac{\partial }{\partial x^i}, w = w^i \frac{\partial }{\partial x^i}\in T_x\mathcal M
 $$ 
 (note that $D^2f(x)$ depends on local coordinates). Therefore $D^2 f(x)$ is a bilinear form on $T_x\mathcal M\times T_x\mathcal M$, and hence $D^2 f(x) \in (T_x\mathcal M)^* \otimes (T_x\mathcal M)^* = (T_x\mathcal M \otimes T_x\mathcal M)^*$. Moreover, $D^2 f(x)v\otimes w = D^2f(x)(v, w)$ for any $v, w\in T_x\mathcal M$, so in the sequel we will frequently omit the notation $D^2 f(x)v\otimes w$ preferring $D^2f(x)(v, w)$ or $\langle D^2f(x),(v, w)\rangle$ instead.
 
 Let $g:T\mathcal M \to \mathbb R$ be $C^2$ and fix $x\in \mathcal M$ and $w\in T_x\mathcal M$. For any $u,v \in T_x\mathcal M$ we define
 \[
D^2_{x,x}g(x, w)(u, v)= \langle D^2_{x,x}g(x, w), (u, v) \rangle:=  \Bigl\langle D^2g(x, w), \bigl((u,0),( v,0)\bigr) \Bigr\rangle,
 \]
  \[
 D^2_{x,w}g(x, w) (u, v)=\langle D^2_{x,w}g(x, w), (u, v) \rangle:=  \Bigl\langle D^2g(x, w), \bigl((u,0),( 0,v)\bigr) \Bigr\rangle,
 \]
  \[
 D^2_{w,w}g(x, w) (u, v)=\langle D^2_{w,w}g(x, w), (u, v) \rangle:=  \Bigl\langle D^2g(x, w), \bigl((0,u),(0, v)\bigr) \Bigr\rangle.
 \]
 The same notation is exploited if $g$ is a $C^2$ function on $S\mathcal M$.
 
 \smallskip

For any $C^{\infty}$ vector field $\sigma:\mathcal M \to T\mathcal M$ and for any $C^{\infty}$ function $f:\mathcal M \to \mathbb R$ we define $\sigma f:= \langle Df, \sigma\rangle = \langle \nabla f, \sigma\rangle_{T\mathcal M}$ and $\sigma^2 f$ is defined recursively, i.e.\
\begin{equation}\label{eq:prelimsigma2xfcq}
  \sigma^2 f(x) = \langle D^2 f(x), (\sigma(x), \sigma(x))\rangle + \bigl\langle D f(x),\langle D \sigma(x), \sigma(x)\rangle  \bigr\rangle,\;\; x\in \mathcal M.
\end{equation}

We refer the reader to \cite{El82,Hs02,IW89,KSF90,TrThFSII} for further acquaintance with SDEs on manifolds.

\begin{remark}\label{rem:whyinRell}
 We need $\mathcal M$ to be a regular submanifold of $\mathbb R^{\ell}$ for the following reasons. First, in this case the inner product of any tangent vectors $u\in T_x\mathcal M$ and $v\in T_y\mathcal M$ for any $x, y\in\mathcal M$ can be defined by means of the inner product of the ambient Euclidean space. This is used in the concept of local characteristics \eqref{eq:loccharforphimufd} used intensively in the sequel. Second, in order to apply the theory of stochastic flows from \cite{KSF90} we need the SDE \eqref{eq:equfortildephikappa} below to be defined in $\mathbb R^{\ell}$, which can be done thanks to \cite[Exercise 2.3]{LeeIRM}. Note that $d(x, y)\eqsim_{\mathcal M} |x-y|$ for any $x, y\in \mathcal M$ as $\mathcal M$ is a regular $C^{\infty}$ submanifold of $\mathbb R^{\ell}$ so $(x, y)\mapsto \frac{d(x, y)}{|x-y|}$ is continuous in $x, y\in \mathcal M$, $x\neq y$, and bounded from above. Finally, later we will frequently need to transform Stratonovich integrals to their It\^o form, which can be done easily in $\mathbb R^{\ell}$.
\end{remark}

For any $s\geq 0$ we set $H^s := H^s(\mathcal M)$ to be a Sobolev space, i.e.\ a space of all $f\in L^2(\mathcal M)$ of the form $f= \sum_{k\geq 1} c_k \hat e_k$ satisfying
\begin{equation}\label{eq:defofHspcq}
 \|f\|_{H^s} := \Bigl(\sum_{k\geq 1} (1+\lambda_k)^{s}c_k^2\Bigr)^{\frac 12}<\infty
\end{equation}
where $(\hat e_k)_{k\geq 1}$ is an orthonormal basis of $L^2(\mathcal M)$ consisting of all eigenvectors of $-\Delta$ (where $\Delta$ is the Laplace-Beltrami operator) and $(\lambda_k)_{k\geq 1}$ are the corresponding eigenvalues (see e.g.\ \cite{So88,Do01,CKRZ,GSch13,TrThFSII}).
In particular, in the case of $\mathcal M = \mathbb T^d$ the space $H^s(\mathcal M)$ will consist of all $f\in L^2(\mathbb T^d)$ such that for the Fourier transform $\hat f:\mathbb Z^d \to \mathbb R$ we have that 
$$
\|f\|_{H^s} = \Bigl(\sum_{z\in \mathbb Z^d} (1 + |z|^2)^{s}|\hat{f}|^2(z)\Bigr)^{\frac 12}<\infty.
$$

Let $\mathcal D:= \{(x,x), x\in \mathcal M\}\subset \mathcal M \times \mathcal M$ denote the diagonal and let $\mathcal D^c:=\mathcal M \times \mathcal M \setminus \mathcal D$ denote the off-diagonal. Then for any fixed $s>0$ analogously to \cite[Section 7.2]{BBP-S} thanks to the Sobolev--Slobodeskij norming of Sobolev spaces (see e.g.\ Sections 1.3.2, 1.3.4, and Chapter 7 in \cite{TrThFSII} and \cite{BPS13,GSch13}) the following holds true for any $C^{\infty}$ function $g:\mathcal M \to \mathbb R$
\begin{equation}\label{eq:Sobolevequivalentnormisd}
\begin{split}
\|g\|_{H^s}& \eqsim_{\mathcal M, s}\|g\|_{L^2} + \sum_{\alpha \in I}\Bigl( \iint_{\mathcal D^c_{\alpha}} \frac{|D^{[s]}g(x)- D^{[s]}g(y)|^2}{|\kappa_{\alpha}(x) - \kappa_{\alpha}(y) |^{2\{s\}+d}} \ud \mu\otimes\mu(x, y) \Bigr)^{\frac{1}{2}}\\
& \eqsim_{\mathcal M, (U_{\alpha})_{\alpha \in I}}\|g\|_{L^2} +\Bigl( \iint_{\mathcal D^c} \frac{|D^{[s]}g(x)- D^{[s]}g(y)|^2}{d(x, y) ^{2\{s\}+d}} \ud \mu\otimes\mu(x, y) \Bigr)^{\frac{1}{2}}
\end{split}
\end{equation}
where $(U_{\alpha})_{\alpha\in I}$ is a finite set of open subsets of $\mathcal M$ covering $\mathcal M$, $(\kappa_{\alpha})_{\alpha\in I}$ are the corresponding coordinate charts (i.e.\ for each $\alpha \in I$ there exist an open set $V_{\alpha}\subset \mathbb R^d$ and a homeomorphic map $\kappa_{\alpha}:U_{\alpha}\to V_{\alpha}$; $\kappa_{\alpha}$ can be assumed exponent, see e.g.\ \cite{GSch13}), $\mathcal D^c_\alpha =  U_{\alpha}\times U_{\alpha} \cap \mathcal D^c$, $[s]$ is the integer part of $s$, and $\{s\}:= s-[s]$.

Let $(W^k)_{k\geq 1}$ be a family of independent Brownian motions on $\mathbb R_+$. Throughout the paper we assume the probability space $(\Omega, \mathcal F, \mathbb P)$ and filtration $\mathbb F = (\mathcal F_t)_{t\geq 0}$ to be generated by $(W^k)_{k\geq 1}$. In particular, $\Omega$ will be the set of all trajectories of $(W^k)_{k\geq 1}$, i.e.\ $\Omega = C(\mathbb R_+;\mathbb R^\infty)$. Therefore for any stopping time $\tau:\Omega \to \mathbb R_+$ we are able to define a {\em shift operator} $\theta_{\tau}:\Omega \to \Omega$ to be $\theta_{\tau}(\omega):= \omega(\tau(\omega)+\cdot) - \omega(\cdot)$.

In the sequel we will also need a sequence $(\widetilde W^m)_{m= 1}^n$ of  independent Brownian motions independent of $(W^k)_{k\geq 1}$ which will be used to generate the operator $T$ defined by \eqref{eq:INTROTda} (see Subsection \ref{subsec:ergoftwopmot} below). Let $(\widetilde {\Omega},\widetilde {\mathcal F}, \widetilde {\mathbb P})$ and $\widetilde {\mathbb F} = (\widetilde {\mathcal F}_t)_{t\geq 0}$ be defined analogously $(\Omega, \mathcal F, \mathbb P)$ and $\mathcal F = (\mathcal F_t)_{t\geq 0}$ (in particular, $\widetilde {\Omega}=C(\mathbb R_+;\mathbb R^n)$). Throughout the paper we set $\overline \Omega := \Omega \times \widetilde {\Omega}$, $\overline {\mathcal F} := \mathcal F\otimes\widetilde {\mathcal F}$, $\overline {\mathbb P} := \mathbb P\otimes\widetilde {\mathbb P}$ with the corresponding product filtration $\overline {\mathbb F}$ defined similarly. For any stopping time $\tau:\overline \Omega \to \mathbb R_+$ we define a shift operator $\theta_{\tau}:\overline \Omega \to \overline \Omega$ to be 
\begin{equation}\label{eq:thetataudef}
\begin{split}
  &\theta_{\tau}(\omega\times \widetilde \omega)\\
  &\;\;\;\;:=\bigl( \omega(\tau(\omega\times \widetilde \omega)+\cdot) - \omega(\tau(\omega\times \widetilde \omega))\bigr)\times \bigl(\widetilde  \omega(\tau(\omega\times \widetilde \omega)+\cdot) - \widetilde \omega(\tau(\omega\times \widetilde \omega))\bigr),
\end{split}
\end{equation}
for any $ \omega\in \Omega$ and $\widetilde \omega\in \widetilde\Omega$.

Later in the paper we will use the following quantitative form of Harris' ergodic theorem presented in \cite[Theorem 1.2]{HM11} (see also \cite{MT09,Khas12}).

\begin{theorem}\label{thm:QuaeHarwqthnms}
 Let $(S, \Sigma, \rho)$ be a measure space, let $(\mathcal P^n)_{n\geq 0}$ be a Markov semigroup on $S$, and let $V:S \to [1, \infty)$ be such that
 \begin{enumerate}[\rm ($i$)]
  \item there exist $\gamma\in(0,1)$ and $C\geq 0$ such that $\mathcal P V \leq \gamma V + C$,
  \item there exist $R>2C/(1-\gamma)$, $\eta = \eta(R)>0$, and a probability measure $\nu$ on $S$ such that 
  $$
  \inf_{s\in \{V\leq R\}}\mathcal P \mathbf 1_{A} (s)>\eta \nu(A),
  $$
 for any $A\in \Sigma$. 
 \end{enumerate}
Then there exist constants $C_0>0$ and $\gamma_0\in(0, 1)$ depending only on $\gamma$, $C$, $R$, and $\eta$ and there exists a unique invariant measure $\mu$ of $(\mathcal P^n)_{n\geq 0}$ such that
\begin{equation}\label{eq:maingofthm:QuaeHarwqthnms}
  \Bigl | \mathcal P^n\psi(s) - \int_{S}\psi \ud \mu \Bigr| \leq C_0 \gamma_0^n V(s)\|\psi\|_{\infty},\;\;\; s\in S,\;\; n\geq 0,
\end{equation}
for any $\psi:S\to \mathbb R$ with $\|\psi\|_{\infty}<\infty$.
\end{theorem}

\section{Stochastic Mixing}\label{sec:mixwithstoch}

We first state the assumptions on $(\sigma_k)_{k\geq 1}$ which we will use in order to prove Theorem \ref{thm:proofofCadpconj}.

\begin{enumerate}[\bf (A)]
\item $(\sigma_k(\cdot), \sigma_k(\cdot))_{k\geq 1}$ is an elliptic system, i.e.\ for any closed off-diagonal subset $U\subset \mathcal D^c$ there exists $C_{U}>0$ such that for any $(x_1, x_2)\in U$ 

\[
 \sum_{k\geq 1} \bigl|\langle \sigma_k(x_1), u\rangle_{x_1} + \langle \sigma_k(x_2), v\rangle_{x_2}\bigr|^2\geq C_U (|u|^2 + |v|^2),\;\;\;u\in T_{x_1}\mathcal M,\;\; v\in T_{x_2}\mathcal M.
\]

\item 

The normalized tangent flow, see \eqref{eq:norm_tangent_flow} below, is elliptic, i.e.\ for any compact $V\subset S \mathcal M$ there exists $C_{V}>0$ such that for any $(x, v)\in V$

\[
 \sum_{k\geq 1} \bigl|\langle \sigma_k(x), u\rangle_{x} + \langle \tilde \sigma_k(x,v), w\rangle_{x,v}\bigr|^2\geq C_V (|u|^2 + |w|^2),\;\;\;u\in T_{x}\mathcal M,\;\; w\in T_{v}(S_x \mathcal M).
\]
where $T_{v}(S_x \mathcal M)$ is the tangent space to $v\in S_x \mathcal M$, and where for each $v\in S_x \mathcal M$ one has that 
\begin{equation}\label{eq:defofotildesigmakdacq}
\tilde \sigma_k(x,v) := \langle D \sigma_k(x), v \rangle - v \Bigl \langle v, \langle D\sigma_k(x), v \rangle \Bigr \rangle_x,\;\;\; k\geq 1.
\end{equation}

\item All $(\sigma_k)_{k\geq 1}$ are $C^{\infty}(\mathcal M; T\mathcal M)$, with
\begin{equation}\label{eq:nescondonsigmaintermsofsums}
\sum_{k\geq 1}\|\sigma_k\|_{\infty}^2 +  \|D \sigma_k\|_{\infty}^2 + \|D^2 \sigma_k\|_{\infty} \|\sigma_k\|_{\infty}<\infty,
\end{equation}
and there exists $\beta\in (0,1]$ such that the maps
\[
x, y\mapsto \sum_{k\geq 1}  D \sigma_k(x)\otimes D \sigma_k(y)\;\; \text{and}\;\;  x, y\mapsto \sum_{k\geq 1}  D^2 \sigma_k(x)\otimes \sigma_k(y),\;\;\; x, y\in \mathcal M,
\]
 are $C^{\beta}$ (i.e.\ $\beta$-H\"older continuous, see \cite[Section 3]{KSF90}), while for any $(x, y)\in \mathcal D^c_{\delta_0}$ and $w=w(x, y)$ (see Remark \ref{rem:PhiPhi-1onDcdleta})
 \begin{align*}
 &\Biggl|\sum_{k\geq 1}\left \langle D^2 w(x, y), \left (\binom{\sigma_k(x)}{\sigma_k(y)},\binom{\sigma_k(x)}{\sigma_k(y)}\right)\right \rangle + \left \langle D w(x, y),  \binom{\langle D\sigma_k(x), \sigma_k(x)\rangle}{\langle D\sigma_k(y), \sigma_k(y)\rangle}  \right \rangle\\
&\quad\quad\quad\quad-\langle D^2\sigma_k(x), (\sigma_k(x), w)\rangle  - \bigl \langle D\sigma_k(x),\langle D\sigma_k(x), w\rangle\bigr\rangle \Biggr|= O(|w|^{1+\beta})
 \end{align*}

and both
$$
 \sum_{k\geq 1} \left\| \left \langle Dw(x, y), \binom{\sigma_k(x)}{ \sigma_k(y)}\right \rangle\otimes\left \langle Dw(x, y), \binom{\sigma_k(x)}{ \sigma_k(y)}\right \rangle- \langle D \sigma_k(x), w \rangle\otimes\langle D \sigma_k(x), w \rangle\right\|,
$$ 
and
\begin{equation}\label{eq:qvofsigma-goan-Dfopsaxw)wfq}
  \sum_{k\geq 1}\left |\left \langle Dw(x, y), \binom{\sigma_k(x)}{ \sigma_k(y)}\right \rangle-\langle D \sigma_k(x),w \rangle\right|^2,
\end{equation}
are bounded by $C|w|^{2+\beta}$
for some uniform constant $C>0$.

\end{enumerate}

The coefficients $\tilde \sigma_k(v)$ appearing in \textbf{(B)} above correspond to the so-called normalized tangent flow 
\begin{equation}\label{eq:norm_tangent_flow}
  \ud \tilde {\phi}_t(v) = \sum_{k\geq 1}\tilde \sigma_k(\tilde {\phi}_t(v)) \circ \ud W^k_t,\;\;\; \tilde {\phi}_0 = v,
\end{equation}
which describes the evolution of the normalized tangent vector in the Lagrangian picture of the solution of the following equation
\begin{equation}\label{eq:flowforkappa=0dsa}
 \ud \phi_t = \sum_{k\geq 1} \sigma_k (\phi_t) \circ \ud W^k_t,\;\;\; \phi_0=x,
\end{equation}
i.e.\ 
$$
\tilde{\phi}_t(v) = \lim_{\eps\to 0}\frac{w\bigl(\phi_t(x), \phi_t(\exp_{x}(\eps v))\bigr)}{\left|w\bigl(\phi_t(x), \phi_t(\exp_{x}(\eps v))\bigr)\right|},
$$ 
where $(x, \eps v)\mapsto \exp_{x}(\eps v)\in \mathcal M$ is the exponential map (see Section \ref{sec:prelim}) and where $w\bigl(\phi_t(x), \phi_t(\exp_{x}(\eps v))\bigr)\in T_{\phi_t(x)}\mathcal M$ is well-defined as thanks to Lemma \ref{lem:phitkat64wppacontinxandkapps} below $\phi$ is a stochastic flow of diffeomorphisms, so $d(\phi_t(x), \phi_t(\exp_{x}(\eps v)) )\to 0$ as $\eps \to 0$ (see Remark \ref{rem:PhiPhi-1onDcdleta}).

\begin{remark}
Thanks to the property \textbf{(C)} the property \textbf{(A)} is equivalent the fact that for any
 $x_1, x_2\in \mathcal M$ with $x_1\neq x_2$ there exists a constant $C(x_1, x_2)>0$ so that
\[
 \sum_{k\geq 1} \bigl|\langle \sigma_k(x_1), u\rangle_{x_1} + \langle \sigma_k(x_2), v\rangle_{x_2}\bigr|^2\geq C(x_1, x_2) (|u|^2 + |v|^2),\;\;\;u\in T_{x_1}\mathcal M,\;\; v\in T_{x_2}\mathcal M,
\]
as in this case the function 
\[
 C(x_1, x_2):= \inf_{\substack{u\in T_{x_1}\mathcal M, v\in T_{x_2}\mathcal M\\
 |u|^2 + |v|^2=1}} \sum_{k\geq 1} \bigl|\langle \sigma_k(x_1), u\rangle_{x_1} + \langle \sigma_k(x_2), v\rangle_{x_2}\bigr|^2
\]
is continuous and hence on any closed set $U\subset \mathcal D^c$ it can be bounded uniformly from below by a constant $C_U>0$. The same can be shown for the condition {\textnormal {\bf (B)}}.
\end{remark}

\begin{remark}
 Note that the last rather technical condition {\textnormal {\bf (C)}} is needed as we are going to apply Theorem \ref{thm:proofofCadpconj} to the Kraichnan model where the sum in \eqref{eq:mainstochdiffeq} is infinite in converse to classical works \cite{DKK04,BS88,Bax86,Bax92} where the corresponding sums are finite.  Further, the reader should not be confused about such a seemingly complicated condition as for any $x\in \mathcal M$ and $u, v\in T_x\mathcal M$ we have that $\langle Dw(x,x), (u,v)\rangle = v-u$, so by embedding $\mathcal M$ into $\mathbb R^{\ell}$ and by approximating $\langle Dw(x,y), (\sigma_k(x),\sigma_k(y))\rangle \approx \sigma_k(y)-\sigma_k(x)$ this condition is based on the fact that
 \begin{align*}
\langle D (\sigma_k(y)-\sigma_k(x)),(\sigma_k(x),\sigma_k(y))\rangle& = \langle D\sigma_k(y),\sigma_k(y)\rangle - \langle D\sigma_k(x),\sigma_k(x)\rangle\\
& \approx \langle D^2\sigma_k(x), (\sigma_k(x), w)\rangle  + \bigl \langle D\sigma_k(x),\langle D\sigma_k(x), w\rangle\bigr\rangle
 \end{align*}
 and
 \[
 \sigma_k(y)-\sigma_k(x) \approx\langle D \sigma_k(x),w \rangle,
 \]
 for $d(x, y)$ sufficiently small.
\end{remark}

Throughout the paper we will assume that $T:C^2(\mathcal M)\to C(\mathcal M)$ is strictly elliptic, i.e.\ for any fixed $x\in \mathcal M$ and for any local coordinates $(x^1, \ldots, x^d)$ $T$ can be written as
\[
 Tf(x)=\sum_{i, j=1}^d a^{ij}(x) \frac{\partial^2 f(x)}{\partial x^i\partial x^j} + \sum_{i=1}^d b^i(x) \frac{\partial f(x)}{\partial x^i},
\]
where $(a^{ij}(x))_{i, j=1}^d$ is a self-adjoint positive definite matrix.
Equivalently, thanks to \eqref{eq:prelimsigma2xfcq} $T$ is strictly elliptic if and only if
\begin{equation}\label{eq:chimbvopicqw}
  \sum_{m=1}^n |\langle \chi_m(x), v\rangle_x|^2>0\;\; \text{for any $x\in \mathcal M$ and $v\in T_x\mathcal M$}.
\end{equation}
In particular, the following lemma holds true. Recall that $\sigma f:= \langle Df, \sigma\rangle = \langle \nabla f, \sigma\rangle_{T\mathcal M}$ for any $C^{\infty}$ vector field $\sigma:\mathcal M \to T\mathcal M$ and for any $C^{\infty}$ function $f:\mathcal M \to \mathbb R$.

\begin{lemma}\label{lem:Tstriellverycoopchimdfowe}
 Let $T$ be strictly elliptic. Then for any $u\in C^1(\mathcal M)$
 \begin{equation}\label{eq:normsvuschimeqnoic}
 \begin{split}
  \sum_{m=1}^n \|\chi_m u\|_{L^2(\mathcal M)}^2 &= \sum_{m=1}^n \|\langle \chi_m ,\nabla u\rangle_x\|_{L^2(\mathcal M)}^2  \\
  &\quad\quad\quad\quad\quad\eqsim \|\nabla u\|_{L^2(\mathcal M; T\mathcal M)}^2 =  \|\nabla u\|_{L^2(\mathcal M; \mathbb R^{\ell})}^2.
  \end{split}
 \end{equation}
\end{lemma}

\begin{proof}
First recall that $\nabla u(x)\in T_x\mathcal M$, so $\|\nabla u(x)\|_{T_x\mathcal M} = \|\nabla u(x)\|_{\mathbb R^{\ell}} = |\nabla u(x)|$, so the last part of \eqref{eq:normsvuschimeqnoic} follows.
 
 Note that by \eqref{eq:chimbvopicqw} for any $x\in \mathcal M$ there exist 
 positive constants $c(x)$ and $C(x)$ such that 
 \[
  c(x)\sum_{m=1}^n |\langle \chi_m(x) ,v\rangle_x|^2  \leq  |v|^2 \leq  C(x)\sum_{m=1}^n |\langle \chi_m(x) ,v\rangle_x|^2,\;\;\; v\in T_x\mathcal M. 
 \]
 Choose $c(x)$ and $C(x)$ sharp.
As $(\chi_m)_{m=1}^n$ are $C^{\infty}$, $c(x)$ and $C(x)$ chosen sharp are continuous in $x\in \mathcal M$. Therefore there exist $c=\min_{x\in \mathcal M}c(x)>0$ and $C=\max_{x\in \mathcal M}C(x)<\infty$, so that
\[
  c\sum_{m=1}^n |\langle \chi_m(x) ,v\rangle_x|^2  \leq  |v|^2 \leq  C\sum_{m=1}^n |\langle \chi_m(x) ,v\rangle_x|^2,\;\;\; v\in T_x\mathcal M, 
 \]
 for any $x\in \mathcal M$, and therefore the lemma follows.
\end{proof}

If we assume {\textnormal {\bf (C)}} and the abovementioned assumption on $T$, then the stochastic equation \eqref{eq:mainstochdiffeq} has a unique solution $(u_t)_{t\geq 0}$ with values in $L^2(\mathcal M)$ so that $u_t\in H^1(\mathcal M)$ for any $t\geq 0$ (see Section \ref{sec:prelim} for the definition of $L^2(\mathcal M)$ and $H^s(\mathcal M)$). Indeed, if one fixes a Gelfand triple $H^1(\mathcal M) \subset L^2(\mathcal M)\subset  H^{-1}(\mathcal M)$, then \eqref{eq:mainstochdiffeq} satisfies the conditions (H1)--(H4) on \cite[p.\ 70]{LR15} as in this case we have that $du_t = Qu_t \ud t + Ru_t \ud W^H_t$, where $W^H$ is a cylindrical Brownian motion over a separable infinite-dimensional Hilbert space $H$ and $Q\in \mathcal L(H^1(\mathcal M); H^{-1}(\mathcal M))$ and $R\in \mathcal L(H^1(\mathcal M);\mathcal L(H, L^2(\mathcal M)))$ are defined via
\[
 Qf := T f + \frac{1}{2}A^2 \sum_{k\geq 1} \Bigl\langle \sigma_k , \nabla\langle\sigma_k , \nabla f\rangle_{T\mathcal M}\Bigr\rangle_{T\mathcal M},
\]
\[
 Rf := -A \sum_{k\geq 1}\langle \sigma_k, \nabla f\rangle_{T\mathcal M} e_k,
\]
for any $f\in H^1(\mathcal M)$, where $(e_k)_{k\geq 1}$ is an orthonormal basis of $H$. (H1) then follows by the linearity, while (H2)--(H4) follows from a direct computation (see also \cite[Example 4.1.7]{LR15}), where one exploits the fact that the sequence $(\|\sigma_k\|_{\infty} \|D\sigma_k\|_{\infty})_{k\geq 1}$ is summable by \eqref{eq:nescondonsigmaintermsofsums} and that by the divergence theorem for manifolds (see e.g.\ \cite[Theorem 16.32]{Le13} and \cite[Theorem 1]{Step16})
\begin{equation}\label{eq:Tff=-chimf2lcvnkw}
 \begin{split}
  & \int_{\mathcal M}Tf(x)f(x) \ud \mu(x) +\sum_{m=1}^n \int_{\mathcal M} (\chi_m f(x))^2 \ud \mu(x) \\
 &\quad\quad\quad\quad\quad= \int_{\mathcal M}\sum_{m=1}^n\langle\chi_m(x) , \nabla(f(x)\chi_m f(x))\rangle_{T\mathcal M} \ud \mu(x)\\
 &\quad\quad\quad\quad\quad\quad\quad\quad\quad\quad= \int_{\mathcal M}\sum_{m=1}^n\text{div}\, (\chi_m(x) f(x)\chi_m f(x)) \ud \mu(x)=0,
 \end{split}
\end{equation}
as $(\chi_m)_{m=1}^n$ are divergence free, 
so $\int_{\mathcal M}Tf(x)f(x) \ud \mu(x) \leq 0$ and hence by \cite[theorem 4.2.4]{LR15} existence and uniqueness of the solution of \eqref{eq:mainstochdiffeq} follows.

In order to prove Theorem \ref{thm:proofofCadpconj} we will first consider the solution of the following stochastic equation on $\mathbb R_+ \times \mathcal M$, obtained from  \eqref{eq:mainstochdiffeq} by rescaling (see Subsection~\ref{sec:Proofofmainthm}),
\begin{equation}\label{eq:mainstochdiffeqwithkappa}
\begin{cases}
 \ud u_t + \sum_{k\geq 1} \langle\sigma_k ,\nabla u_t \rangle_{T\mathcal M}\circ \ud W^k_t  &= \kappa T u_t \ud t,\\
 u_0 &= u\in H^1(\mathcal M),
\end{cases}
\end{equation}
where $(\sigma_k)$ satisfy {\textnormal{\bf (A)--(C)}} and where $\kappa\geq 0$. Note that \eqref{eq:mainstochdiffeqwithkappa} has a unique solution with values in $L^2(\mathcal M)$ in the case of $\kappa>0$ for the same reason as \eqref{eq:mainstochdiffeq} (see the argument above).

 \begin{remark}\label{rem:kappa0icn2938}
 We comment on the well-posedness of solutions in the case $\kappa=0$ separately, that is, the existence and uniqueness of solutions to 
\begin{equation}\label{eq:kappa0equaoicwp}
\begin{cases}
 \ud u_t + \sum_{k\geq 1} \langle\sigma_k ,\nabla u_t \rangle_{T\mathcal M}\circ \ud W^k_t  &=0,\\
 u_0 &= u\in H^1(\mathcal M).
\end{cases}
\end{equation}
In this case, the standard methods for stochastic PDE do not apply, due to the lack of viscosity.

We say that an adapted process $(u_t)_{t\geq 0}$ is a solution to \eqref{eq:kappa0equaoicwp} if $$(u_t)_{t\geq 0}\in L^2(\Omega; L^2_{loc}(\mathbb R_+;H^1(\mathcal M)))\cap L^2(\Omega; C_{loc}(\mathbb R_+;L^2(\mathcal M)))$$and  a.s., for all $t\ge0,$
 \begin{equation}\label{eq:formforuton230forkappa=0}
 \begin{split}
 u_t =& -  \sum_{k\geq 1}\int_0^t \langle\sigma_k ,\nabla u_s \rangle_{T\mathcal M} \ud W^k_s - \sum_{k\geq 1} \int_0^t \bigl \langle \sigma_k, \nabla\langle\sigma_k ,\nabla u_s \rangle_{T\mathcal M} \bigr\rangle_{T\mathcal M} \ud s,
 \end{split}
 \end{equation}
as an equation in $L^2(\mathcal M)$.

In order to show that a solution to \eqref{eq:kappa0equaoicwp}  exists
we consider the flow of $C^1$-diffeomorphisms $(\phi_t)_{t\geq 0}$ to $\ud\phi_t(x) = \sum_{k\geq 1} \sigma_k(\phi_t(x))\circ \ud W^k_t$, $\phi_0(x) = x\in \mathcal M$ (see Lemma \ref{lem:phitkat64wppacontinxandkapps} below), and set $u_t(x):= u(\phi_t^{-1}(x))$. Then, $u_t$ satisfies $(u_t)_{t\in[0, T]}\in L^p(\Omega\times [0, T];L^2(\mathcal M))$ for any $p\geq 1$ and $T>0$. Indeed, $\phi_t$ is measure preserving by Remark \ref{rem:phikappameasupres} below which implies $\|u_t\|_{L^2(\mathcal M)} = \|u\|_{L^2(\mathcal M)}$. Moreover, as  by Lemma \ref{lem:phitkat64wppacontinxandkapps} $(\phi_t)_{t\geq 0}$ is continuous in $t$  in the sense that for $\mathbb P$-a.e.\ $\omega\in \Omega$ and for any $T>0$ one has that 
$$
\sup_{x\in \mathcal M}d((\phi_t)^{-1}(x), (\phi_s)^{-1}(x))\lesssim_{\omega, T}|t-s|^{1/4},\;\;\;\;t,s\in [0, T],
$$ 
$(u_t)_{t\geq 0}$ is $\mathbb P$-a.s.\ $L^2(\mathcal M)$-continuous for any $u\in C^{\infty}(\mathcal M)$. Continuity for a general $u\in L^2(\mathcal M)$ then follows by contractivity and a limiting argument.

Furthermore,  by Lemma \ref{lem:phitkat64wppacontinxandkapps} we have that $(u_t)_{t\in[0, T]}\in L^p(\Omega\times [0, T];H^1(\mathcal M))$ for any $p\geq 1$ and $T>0$, since
\begin{equation}\label{eq:nablaudsapojt}
\langle \nabla u_t(x), v\rangle_x = \bigl\langle \nabla u(\phi_t^{-1}(x)), \langle D\phi_t^{-1}(x), v \rangle \bigr\rangle_{\phi_t^{-1}(x)},\;\;\; x\in \mathcal M,\;\;\; v\in T_x\mathcal M.
\end{equation}

Next, let us show that $(u_t)_{t\geq 0}$ satisfies \eqref{eq:formforuton230forkappa=0} in $H^{-1}(\mathcal M)$. To this end note that by the condition {\textnormal {\bf (C)}}, \eqref{eq:nablaudsapojt}, and Lemma \ref{lem:phitkat64wppacontinxandkapps} for any $T>0$ 
$$
\mathbb E\sum_{k\geq 1}\int_0^T \bigl\|\langle\sigma_k ,\nabla u_t \rangle_{T\mathcal M}  \bigr\|^2_{L^2(\mathcal M)}  + \left\|\bigl \langle \sigma_k, \nabla\langle\sigma_k ,\nabla u_s \rangle_{T\mathcal M} \bigr\rangle_{T\mathcal M} \right\|_{H^{-1}(\mathcal M)}\dd t<\infty,
$$ 
so the right-hand side of \eqref{eq:formforuton230forkappa=0} is well-defined in $H^{-1}(\mathcal M)$. Further, by Remark \ref{rem:phikappameasupres}, It\^o's formula, and the fact that $u_t\in H^1(\mathcal M)\subset L^2(\mathcal M)$ a.s.\ by \eqref{eq:nablaudsapojt} we have that  for any $\psi\in C^{\infty}(\mathcal M)$ a.s.\ for any $t\geq 0$
\begin{align*}
{}_{H^{-1}(\mathcal M)}\langle u_t, \psi\rangle_{H^1(\mathcal M)} &= \int_{\mathcal M} u(\phi_t^{-1}(x))\psi(x)\ud \mu(x) = \int_{\mathcal M} u(x)\psi(\phi_t(x))\ud \mu(x)\\
&= \int_{\mathcal M}u(x) \sum_{k\geq 1}\int_0^t \bigl\langle \sigma_k(\phi_s(x)), \nabla \psi(\phi_s(x)) \bigr\rangle_{\phi_s(x)}\circ\dd W^k_s\ud\mu(x)\\
&\stackrel{(*)}= \sum_{k\geq 1}\int_0^t\int_{\mathcal M}u(x) \bigl\langle \sigma_k(\phi_s(x)), \nabla \psi(\phi_s(x)) \bigr\rangle_{\phi_s(x)}\ud\mu(x)\circ\dd W^k_s\\
&= \sum_{k\geq 1}\int_0^t\int_{\mathcal M}u_s(x) \bigl\langle \sigma_k(x), \nabla \psi(x) \bigr\rangle_{x}\ud\mu(x)\circ\dd W^k_s\\
&= -\int_0^t\sum_{k\geq 1}\int_{\mathcal M} \bigl\langle \sigma_k(x), \nabla u_s(x) \bigr\rangle_{x}\psi(x)\ud\mu(x)\circ\dd W^k_s\\
&=-\int_0^t\sum_{k\geq 1}\Bigl \langle \bigl\langle \sigma_k, \nabla u_s \bigr\rangle_{T\mathcal M},\psi\Bigr\rangle_{L^2(\mathcal M)}\circ\dd W^k_s
\end{align*}
where $(*)$ follows from the stochastic Fubini theorem \cite[Theorem 65]{Prot}, the condition {\textnormal {\bf (C)}}, and the fact that $u\in H^1(\mathcal M)\subset L^2(\mathcal M)$. Since $C^{\infty}(\mathcal M)$ is dense in $H^1(\mathcal M) = (H^{-1}(\mathcal M))^*$, \eqref{eq:formforuton230forkappa=0} holds true in $H^{-1}(\mathcal M)$. Since both $(u_t)_{t\geq 0}$ and $\sum_k \langle\sigma_k ,\nabla u_{\cdot} \rangle_{T\mathcal M} \cdot W^k$ have $L^2(\mathcal M)$-valued continuous versions,  
 \eqref{eq:formforuton230forkappa=0} holds in $L^2(\mathcal M)$.

Uniqueness of solutions follows from linearity and an application of the
 It\^o formula \cite[Theorem 4.2.5 and (4.30)]{LR15}.
\end{remark}

As we will see later in Section \ref{sec:Proofofmainthm}, \eqref{eq:mainstochdiffeqwithkappa} with $\kappa=1/A^2$ corresponds to \eqref{eq:mainstochdiffeq} by rescaling. For us it will be easier to work with \eqref{eq:mainstochdiffeqwithkappa} as in this setting we are able to let $\kappa\to 0$ (corresponding $A\to \infty$ in \eqref{eq:mainstochdiffeq}) and consider the behaviour of the equation for $\kappa=0$ as the solution of \eqref{eq:mainstochdiffeqwithkappa} is in a sense continuous in $\kappa$. 

We first establish mixing of $u_t$ in \eqref{eq:mainstochdiffeqwithkappa} in the sense of \cite{BBP-S,LTD11,MMGVP}, i.e.\ we show that $\|u_t\|_{H^{-s}}$ vanishes for any $s>0$ with an exponential rate uniform in $\kappa$, and based on this convergence show \eqref{eq:INTROmainthmwithCiqo}.

\begin{theorem}\label{thm:H-sconvergensofHsinvartext}
 Let $(\sigma_k)_{k\geq 1}$ satisfy the conditions {\textnormal {\bf (A)--(C)}} and let $(u_t)_{t\geq 0}$ be the solution of \eqref{eq:mainstochdiffeqwithkappa} with $u$ being  mean-zero in $H^s(\mathcal M)\cap H^1(\mathcal M)$ for some $s\in(0, 1+\beta/2)$. Then there exist $\kappa_0>0$ and $\gamma_0>0$ independent of $s$ and $u$ such that for any $\kappa \in [0, \kappa_0]$ and $\gamma\in(0, \gamma_0)$ there exists $D_{\kappa,\gamma}:\Omega \to [1, \infty)$ so that 
 \begin{equation}\label{eq:H-sHsineqsfa}
  \|u_t\|_{H^{-s}} \leq D_{\kappa,\gamma}e^{-\gamma s t} \|u\|_{H^s},\;\;\; t> 0.
 \end{equation}
 
 Moreover, $\sup_{\kappa\in[0, \kappa_0]}\mathbb E |D_{\kappa,\gamma}|^p<\infty$ for any $1\leq p<\tfrac{9d \gamma_0}{2\gamma s}$.
\end{theorem}

\begin{remark}\label{rem:Hsforkappa>0c39jix1}
Note that for $\kappa>0$ the SPDE \eqref{eq:mainstochdiffeqwithkappa} is well-posed for a general $u\in L^2(\mathcal M)$ thanks to \cite[Chapter 4.2]{LR15}. Since,  due to \cite[Proposition 2.4.10]{LR15}, the mapping $u\mapsto u_t$ is in $\mathcal L(L^2(\mathcal M), L^2(\overline \Omega;L^2(\mathcal M)))$ for any $t\geq 0$, \eqref{eq:H-sHsineqsfa} remains valid for any initial value $u\in H^s$ with $s\in[0, 1)$.
\end{remark}

The structure of the proof of Theorem \ref{thm:H-sconvergensofHsinvartext} is inspired by the one presented in \cite{BBP-S}, where the authors exploit Theorem \ref{thm:QuaeHarwqthnms} for the two-point Lagrangian flow corresponding to passive scalars transported by solutions to the stochastic Navier-Stokes equation. Note that in contrast to the present work this leads to a random PDE rather than a stochastic PDE. One major difference of the present work to \cite{BS88} and \cite{BBP-S} is that here we are addressing the case of irregular coefficients, so that the results of \cite{BS88} do not apply and have to be generalized (see e.g.\ Subsection \ref{sec:LyaexmomLyafun}). 

The following analysis relies on the Lagrangian flow corresponding to \eqref{eq:mainstochdiffeqwithkappa} which for  each $\kappa\geq 0$ is defined by (see Lemma \ref{lem:phitkat64wppacontinxandkapps} below and \cite[Chapter 3.1]{Hs02})

\begin{equation}\label{eq:equfortildephikappa}
\begin{cases}
  \ud\phi_t^{\kappa}(x) &= \sum_{k\geq 1} \sigma_k(\phi_t^{\kappa}(x)) \circ \ud W_{t}^k + \sqrt{2\kappa} \sum_{m=1}^n \chi_m(\phi_t^{\kappa}(x)) \circ \ud \widetilde W_t^m, \\
 \phi_0^{\kappa}(x) &= x,
\end{cases}
\end{equation}
where $\widetilde W = (\widetilde W^m)_{m=1}^n$ is a standard Brownian motion in $\mathbb R^n$ independent of $(W^k)_{k\geq 1}$ and $x\in \mathcal M$. 

\begin{remark}\label{rem:phikappameasupres} 
Note that $\phi^{\kappa}_t$ is a.s.\ measure preserving for any $t\geq 0$ by \cite[Theorem 4.2]{Bax89} and the formula \cite[(4)]{Cha86} as $(\sigma_k)_{k\geq 1}$ are divergence-free.
\end{remark}

Let us shortly recall how \eqref{eq:mainstochdiffeqwithkappa} is connected with \eqref{eq:equfortildephikappa}. 
Let $(\phi_t^{\kappa}(x))_{t\geq 0}$ be the flow of solutions of \eqref{eq:equfortildephikappa} (see Lemma \ref{lem:phitkat64wppacontinxandkapps} below), let $u\in L^2(\mathcal M)$ be as in \eqref{eq:mainstochdiffeqwithkappa}, and set
\begin{equation}\label{eq:Utviaphikaplxw}
u_t:= \mathbb E_{\widetilde W} u\bigl((\phi_t^{\kappa})^{-1}(\cdot)\bigr):=\mathbb E \bigl(u\bigl((\phi_t^{\kappa})^{-1}(\cdot)\bigr)\big|(W^k)_{k\geq 1}\bigr), \;\; t\geq 0,
\end{equation}
where the conditional expectation $\mathbb E_{\widetilde W}$ is well-defined on $L^p(\overline {\Omega};L^2(\mathcal M))$ for any $p\geq 1$ by \cite[Section 2.6]{HNVW1}, as since  $\phi^{\kappa}_t$ is a measurable $C^1$-diffeomorphism-valued  function, 
$u\bigl((\phi_t^{\kappa})^{-1}(\cdot)\bigr):\Omega\to L^2(\mathcal M)$ is strongly $\overline{\mathcal F}$-measurable by \cite[Theorem 1.1.6]{HNVW1}.

We now show that, for $u\in C^{\infty}(\mathcal M)$, $(u_t)_{t\geq 0}$ defined in \eqref{eq:Utviaphikaplxw} is a solution to \eqref{eq:mainstochdiffeqwithkappa} in the sense of \cite[Definition 4.2.1]{LR15}. In this case $(u_t)_{t\geq 0}$ solves \eqref{eq:mainstochdiffeqwithkappa} in $H^{-1}(\mathcal M)$ analogously to Remark \ref{rem:kappa0icn2938}. From the definition of $u_t$ in \eqref{eq:Utviaphikaplxw} it immediately follows that
$(u_t)_{t\in [0, T]}\!\in\! L^2(\overline{\Omega};C([0, T]; L^2(\mathcal M)))$. In order to show that then $u$ is also a solution to \eqref{eq:mainstochdiffeqwithkappa} it remains to show that 
$
(u_t)_{t\in [0, T]}\!\in\! L^2(\overline{\Omega}\times[0, T]; H^1(\mathcal M)),
$ 
for any $T>0$. To this end notice that 
for any $x\in \mathcal M$ and $v\in T_x\mathcal M$
\[
\langle \nabla u\bigl((\phi^{\kappa}_t)^{-1}(x)\bigr), v\rangle_x = \Bigl\langle \nabla u\bigl((\phi^{\kappa}_t)^{-1}(x)\bigr), \bigl\langle (D\phi^{\kappa}_t)^{-1}(x), v\bigr \rangle \Bigr\rangle_{(\phi^{\kappa}_t)^{-1}(x)},
\]
which yields that $(u_t)_{t\in [0, T]}\in L^{p\vee 2}(\overline{\Omega}\times[0, T]; H^1(\mathcal M))$ for any $T>0$ and $p\geq 1$ by Lemma \ref{lem:phitkat64wppacontinxandkapps} below. 
The fact that $(u_t)_{t\geq 0}$ is the solution of \eqref{eq:mainstochdiffeqwithkappa}  for a general $u\in L^2(\mathcal M)$ then follows from \cite[Proposition 4.2.10]{LR15} and the continuity of a conditional expectation.

\begin{remark}\label{rem:stoppingrimaforformfout}
Note that the formula \eqref{eq:Utviaphikaplxw} holds true for any $\mathbb F$-stopping time $\tau:\Omega \to \mathbb R_+$. Moreover, one can show that in this case
\[
u_{\tau}(x) =\mathbb E \bigl(u\bigl((\phi_{\tau}^{\kappa})^{-1}(x)\bigr)\big|(W^k)_{k\geq 1},(\widetilde W^m_{t+\tau} -\widetilde W^m_{\tau} )_{t\geq 0, m=1,\ldots n}\bigr),\;\;\; x\in \mathcal M
\]
as $\phi_{\tau}^{\kappa}$ is independent of $(\widetilde W^m_{t+\tau} -\widetilde W^m_{\tau})_{t\geq 0, m=1,\ldots n}$.
\end{remark}

Let us start by showing that $\phi_t^{\kappa}(x)$ is continuous in both $x$ and $\kappa$. Recall that $d:\mathcal M\times \mathcal M\to \mathbb R_+$ is the distance function on $\mathcal M$. Recall also that $\mathcal M$ can be isometrically embedded into $\mathbb R^{\ell}$, i.e.\ there exists an isometric embedding $\iota:\mathcal M \hookrightarrow \mathbb R^{\ell}$. Thus, for any $x\in \mathcal M$ there exists the linear isomorphic embedding $\iota_x:=D\iota(x):T_x\mathcal M \hookrightarrow \mathbb R^{\ell}$. Therefore, any linear operator $R:T_x\mathcal M \to T_y\mathcal M$ can be extended to a linear operator $R^{ex}\in \mathcal L(\mathbb R^{\ell})$ as follows: $R^{ex} \iota_x(v) = \iota_{y}(Rv)$ for any $v\in T_x\mathcal M$, and $R^{ex} u = 0$ for any $u\in \mathbb R^{\ell}$ orthogonal to $\iota_x(T_x \mathcal M)$. For any $x_1, x_2, y_1, y_2\in \mathcal M$ and for any $R_1\in \mathcal L(T_{x_1}\mathcal M, T_{y_1}\mathcal M)$ and $R_2\in \mathcal L(T_{x_2}\mathcal M, T_{y_2}\mathcal M)$ we set
\begin{equation}\label{eq:normdiffR1R2difftangsps}
\|R_1-R_2\|:= \left\| R_1^{ex} - R_2^{ex}\right\|_{\mathcal L(\mathbb R^{\ell})}.
\end{equation}

\begin{lemma}\label{lem:phitkat64wppacontinxandkapps}
For any $\kappa\geq 0$ there exists a stochastic flow of $C^1$-diffeomorphisms $(\phi_t^{\kappa})_{t\geq 0}$ on $\mathcal M$ which satisfies \eqref{eq:equfortildephikappa} for any $x\in \mathcal M$. Moreover, for any $T\geq 0$, $\kappa_0>0$, and $p \in \mathbb R$ there exists a constant $C>0$ such that for any  $\kappa, \kappa' \in [0, \kappa_0]$ and $t, t'\in[0, T]$
\begin{equation}\label{eq:phikappaconeinxx'}
\begin{split}
\mathbb E d(\phi^{\kappa}_t(x), \phi^{\kappa'}_{t'}(x'))^p&+ d\bigl((\phi^{\kappa}_t)^{-1}(x),(\phi^{\kappa'}_{t'})^{-1}(x')\bigr)^p \\
&\quad\leq C \Bigl(|\kappa-\kappa'|^{\frac p2} + d(x, x')^p+ |t-t'|^{\frac p2}\Bigr),\;\;\; x, x'\in \mathcal M,
\end{split}
\end{equation}
Further, for any $T>0$ there exists a random variable $\widetilde C:\overline \Omega \to \mathbb R_+$ having moments of all orders $p\geq 1$ such that  for any  $\kappa, \kappa' \in [0, \kappa_0]$, $x, x'\in \mathcal M$, and $t, t'\in [0, T]$
\begin{equation}\label{eq:dphikappatxphikaa't'x'bqd}
\begin{split}
d(\phi^{\kappa}_t(x),\phi^{\kappa'}_{t'}(x')) &+ d\bigl((\phi^{\kappa}_t)^{-1}(x),(\phi^{\kappa'}_{t'})^{-1}(x')\bigr) \\
&\quad\leq \widetilde C \Bigl(|\kappa-\kappa'|^{\frac 14} + d(x,x')^{\frac 34} + |t-t'|^{\frac 14}\Bigr)
\end{split}
\end{equation}
and
\begin{equation}\label{eq:Dphikappa-Dphikappa'bddbykpkaf[w}
\begin{split}
 \| D\phi^{\kappa}_t(x)- D \phi^{\kappa'}_{t'}(x')\| &+ \bigl \| (D\phi^{\kappa}_t(x))^{-1}- (D \phi^{\kappa'}_{t'}(x'))^{-1}\bigr\| \\
 & \quad\leq \widetilde C \Bigl(|\kappa-\kappa'|^{\frac {\beta}4} + d(x,x')^{\frac {\beta}2} + |t-t'|^{\frac 14}\Bigr),
 \end{split}
\end{equation}
where $\beta$ is as in the condition {\textnormal {\bf (C)}} 
and where the norms in \eqref{eq:Dphikappa-Dphikappa'bddbykpkaf[w} are defined by \eqref{eq:normdiffR1R2difftangsps}.
\end{lemma}

\begin{proof}
Let us define a flow $\phi$ on $\mathcal M \times \mathbb R$
\begin{equation}\label{eq:fowwithmudefwvds}
\begin{cases}
  \ud\phi_t(x, \mu) &= \sum_{k\geq 1} \sigma_k(\phi_t(x, \mu)) \circ \ud W_{t}^k \\
  &\quad\quad\quad\quad\quad\quad\quad\quad\quad\quad\quad+ \mu\sum_{m=1}^n \chi_m(\phi_t(x,\mu)) \circ \ud \widetilde W_t^m, \\
 \phi_0(x, \mu) &= x,
\end{cases}
\end{equation}
where $(x, \mu) \in \mathcal M \times \mathbb R$. Note that $\phi(\cdot, \mu)$ coincides with \eqref{eq:equfortildephikappa} for $\mu = \sqrt{2\kappa}$. Then by \cite[Theorem V.26]{Prot} (recall that all our integrals can be considered in $\mathbb R^{\ell}$, see Remark \ref{rem:whyinRell}) we have that
\begin{align*}
 \ud\phi_t(x, \mu) = \sum_{k\geq 1} \sigma_k(\phi_t(x, \mu)) \ud W_{t}^k  &+ \mu\sum_{m=1}^n \chi_m(\phi_t(x,\mu)) \ud  \widetilde W_t^m\\
& + \frac{1}{2}\sum_{k\geq 1}\langle D\sigma_k(\phi_t(x, \mu)), \sigma_k(\phi_t(x, \mu))\rangle \ud t\\
& + \frac{\mu}{2} \sum_{m=1}^n\langle D\chi_m(\phi_t(x, \mu)), \chi_m(\phi_t(x, \mu))\rangle \ud t.
\end{align*}

Therefore the flow \eqref{eq:fowwithmudefwvds} has the local characteristics $(a, b, A)$ (see \cite[pp.\ 79--84]{KSF90}) defined for any $t\geq 0$, $ x, x'\in \mathcal M$, and $\mu, \mu'\in \mathbb R$ by
\begin{equation}\label{eq:loccharforphimufd}
\begin{split}
a((x, \mu), (x', \mu'),t) &= \sum_{k\geq 1}\langle  \sigma_k(x),  \sigma_k(x') \rangle + \mu \mu'\sum_{m=1}^n \langle  \chi_m(x),  \chi_m(x') \rangle,\\
b((x, \mu), t)&= \frac{1}{2}\sum_{k\geq 1}\langle D\sigma_k(x), \sigma_k(x)\rangle+ \frac{\mu}{2} \sum_{m=1}^n\langle D\chi_m(x), \chi_m(x)\rangle,\\
A_t &= t,
\end{split}
\end{equation}
where the inner product $\langle \cdot, \cdot\rangle$ is defined in Section \ref{sec:prelim}.
Note that by \eqref{eq:nescondonsigmaintermsofsums} and condition {\textnormal {\bf (C)}} we have that both $a(\cdot, \cdot, t)$ and $b(\cdot, t)$ are in $C^{1, \beta}$ for any fixed $t\geq 0$ and for $\beta>0$ from the condition {\textnormal {\bf (C)}}. Consequently, by \cite[Theorem 4.6.5]{KSF90} we have that the flow $\phi$ is a stochastic flow of $C^{1, \beta}$ diffeomorphisms. Therefore \eqref{eq:phikappaconeinxx'} for $\phi^{\kappa}$  follows from \cite[Lemma 4.5.5 and 4.5.6]{KSF90}.

Equation \eqref{eq:phikappaconeinxx'} for $(\phi^{\kappa})^{-1}$ follows analogously to the same inequalities for $\phi^{\kappa}$ and the fact that $\phi$ is a  {\em backward} Brownian flow. Indeed, as $\phi$ is a Brownian flow (since its local characteristics are deterministics), by modifying \cite[Theorem 4.2.10]{KSF90} (note that thanks to the proof of \cite[Theorem 4.2.10]{KSF90} we do not need $\phi$ to be a $C^2$ flow necessarily but only that $\sum_{k \geq 1} \langle D\sigma_k, \sigma_k\rangle$ being $C^{1, \beta}$, which is guaranteed by {\textnormal {\bf (C)}}) we have that $\phi$ has a $C^1$ backward stochastic flow with the local characteristics $(a, -b, A)$, so the desired holds true.

Let us show \eqref{eq:dphikappatxphikaa't'x'bqd} and \eqref{eq:Dphikappa-Dphikappa'bddbykpkaf[w}. First, \eqref{eq:dphikappatxphikaa't'x'bqd} follows from \eqref{eq:phikappaconeinxx'} and \cite[Theorem 1.4.1]{KSF90}. Next, analogously to \cite[Theorem 4.6.4, and the proof of Corollary 4.6.7]{KSF90} thanks to \eqref{eq:loccharforphimufd} and Remark \ref{rem:whyinRell} we can show that there exists $c>0$ such that for any fixed $t, t'\in[0, T]$, $x, x' \in \mathcal M$, and $\mu, \mu'\in [0, \sqrt{\kappa_0}]$
\begin{equation}\label{eq:EcowemkDphikappatx-Dphikappa't'x'}
\mathbb E \|D\phi_t(x, \mu) - D\phi_{t'}(x', \mu')\|^p \leq c \bigl (|\mu-
\mu'|^{p\beta} + d(x, x')^{p\beta} + |t-t'|^{ \frac p2}\bigr).
\end{equation}
Let 
$$
\widetilde C:= \sup_{x, x'\in\mathcal M, \mu, \mu'\in [0, \sqrt{\kappa_0}], t, t'\in[0, T]}\frac{|D\phi_t(x, \mu) - D\phi_{t'}(x', \mu')|}{|\mu-\mu'|^{\frac {\beta}2} + |x-x'|^{\frac {\beta}2} + |t-t'|^{\frac 14}},
$$
where we set $\frac 00 = 0$ for simplicity. Then by \cite[Theorem 1.4.1]{KSF90} $\widetilde C$ has moments of all orders (we leave the technical details to the reader), which is exactly \eqref{eq:Dphikappa-Dphikappa'bddbykpkaf[w}. The part of \eqref{eq:Dphikappa-Dphikappa'bddbykpkaf[w} concerning $ \| (D\phi^{\kappa}_t(x))^{-1}- (D \phi^{\kappa'}_{t'}(x'))^{-1}\|$ follows similarly via exploiting that $\phi$ is a  backward Brownian flow.
\end{proof}

Later we will need the following corollaries.

\begin{corollary}\label{cor:supofDphiontsciintegrs}
Fix $\kappa_0$ and $T>0$. Then for any $p\geq 1$ there exist $\delta>0$ and an integrable random variable $C$ such that a.s.
\begin{equation}\label{eq:supofDintkappaandx}
\sup_{\kappa\in[0, \kappa_0],t\in [0, T],x\in \mathcal M}\Bigl(\| D\phi^{\kappa}_t(x) \!-\! I_{T_{x}\mathcal M}\|^p+\bigl \| \bigl(D\phi^{\kappa}_t(x)\bigr)^{-1}\! -\! I_{T_{\phi^{\kappa}_t(x)}\mathcal M}\bigr\|^p\Bigr)<CT^{\delta}.
\end{equation}
\end{corollary}

\begin{proof}
The corollary is a direct consequence of \eqref{eq:Dphikappa-Dphikappa'bddbykpkaf[w}.
\end{proof}

\begin{corollary}\label{cor:Fstopingtimwlidtstochdlow}
Let $\tau$ be an $\mathbb F$-stopping time. Then for any $\kappa \geq 0$ we have that $(\phi^{\kappa,\tau}_t)_{t\geq 0}:= (\phi^{\kappa}_{t+\tau}(\phi^{\kappa}_{\tau})^{-1})_{t\geq 0}$ is a Brownian flow of homeomorphisms. Moreover, $(\phi^{\kappa,\tau}_t)_{t\geq 0}$ depends only on $(W^k_{t+\tau} -W^k_{\tau} )_{t\geq 0,k\geq 1} $ and $(\widetilde W^m_{t+\tau} -\widetilde W^m_{\tau} )_{t\geq 0, m=1,\ldots, n}$ and it has the same distribution as $(\phi^{\kappa}_t)_{t\geq 0}$ as a random element in $C(\mathbb R_+; C(\mathcal M; \mathcal M))$.
\end{corollary}

\subsection{Lyapunov exponents and moment Lyapunov functions for SDEs with H\"older coefficients}\label{sec:LyaexmomLyafun}

Let stochastic flow $(\phi_t)_{t\geq 0}$ be defined by \eqref{eq:flowforkappa=0dsa} with $(\sigma_k)_{k\geq 1}$ satisfying the conditions {\textnormal {\bf (A)--(C)}}.
The goal of the subsection is to show that in this case the  the leading Lyapunov exponent
\begin{equation}\label{eq:topLyapexpdef}
\lambda_1 := \lim_{t\to \infty} \frac 1t \log \|D\phi_t(x)\|
\end{equation}
exists, is independent of choice of $x\in \mathcal M$ and positive and that the moment Lyapunov function 
 \begin{equation}\label{eq:defofLambdap}
    \Lambda(p) := -\lim_{t\to \infty} \frac 1t \log \mathbb E_{(x,v)} |D\phi_t(x) v|^{-p},\;\;\; x\in \mathcal M , \;\;v \in S_x\mathcal M.
 \end{equation}
exists for any $p\in [-p_*, p_*]$ for some $p_*>0$ so that the following proposition holds true. 

 \begin{proposition}\label{prop:propofLambdapandpsipliekconx}
Let $\Lambda(p)$ be defined by \eqref{eq:defofLambdap}. Then
\begin{enumerate}[\rm (A)]
\item $p\mapsto \Lambda(p)$ is continuous and concave in $p\in [-p_*,p_*]$,
\item $\Lambda(p) = \lambda_1 p + o(p)$ as $p\to 0$, where $\lambda_1$ is defined by \eqref{eq:topLyapexpdef}.
\end{enumerate}
\end{proposition}

A very similar statement has been shown in \cite{AOP86,BS88} for smooth and finitely many coefficients $\sigma_k$. However, one main aim in this work is to include the full range of coefficients amenable to the analysis via the associated Lagrangian stochastic flow (cf.\ p.\ \pageref{pg:whyonly2+beta} in the introduction). In this case, 
the corresponding PDE for $\phi$ and $\tilde \phi$ (see \eqref{eq:flowforkappa=0dsa} and \eqref{eq:norm_tangent_flow}) do not have $C^{\infty}$, but only H\"older continuous coefficients. Therefore, the H\"ormander theorem cannot be directly applied, as it was done in \cite{AOP86,BS88}, but one needs to use ellipticity and  to adapt the techniques from \cite{BBP-S,BBP-S1911a}.

\begin{proposition}\label{prop:Lyapconstexistsandpos}
$\lambda_1$ defined by \eqref{eq:topLyapexpdef} exists and is positive.
\end{proposition}

\begin{proof}
Existence of $\lambda_1$ follows from \cite[Theorem 2.2 and Corollary 2.3]{Bax89}. Let us show that $\lambda_1>0$. To this end we will need to use \cite[Theorem 6.8]{Bax89} (see also \cite[p.\ 4]{DKK04}), which states that as $\phi$ is measure-preserving by Remark \ref{rem:phikappameasupres}, $\lambda_1>0$ if
\begin{enumerate}[(i)]
\item  for any $t>0$
\begin{multline*}
\mathbb E \int_{\mathcal M}\sup_{0\leq s\leq t}\log^+\|D\phi_s(x)\| + \log^+\|(D\phi_s(x))^{-1}\|\\
+\log^+\|D\phi_t\phi_s^{-1}(x)\| + \log^+\|(D\phi_t\phi_s^{-1}(x))^{-1}\|\ud \mu(x)<\infty,
\end{multline*}
\item the relative entropy $\mathbb Eh(\mu, \mu_t)$ is finite, { where $\mu_t:= \mu\circ \phi_t^{-1}$} (see \cite[p.\ 523]{Bax89}),
\item there is no Riemannian distance $d'$ on $\mathcal M$ which is preserved by $\phi$,
\item there are no proper (i.e.\ of dimension no less than 1 and no more than $d-1$) tangent subspaces $E_x^1,\ldots E_x^p$ of $T_x\mathcal M$, $x\in \mathcal M$, such that $D\phi_t(x)E^i_x = E^{\sigma_t(i)}_{\phi_t(x)}$, where $\sigma_t$ is some permutation depending only on $\phi$, $t$, and $x$.
\end{enumerate}

Note that $\phi$ is a Brownian flow which is a backward flow as well (see the proof of Lemma \ref{lem:phitkat64wppacontinxandkapps}), so $(i)$ follows from \eqref{eq:Dphikappa-Dphikappa'bddbykpkaf[w} implied to both flow $(\phi_s)_{0\leq s\leq t}$ and the backward flow $(\phi_t\phi_{t-s}^{-1})_{0\leq s\leq t}$. $(ii)$ holds as $\mathbb Eh(\mu, \mu_t)=0$, where $\mu_t:= \mu\circ \phi_t^{-1} = \mu$ as $\phi$ is measure-preserving by Remark \ref{rem:phikappameasupres}. 
 
Both conditions $(iii)$ and $(iv)$ hold by Proposition \ref{prop:appendtkernelHoldcoegg} (see also \cite[Section 2.1]{DKK04}). Indeed, for any $x, y\in \mathcal M$, $x\neq y$, the measure $\nu$ on $(0, \infty)\times \mathcal M\times \mathcal M$ equalling 
$$
\nu([s, t)\times A\times B) = \int_s^t \mathbb P(\phi_r(x)\in A, \phi_r(y)\in B)\ud r,\;\;\; 0\leq s\leq t,\;\;\; A, B\in \mathcal B(\mathcal M),
$$
solves the Fokker-Planck equation \eqref{eq:appenexistskernelHoldcoeff} with $a^{ij}:= \sum_{k\geq 1}(\sigma_k(\cdot), \sigma_k(\cdot))$, $b^i:= \tfrac 12 \sum_{k\geq 1}(\langle D\sigma_k(\cdot), \sigma_k(\cdot),\langle D\sigma_k(\cdot), \sigma_k(\cdot))$, and $c=0$, so due to the conditions  {\textnormal {\bf (A)--(C)}} and Proposition \ref{prop:appendtkernelHoldcoegg} $\nu$ has a density with respect to $\ud t\dd\mu\dd\mu$. Therefore there is no Riemannian distance $d'$ so that $d'(\phi_t(x), \phi_t(y))=d'(x, y)$ a.s.\ for any $t\geq 0$ so the  condition $(iii)$ is satisfied. The same can be shown for condition $(iv)$ (the corresponding $a^{ij}$ and $b^i$ generated by $(\phi_t(\cdot), \tilde \phi_t(\cdot))$ are then H\"older continuous, see Section \ref{sec:mixwithstoch}).
\end{proof}

Let us now prove existence of the moment Lyapunov function $\Lambda(p)$ defined by \eqref{eq:defofLambdap}. 
For any $p\in \mathbb R$ and $\psi\in C(S\mathcal M)$ define the ``twisted'' Markov semigroup $\hat{P}^{p}:\mathbb R_+ \to \mathcal L(C(S\mathcal M))$ by
\begin{equation}\label{eq:defofPhatpkappadsa}
 \hat{P}^{p}_t \psi(x, v) := \mathbb E_{(x, v)} |D \phi_t v|^{-p} \psi(x_t, v_t),\;\;\;x\in \mathcal M,\;\; v\in S_x\mathcal M,\;\; t\geq 0,
\end{equation}
{ where for any $v\in S_x\mathcal M$ we set $v_t:= \tilde \phi_t(v) = \tfrac {D\phi_t v}{|D\phi_t v|}$ (see \eqref{eq:norm_tangent_flow}).}

For any $t\geq 0$ set $\hat{P}_t:=\hat{P}^{0}_t$. Note that $\hat{P}_t$ has an eigenvector $\mathbf 1$ with the eigenvalue $1$. Let us show that this eigenvalue is leading and that $\hat{P}^{0}_t$ has a spectral gap in $1$ for $t$ big enough.

\begin{lemma}\label{lem:Pt0hasspectgap}
 For $T\geq 0$ big enough $\hat{P}_{T}$ as an operator on $C(S\mathcal M)$ has a spectral gap in point $\{1\}$, i.e.\ there exists $\eps\in(0,1)$ so that $\{1\} \subset \sigma(\hat{P}_{T}) \subset \{1\} \cup B(0, \eps)$.
\end{lemma}

\begin{proof}
 For the proof we apply Theorem \ref{thm:QuaeHarwqthnms}, which condition $(i)$ is satisfied for $V\equiv 1$, $\gamma=\frac 12$, $C=1$, and which condition $(ii)$ holds thanks to Proposition \ref{prop:appendxHarnackmancewell} and condition {\textnormal {\bf (B)}}. Therefore there exists an invariant measure $\nu$ of $(\hat{P}_t)_{t\geq 0}$. Let $C^0_{\nu}(S\mathcal M)$ denote the subspace of $C(S\mathcal M)$ of all continuous mean-zero with respect to $\nu$ functions.
 Fix $0<\delta <1/10$.
 It is sufficient to show that for $T$ big enough
 \begin{equation*}\label{eq:Pthatissmallonmeanzero}
 \|\hat{P}_{T}\|_{\mathcal L(C^0_{\nu}(S\mathcal M), C(S\mathcal M))}<\delta
 \end{equation*} 
 which follows directly from \eqref{eq:maingofthm:QuaeHarwqthnms}.
\end{proof}

Our goal is to show the same for $(\hat{P}^{p}_t)_{t\geq 0}$ for $p$ small enough. We start with the following lemma.

\begin{lemma}\label{lem:coninpofPkappap}
 Fix $t\geq 0$. Then for any $p\in \mathbb R$
 $$
\lim_{q\to p}   \|\hat{P}^{p}_t - \hat{P}^{q}_t\|_{\mathcal L(C(S\mathcal M))} = 0.
 $$ 
\end{lemma}

\begin{proof}
It is sufficient to note that
   \begin{equation}\label{eq:uppebdforCcontofPhatpkappa}
  \begin{split}
 &\left\|\mathbb E_{(x, v)} |D \phi_t v|^{-p} \bigl||D \phi_t v|^{-(q-p)} - 1\bigr| \psi(x_t, v_t) \right\|_{\infty}\\
    &\quad\leq \left\|\mathbb E_{(x, v)} |D \phi_t v|^{-p}\bigl||D \phi_t v|^{-(q-p)} - 1\bigr| \right\|_{\infty} \|\psi\|_\infty\\
    &\quad = \sup_{x\in \mathcal M, v\in S_x\mathcal M} \bigl|\mathbb E_{(x, v)}|D \phi_t v|^{-p} \bigl(| D \phi_t v|^{-(q-p)} - 1\bigr)\bigr|\\
    & \quad= \sup_{x} \max\Big\{\mathbb E_{x}\|D \phi_t \|^{|p|}|\|D \phi_t\|^{p-q} - 1|,\\
    &\quad\quad\quad\quad\quad\quad\quad\quad\quad\quad\mathbb E_{x}\bigl\|(D \phi_t)^{-1}\bigr\|^{|p|}\left|\bigl\| (D \phi_t)^{-1}\bigr\|^{p-q} - 1\right|\Big\}\\
    &\quad \stackrel{(*)} \lesssim_{\phi,p}  
    |p-q| \sup_{x} \max\Big\{\mathbb E_{x}\left (\|D \phi_t\|^{p-q} \log \|D \phi_t\| \right)^{2},\\
    &\quad\quad\quad\quad\quad\quad\quad\quad\quad\quad\mathbb E_{x}\left(\bigl\| (D \phi_t)^{-1}\bigr\|^{p-q} \log \bigl\| (D \phi_t)^{-1}\bigr\| \right)^{2}\Big\}^{1/2}
  \end{split}
 \end{equation}
 where 
 $(*)$ follows from Corollary \ref{cor:supofDphiontsciintegrs}, H\"older's inequality, and the fact that
 \begin{equation}\label{eq:obvineqforab-1}
 |a^b-1| \leq |b \log a| a^b,\;\;\;\; a>0,\;\; b\in \mathbb R.
 \end{equation}
 The latter expression of \eqref{eq:uppebdforCcontofPhatpkappa} vanishes as $q\to p$ by  Corollary \ref{cor:supofDphiontsciintegrs}.
\end{proof}

Therefore the following proposition holds true.

\begin{proposition}\label{prop:appendp_8ceij}
There exist $p_*>0$ and $T>0$ such that $\hat P^{p}_T$ has a real leading simple eigenvalue $\lambda(p)$ for any $p\in[-p^*, p^*]$, i.e.\ there exist $\eps\in(0, 1/4)$, a one dimensional projection $\pi^p\in \mathcal L(C(S\mathcal M))$, and an eigenvalue $\lambda(p)>0$ of $\hat P^p_T$ such that $|1-\lambda(p)|<\eps$, $\|\hat P^p_T - \lambda(p)\pi^p\|_{\mathcal L(C(S\mathcal M))}<\eps$, and $\hat P^p_T$ and $\pi^p$ commute. Moreover, for the corresponding eigenfunction $\psi_p\in C(S\mathcal M)$ one has that $\|\psi_p-1\|_{\infty}\to 0$ as $p\to 0$.
\end{proposition}

\begin{proof}
Let $\nu$ be as in the proof of Lemma \ref{lem:Pt0hasspectgap} and $\pi\in \mathcal L(C(S\mathcal M))$ be the corresponding projection $\pi f = \int f \ud \nu$. Fix $\delta<1/10$ and fix $T>0$ to be such that $\|\hat P_T - \pi\|_{\mathcal L(C(S\mathcal M))}<\delta$ (which is possible by Lemma \ref{lem:Pt0hasspectgap}).

By Lemma \ref{lem:coninpofPkappap} for $p$ small enough we have that for any $z\in \mathbb C$ with $|z-1|=2\delta$ one has that $(z-\hat P^p_T)^{-1}$ is uniformly bounded in $z$. Indeed, for any $z$ with $|z-1|=2\delta$
 \[
 (z-\hat P_T)^{-1} = (I-(z-\pi)^{-1}(\hat P_T - \pi))^{-1} (z-\pi)^{-1},
 \]
  so (here and later in the proof we will for simplicity write $\|\cdot\|$ instead of \linebreak $\|\cdot\|_{\mathcal L(C(S\mathcal M))}$)
   \begin{multline*}
\| (z-\hat P_T)^{-1}\| = \|(I-(z-\pi)^{-1}(\hat P_T - \pi))^{-1}\|\| (z-\pi)^{-1}\|\\
 \leq \tfrac 12\delta^{-1} (1-\tfrac 12\delta^{-1}\|\hat P_T - \pi\|)^{-1} \leq \delta^{-1}.
 \end{multline*}
Therefore we can fix $p_0>0$ such that for $p\in [-p_0, p_0]$ one has that by Lemma \ref{lem:coninpofPkappap} $\|\hat P^p_T-\hat P_T\|\leq \delta^3$ so
  \[
  \|(z-\hat P^p_T)^{-1}\| = \|(I-(z-\hat P_T)^{-1}(\hat P^{p}_T- \hat P_T))^{-1} (z-\hat P_T)^{-1}\|<2\delta^{-1}
  \]
   is uniformly bounded in $z$ with $|z-1|=2\delta$. 
In particular, by Lemma \ref{lem:coninpofPkappap}
  \begin{equation}\label{eq:approxPhatpandapcda15}
  \begin{split}
&\| (z-\hat P^p_T)^{-1}  - (z-\hat P_T)^{-1} \|\\
 &\quad\quad\quad\quad\quad= \| (z-\hat P^p_T)^{-1} (\hat P_T - \hat P^p_T)  (z-\hat P_T)^{-1}\|\to 0,\;\;\; p\to 0.
 \end{split}
 \end{equation}
 For any $p\in[-p_0, p_0]$ by operator Cauchy's integral formula define the following spectral projector (we refer the reader to \cite{SchLN20} for further acquaintance with spectral projectors)
 \begin{equation}\label{eq:defofpipappcmw}
 \pi^{p} := \frac{1}{2\pi i} \int_{|z-1|=2\delta}(z-\hat P^p_T)^{-1} \ud z.
 \end{equation}
 Then $\|\pi^p-\pi\|\to 0$ as $p\to 0$ thanks to \eqref{eq:approxPhatpandapcda15}. Note that as $\pi^p$ is a projection and as the spectrum of $\pi^p$ lies inside the circle $\{z\in \mathbb C:|z-1|\leq 2\delta\}$, there exists small $p_*>0$ such that $\pi^p$ is a rank one operator and such that $\|\pi^p-\pi\|<\delta$ for all $p\in [-p_*, p_*]$.
 
As $\pi^p$ is a spectral projector, let $\psi_p:= \pi^p\mathbf 1\in C(S\mathcal M;\mathbb C)$ and $\lambda(p)\in \mathbb C$ be such that $\pi^p\psi_p=\psi_p$, $\hat P^p_T\psi_p=\lambda(p)\psi_p$. Then both $\psi_p$ and $\lambda(p)$ are real as $\pi^p$ is real ($\overline \pi^p = \pi^p$ by \eqref{eq:defofpipappcmw}). Moreover, as $\|\pi^p-\pi\|<\delta$ by \eqref{eq:approxPhatpandapcda15} and \eqref{eq:defofpipappcmw} and as $\|\hat P^p_T-\hat P_T\|\leq \delta^3$, we have that $\|\psi_p-\mathbf 1\|_{\infty}<\delta$ and $|\lambda(p)-1|<2\delta$.

 Note that $\|\psi_p-1\|_{\infty}\to 0$ due to the definition of $\psi_p$, as $\|\pi^p-\pi\|\to 0$ in the operator norm, and as $\pi \mathbf 1=\mathbf1$.
\end{proof}

Now let us show that $\lambda(p)=e^{-\Lambda(p)T}$ with $\Lambda(p)$ defined by \eqref{eq:defofLambdap}.

\begin{proposition}\label{prop:Lambda(p)defandpsiprequxwqxwq}
Let $p\in [-p_*, p_*]$. Then for any $t>0$, $\hat P^p_t$ has an eigenvalue $e^{-\Lambda(p)t}$ with $\Lambda(p)$ defined by \eqref{eq:defofLambdap}. In particular, $\lambda(p)$ from the proof of Proposition \ref{prop:appendp_8ceij} equals to $e^{-\Lambda(p)T}$. Further, $\psi_p$ is the corresponding eigenfunction with $\psi_p\in C^{2+\beta/2}(S\mathcal M)$.
\end{proposition}

\begin{proof}
Let $\pi^p$ be defined by \eqref{eq:defofpipappcmw}. Then due to the proof of Proposition \ref{prop:appendp_8ceij} $\|\hat P^p_T - \lambda(p) \pi^p\|<2\delta$. Let $R^p:= \hat P^p_T - \lambda(p) \pi^p$. By the definition of $\psi_p$ (see the proof of Proposition \ref{prop:appendp_8ceij}), the fact that $\pi^p\psi_p=\psi_p$, and the fact that $\pi^p$ is a spectral projector (so $\pi^pR^p = R^p\pi^p=0$) for any natural $n\geq 1$ we have that
\[
\mathbb E_{(x, v)} |D\phi_{nT}(x)v|^{-p} = (\hat P^p_T)^n\mathbf 1(x, v) = \lambda(p)^{n-1}\psi_p(x, v) + (R^p)^n \mathbf 1(x, v),
\]
so as $\|\psi_p-\mathbf 1\|_{\infty}<\delta$ and as $\|R^p\|<2\delta\leq1/5$ we can define 
\begin{equation}\label{eq:appendefofLampbda(p)}
\Lambda(p):= -\frac{1}{nT}\lim_{n\to \infty}\log\mathbb E_{(x, v)} |D\phi_{nT}(x)v|^{-p}, 
\end{equation}
 so that $e^{-\Lambda(p)Tn}$ is an eigenvalue of $\hat P^p_{nT}$ with the corresponding eigenvector $\psi_p$. The same statement for general $t$ follows from \cite[Corollary A-III.6.4]{AGGG86}. \eqref{eq:defofLambdap} holds by the fact that $\sup_{0\leq t\leq T}\|\hat P^p_t\| <\infty$ thanks to \eqref{eq:Dphikappa-Dphikappa'bddbykpkaf[w}, so \eqref{eq:defofLambdap} coincides with \eqref{eq:appendefofLampbda(p)}.

Finally, let us show that $\psi_p\in C^{2+\beta/2}(S\mathcal M)$. First note that analogously to \cite[Lemma 2.3]{AOP86} (see also the formula \cite[(3.14)]{BS88}) $(\hat P^p_t)_{t\geq 0}$ has a generator $\hat L^p$ of the form (here $\psi \in C^{\infty}(S\mathcal M)$, $(x, v)\in S\mathcal M$)
\begin{align*}
\hat L^p \psi(x,v) &= \frac 12 \sum_{k\geq 1} \bigl\langle \partial_{x,x} \psi(x, v),(\sigma_k(x), \sigma_k(x)) \bigr \rangle + \bigl\langle \partial_{v,v} \psi(x, v),(\tilde\sigma_k(x,v), \tilde\sigma_k(x,v)) \bigr \rangle\\
&\quad + 2  \bigl\langle \partial_{x,v} \psi(x, v),(\sigma_k(x), \tilde\sigma_k(x,v)) \bigr \rangle \\
&\quad- 2p\bigl(\langle \partial_x \psi(x, v), \sigma_k(x) \rangle + \langle \partial_v \psi(x, v), \tilde \sigma_k(x,v) \rangle\bigr)\bigl \langle v, \langle D\sigma_k(x), v\rangle\bigr \rangle\\
&\quad +\psi(x, v)\Big[ p(p+2) \bigl \langle v, \langle D\sigma_k(x), v\rangle\bigr \rangle^2 - p \langle D\sigma_k(x), v\rangle^2\\
&\quad\quad-p\bigl \langle v,\langle D^2 \sigma_k(x),(v, \sigma_k(x))\rangle \bigr\rangle-p\bigl \langle v, \langle D\sigma_k(x), \langle D\sigma_k(x), v\rangle\rangle\bigr\rangle \Big]
\end{align*}
(we leave the calculations to the reader), which coefficients are H\"older-continuous. Fix $q\geq 1$. Due to the condition {\textnormal {\bf (B)}} and \cite[Theorem 1.5.1 and 8.1.1]{KrS08} we have that there exists a constant $C_p$ such that $\|\hat L^p \phi\|_{L^q(\mathcal M)} +C_p\| \phi\|_{L^q(\mathcal M)} \eqsim \| \phi\|_{W^{2, q}}$ for any $\phi\in C^{\infty}(S\mathcal M)$. In particular, as $\hat L^p$ is a generator of $(\hat P^p_t)_{t\geq 0}$,
\begin{multline*}
\|\psi_p\|_{W^{2, q}} \eqsim \|\hat L^p \psi_p\|_{L^q(\mathcal M)} +C_p\| \psi_p\|_{L^q(\mathcal M)}\\
 = \lim_{t\to 0} \Bigl \| \frac{\hat P^p_t \psi_p - \psi_p}{t} \Bigr\|_{L^q(\mathcal M)} +C_p\| \psi_p\|_{L^q(\mathcal M)}=(C_p+|\Lambda(p)|)\| \psi_p\|_{L^q(\mathcal M)}<\infty.
\end{multline*}

By choosing $q$ high enough and by using the Sobolev embedding theorem we get that $\psi_p\in C^{2-\eps}(S\mathcal M)$ for any $\eps>0$. Now, fix a local chart $U\subset S\mathcal M$ and fix a $C^{\infty}$-function $\zeta:U\to[0, 1]$ which equals $1$ within a ball $B_1$ in $U$ and which vanishes outside another ball $B_2\subset U$ compactly containing $B_1$. It is sufficient to show that $\psi_p \zeta \in C^{2+\beta/2}(B_2)$. To this end it is enough to notice that $\psi_p \zeta$ is the unique $W^{2, 2}$ solution $\phi$ of the equation $\tilde L \phi-\lambda \phi =f$ for some fixed $\lambda>1$ thanks to \cite[Theorem 1.5.4]{KrS08}, where $\tilde L$ is a pure quadratic term of $\hat L^p$ and where $f$ is defined via $\psi_p$ and $\zeta$ in the corresponding way and is $C^{1-\eps}$-H\"older continuous. This unique solution must be $C^{2+\beta/2}$ by \cite[Theorem 6.5.3]{KrH96}.
\end{proof}

Let us finally show that $\Lambda(p)'|_{p=0}=\lambda_1$.

\begin{proof}[Proof of Proposition \ref{prop:propofLambdapandpsipliekconx}]
Concavity of $p\mapsto \Lambda(p)$ follows directly from \eqref{eq:defofLambdap}. The continuity follows from the concavity.

Let us now show that $\Lambda(p)'|_{p=0}=\lambda_1$.
First note that by Jensen's inequality
\begin{multline*}
\Lambda(p) = -\lim_{t\to \infty} \frac 1t \log \mathbb E_{(x, v)} |D\phi_t(x)v|^{-p}\\
 \leq  -\lim_{t\to \infty} \frac 1t \mathbb E_{(x, v)} \log |D\phi_t(x)v|^{-p} = p \lim_{t\to \infty} \frac 1t \mathbb E_{(x, v)} \log |D\phi_t(x)v| = p\lambda_1,
\end{multline*}
where the latter equality follows e.g.\ from \cite[Theorem 2.2 and Corollary 2.3]{Bax89}.

It remains to show that $p\mapsto \Lambda(p)$ is continuously differentiable in $p\in[-p_*, p_*]$. Fix $T$ as in the proof of Proposition \ref{prop:appendp_8ceij}. It is sufficient to prove that $p\mapsto \lambda(p)=e^{-\Lambda(p)T}$, $p\in [-p_*,p_*]$, is  continuously differentiable. To this end first notice that $p\mapsto \hat P^p_t$ is $C^{\infty}$ in $p\in \mathbb R$ by \eqref{eq:defofPhatpkappadsa}. Therefore $p\mapsto \psi_p= \pi^p\mathbf 1$, $p\in [-p_*,p_*]$, is  continuously differentiable by \eqref{eq:defofpipappcmw}, and hence $p\mapsto \lambda(p) \psi_p= \hat P^p_T \psi_p$, $p\in [-p_*,p_*]$ is $C^1$ as well. As $\psi_p$ does not vanish (recall that $\|\psi_p-\mathbf 1\|_{\infty}<2\delta $ by the proof of Proposition \ref{prop:appendp_8ceij}), $\lambda(p)$ is continuously differentiable in $p\in [-p_*, p_*]$.
\end{proof}

\subsection{Construction of a Lyapunov function for the two-point motion}\label{subsec:constrofVkappa}
The goal of this section is to construct a Lyapunov function in order to prove the drift condition $(i)$ from the Harris' Theorem \ref{thm:QuaeHarwqthnms} for the two-point motion.
Let $(P^{(2), \kappa}_t)_{t\geq 0}$ be the Markov semigroup on $C(\mathcal D^c)$ associated with the two-point process generated by \eqref{eq:equfortildephikappa}.
This means that for any continuous $\psi: \mathcal D^c \to \mathbb R$ we have that
\begin{equation}\label{eq:P(2)kappadefinition}
 P^{(2), \kappa}_t \psi(x, y) = \overline{\mathbb E}_{x, y} \psi(x^{\kappa}_t, y^{\kappa}_t),\;\;\; t\geq 0,
\end{equation}
where we denote $(x^{\kappa}_t)_{t\geq 0} := (\phi_t^{\kappa}(x))_{t\geq 0}$ and $(y^{\kappa}_t)_{t\geq 0} := (\phi_t^{\kappa}(y))_{t\geq 0}$ for simplicity and where the processes $x^{\kappa}$ and $y^{\kappa}$ are driven by the {\em same noise paths} $(W^k)_{k\geq 1}$ and $(\widetilde W^m)_{m=1}^n$. Let $ P^{(2)}_t:=  P^{(2), 0}_t$, $\phi_t:= \phi^{0}_t$, $x_t:= x^0_t$, and $y_t:= y^0_t$ for any $t\geq 0$. Note that as  $(\phi_t^{\kappa})_{t\geq 0}$ is a flow of diffeomorphisms, a.s.\ $x^{\kappa}_t \neq  y^{\kappa}_t$ for any $x\neq y$.

We would like to show that there exist $p>0$ and $\widetilde V_p:\mathcal D^c \to \mathbb R_+$ such that $\widetilde V_p(x, y)\eqsim d(x, y)^{-p}$ for all $(x, y)\in \mathcal D^c$ and $P^{(2), \kappa}_t\widetilde V_p \leq \lambda \widetilde V_p + C$ for some fixed $t>0$, $\lambda\in (0, 1)$, and $C\geq 0$, and for any $\kappa\in[0, \kappa_0]$. Such a function $\widetilde V_p$ for the case of finitely many $\sigma_k$'s and $\kappa=0$ was constructed by Baxendale and Stroock in \cite{BS88}. The construction provided in \cite{BS88} cannot be applied to the case considered in this work, since it relies on the smoothness of the coefficients. Therefore, we need to recheck and adapt the approach from \cite{BS88} to the present irregular setting. 

The major tool used in \cite{BS88}, which we are going to exploit as well is
 the ``twisted'' Markov semigroup $\hat{P}^{p}:\mathbb R_+ \to \mathcal L(C(S \mathcal M ))$ which linearizes the two-point motion and which is defined by 
\begin{equation*}
 \hat{P}^{p}_t \psi(x, v) := \mathbb E_{(x, v)} |D \phi_t v|^{-p} \psi(x_t, v_t),\;\;\;x\in \mathcal M,\;\; v\in S_x \mathcal M,\;\; t\geq 0,
\end{equation*}
 where for any $v\in S_x \mathcal M$ we set $v_t:= \tilde \phi_t(v) = \tfrac {D\phi_t v}{|D\phi_t v|}\in S_{x_t} \mathcal M$ (see \eqref{eq:norm_tangent_flow} and Subsection \ref{sec:LyaexmomLyafun} for more information).
Then according to Subsection \ref{sec:LyaexmomLyafun} the following proposition holds true.

\begin{proposition}\label{prop:maintextpsipexinsandinrwc}
There exists $p_*>0$ such that for any $p\in [-p_*, p_*]$ there exists $\psi_p\in C^{2+\beta/2}(S \mathcal M)$ with $\|\psi_p-\mathbf 1\|_{\infty}<1/10$ such that $ \hat{P}^{p}_t \psi_p = e^{-\Lambda(p)t}\psi_p$ for any $t\geq 0$.
\end{proposition}

 For any $q\in \mathbb R$ let us define $C_q(\mathcal D^c)$ to be the space of all continuous functions $\psi$ on $\mathcal D^c$ which norm
\[
\|\psi\|_{C_q(\mathcal D^c)} := \sup_{(x, y) \in \mathcal D^{c}} \frac{|\psi(x, y)|}{d(x, y)^{-q}},
\]
is finite and $\psi(x, y) = o(d(x, y)^{-q})$. Note that $\|\psi\|_{\infty} \geq \|\psi\|_{C_q(\mathcal  D^c)}$ for $q> 0$, so $C(\mathcal D^c) \subset C_q(\mathcal  D^c)$ in this case.

\begin{lemma}\label{lem:P(2)tisstrocononCp}
Fix $q \geq 0$. Then $(P_t^{(2)})_{t\geq 0}$ forms a $C_0$-semigroup of bounded linear operators on $C_{q}(\mathcal  D^c)$.
\end{lemma}

\begin{proof}
First let us show that $P_t^{(2)}$ is bounded on $C_{q}(\mathcal  D^c)$ for any $t\geq 0$. To this end note that for any $(x, y)\in \mathcal D^c$ (see Subsection \ref{subsec:ergoftwopmot}  for the definition of $(x_t)_{t\geq 0}$ and $(y_t)_{t\geq 0}$),
\begin{multline*}
|P_t^{(2)} \psi(x, y)|= |\mathbb E_{(x, y)}\psi(x_t, y_t)|\\
 \leq \mathbb E_{(x, y)}d(x_t, y_t)^{-q}\frac{| \psi(x_t, y_t)|}{d(x_t, y_t)^{-q}} \leq \mathbb E_{(x, y)} d(x_t, y_t)^{-q} \|\psi\|_{C_{q}(\mathcal  D^c)},
\end{multline*}
therefore it is sufficient to show that
\begin{equation}\label{eq:vsoEdxtyt-pCtdxy-p}
\mathbb E_{(x, y)} d(x_t, y_t)^{-q}<C_t d(x, y)^{-q},
\end{equation}
for some constant $C_t>0$ and for any $t\geq 0$, which follows from \eqref{eq:phikappaconeinxx'}. Further, assume that for any $(x, y)\in \mathcal D^c$ we have $|\psi(x, y)|\leq \eta(d(x, y)) d(x, y)^{-q}$ for some fixed bounded nondecreasing $\eta:\mathbb R_+ \to \mathbb R_+$ with $\eta(s)\to 0$ as $s\to 0$ (such $\eta$ exists as $\psi\in C_q(\mathcal  D^c)$). Then for any $(x, y)\in \mathcal D^c$
\begin{multline*}
|P_t^{(2)} \psi(x, y)|d(x, y)^{q} \leq \mathbb E_{(x, y)}\frac{d(x_t, y_t)^{-q}}{d(x, y)^{-q}}\frac{| \psi(x_t, y_t)|}{d(x_t, y_t)^{-q}}\\
 \leq \mathbb E_{(x, y)} \frac{d(x_t, y_t)^{-q}}{d(x, y)^{-q}} \eta(d(x_t, y_t))\leq \Bigl(\mathbb E_{(x, y)} \frac{d(x_t, y_t)^{-2q}}{d(x, y)^{-2q}} \Bigr)^{1/2}\sqrt{ \mathbb E_{(x, y)}  \eta^2(d(x_t, y_t))},
\end{multline*}
where $\mathbb E_{(x, y)} \tfrac{d(x_t, y_t)^{-2q}}{d(x, y)^{-2q}}$ is bounded uniformly in $(x, y)\in \mathcal D^c$ by  \eqref{eq:phikappaconeinxx'} and where 
\begin{equation}\label{eq:eta2dxtytqwdcwwq}
\mathbb E_{(x, y)}  \eta^2(d(x_t, y_t)) = o(1)\;\;\; \text{as}\;\;\;d(x, y)\to 0
\end{equation}
by the fact that $d(x_t, y_t) \leq \widetilde C d(x, y)^{\frac 34}$ for some integrable $\widetilde C:\Omega \to \mathbb R$ thanks to \eqref{eq:dphikappatxphikaa't'x'bqd}, so \eqref{eq:eta2dxtytqwdcwwq} follows from the fact that $\eta$ is bounded and converges to zero at zero. Therefore $P_t^{(2)} \psi \in C_q(\mathcal  D^c)$.

Now let us show that for any $\psi\in C_q(\mathcal  D^c)$ we have that $P_t^{(2)} \psi \to \psi$ in $C_q(\mathcal  D^c)$ as $t\to 0$. By the uniform boundedness of operators $(P_t^{(2)})_{t\in[0,1]}$ (due to the universality of constant $C_t$ from \eqref{eq:vsoEdxtyt-pCtdxy-p} for $t$ small, see \eqref{eq:phikappaconeinxx'}) we may assume that $\psi$ is from a dense subset, namely $\psi\in C^{1}(\mathcal D^c)\cap C_q(\mathcal D^{c})$ having a compact domain. Fix $t\geq 0$ and $(x, y)\in \mathcal D^{c}$. Then as $|D\psi|$ is uniformly we have that
\begin{equation}\label{eq:Pt(2)psi-psid)xy(-q}
\begin{split}
\frac{|P_t^{(2)} \psi(x, y) - \psi(x, y)|}{d(x, y)^{-q}} &= \frac{|\mathbb E_{(x, y)}\psi(x_t, y_t)- \psi(x, y)|}{d(x, y)^{-q}} \\
&\lesssim d(x, y)^{q}\mathbb E_{(x, y)} \|D \psi\|_{\infty} \bigl(d(x, x_t) + d(y, y_t)\bigr)
\end{split}
\end{equation}
which vanishes uniformly in $(x, y)\in \mathcal D^c$ as $t\to 0$ by \eqref{eq:dphikappatxphikaa't'x'bqd} and the fact that $\mathcal M$ is a compact.
\end{proof}

The following proposition connects $\Lambda(p)$ and $\widetilde V_p$ with $(P^{(2)}_t)_{t\geq 0}$.

 \begin{proposition}\label{prop:exiofVtildeoimc}
 Let $L^{(2)}$ be the generator of $(P^{(2)}_t)_{t\geq 0}$ on $C_p(\mathcal  D^c)$. Then there exists $p^*>0$ such that for any $p\in (0,p^*]$ there exist a $C^{2+\beta/2}_{loc}$-function $\widetilde V_p:\mathcal D^c \to [1, \infty)$ and constants  $c_p, K>0$ such that $\widetilde V_p$ is in the domain of $L^{(2)}$,
\begin{equation}\label{eq:widetildVpeqaimxoimxq}
\frac 1K d(x, y)^{-p}\leq\widetilde  V_p(x, y)\leq K d(x, y)^{-p},\;\;\; (x, y)\in \mathcal D^c,
\end{equation}
and
\begin{equation}\label{eq:driftcondinpropow}
L^{(2)} \widetilde V_p(x, y) \leq -\Lambda(p)\widetilde V_p(x, y) + c_p,\;\;\; (x, y)\in \mathcal D^c,
\end{equation}
where the latter expression is considered in $C_{p+3}(\mathcal  D^c)$.
 \end{proposition}

Let us start proving Proposition \ref{prop:exiofVtildeoimc} by constructing the Lyapunov function $\widetilde  V_p$. 
 Set $p^*:= p_*\wedge \tfrac{\beta}{4}$ with $p_*$ being from Proposition \ref{prop:maintextpsipexinsandinrwc}. Fix any $p\in (0, p_*]$. Let $\psi_p$ be as in  Proposition \ref{prop:maintextpsipexinsandinrwc}. Set for any $(x, w)\in T \mathcal M $
\begin{equation}\label{eq:defoffpkappasd}
 f_{p}(x, w) := 
\begin{cases}
|w|^{-p} \psi_{p} (x, w/|w|),\;\;\; &x\in\mathcal M, \;\; w\in T_x\mathcal M\setminus \{0\},\\
0,\;\;\;&x\in \mathcal M, \;\;w=0.
\end{cases}
\end{equation}
Let $\Phi$ and $\delta_0$ be as in Section \ref{sec:prelim}, let
\begin{equation}\label{eq:defofhpcmiw}
  h_{p}(x, w) =  f_{p}(x, w) \chi(|w|),\;\; x\in \mathcal M, \;\; w\in T_x\mathcal M,
\end{equation}
where $\chi:\mathbb R_+ \to [0,1]$ is a nonincreasing $C^{\infty}$ function with $\chi(t) = 1$ for $0\leq t\leq \delta_0/4$ and $\chi(t)=0$ for $t\geq \delta_0/2$, and let
\begin{equation}\label{eq:defofVkapapfasdcoi}
 \widetilde V_p(x, y) := 
 \begin{cases}
 h_{p}(\Phi^{-1}(x, y)),\;\;\; &(x,y)\in D^c_{\delta_0},\\
 0,\;\;\; &(x, y)\in \mathcal M \times \mathcal M \setminus D^c_{\delta_0},
 \end{cases}
\end{equation}
where $\Phi^{-1}$ is defined on $\mathcal D^c_{\delta_0}$ as in Remark \ref{rem:PhiPhi-1onDcdleta}.

In order to obtain Proposition \ref{prop:exiofVtildeoimc} we need to prove that $ \widetilde V_p$ is in the domain of  $ L^{(2)}$ and we need to show how does $L^{(2)} \widetilde V_p(x, y)$ look like for $x$ and $y$ close enough. While the first fact follows directly form the regularity of $\widetilde V_p$ and its behaviour near the diagonal (see the proof of Proposition \ref{prop:exiofVtildeoimc} below), in order to show what $L^{(2)} \widetilde V_p$ is close to the diagonal we will need a linearized version of $P^{(2)}_t$, namely $TP_t$, defined for any $\psi\in C_0^{\infty}(T\mathcal M)$ and $t\geq 0$ by
 \[
  TP_t\psi(x, w) := \mathbb E_{(x, w)} \psi(x_t, D \phi_t(x)w),\;\;\; x\in \mathcal M,\;\; w\in T_x\mathcal M.
 \]
 Then for the generator $T\mathcal L$ of $(TP_t)_{t\geq 0}$ and for any $(x, y)\in \mathcal D^c$ with $d(x, y)<\delta_0$ and $w=w(x, y)$ one would have 
 \begin{equation}\label{eq:L(2)tildeVp-TLhpfwe}
 \begin{split}
  L^{(2)}{\widetilde V_p} (x, y) - T\mathcal L& h_{p}(x, w) = \lim_{t\to 0} \frac{ \mathbb E_{(x, y)} {\widetilde V_p}(x_t, y_t) -{\widetilde V_p}(x, y) }{t} \\
  &\quad\quad\quad\quad- \lim_{t\to 0} \frac{\mathbb E_{(x, w)}h_{p} (x_t, D \phi_t(x)w) - h_{p}(x, w)}{t}\\
  &= \lim_{t\to 0} \frac{  \mathbb E_{(x, y)}{\widetilde V_p}(x_t, y_t) -\mathbb E_{(x, w)}h_{p} (x_t, D \phi_t(x)w) }{t}\\
  &= \lim_{t\to 0}\mathbb E_{(x, w)} \frac{ h_{p} (x_t, w_t)\mathbf 1_{d(x_t, y_t)<\delta_0} -h_{p} (x_t, D \phi_t(x)w) }{t},\\
  &\stackrel{(*)}= \lim_{t\to 0}\mathbb E_{(x, w)} \frac{ h_{p} (x_t, w_t) -h_{p} (x_t, D \phi_t(x)w) }{t},
\end{split}
\end{equation}
 where $w=w(x, y)\in T_x\mathcal M$ and $w_t = w(x_t, y_t)\in T_{x_t}\mathcal M$ are defined by Remark \ref{rem:PhiPhi-1onDcdleta}, and where $(*)$ holds by \eqref{eq:defofhpcmiw}.
Therefore we only need to calculate $T\mathcal L h_{p}(x, w)$ (which equals $-\Lambda(p) h_p(x, w)$, see Proposition \ref{prop:Tlkappahpkappa=das} below) and the right-hand side of \eqref{eq:L(2)tildeVp-TLhpfwe} for $w$ with $|w|$ sufficiently small.

We will start with the following technical lemma concerning derivatives of $h_p$. Let $\beta$ be as in the condition {\textnormal {\bf (C)}}.

\begin{lemma}\label{lem:estimforhpkappaandderivac}
$h_p$ is $C^{2+\beta/2}_{loc}(T\mathcal M)$. Moreover, there exists $C>0$ such that for any $x\in \mathcal M$, for any $w\in T_x \mathcal M$ small enough, and for any $p\in [-p_*, p_*]$ we have that 
 $$
 |h_{p}(x,w)|, |D_x h_{p}(x,w)|, \|D^2_{x,x} h_{p}(x,w)\| \leq C|w|^{-p},
 $$
 $$
 |D_w h_{p}(x,w)|, \|D^2_{x,w} h_{p}(x,w)\| \leq C|w|^{-p-1},
 $$
 $$ 
 \|D^2_{w,w} h_{p}(x,w)\| \leq C |w|^{-p-2}.
 $$
\end{lemma}

\begin{proof} 
The fact that $h_p$ is locally $C^{2+\beta/2}$ follows directly from the definition of $h_p$ and Proposition \ref{prop:maintextpsipexinsandinrwc}.

Let us turn to the desired estimates. We will only show the case of $D_w$, the rest can be done similarly. As $\chi(|w|) = 1$ for $w$ small,
\begin{equation}\label{eq:expressionfornablawfpkappaSigmadsad}
\begin{split}
D_{ w} h_{p} (x, w) &= D_{ w} ( |w|^{-p}\psi_{p} (x, {w/|w|}) )\\
&= D_{ w} |w|^{-p}  \psi_{p} (x, {w/|w|}) + |w|^{-p}D_{ w}\psi_{p} (x, {w/|w|})  \\
& \stackrel{(*)}= -p |w|^{-p-2} \langle w, \cdot \rangle \psi_{p} (x, {w/|w|})\\
&\quad \quad \quad \quad + |w|^{-p} \Bigl\langle D_{ w/|w|} \psi_{p} (x, {w/|w|}) , \frac{ \cdot |w|^2 - w \langle w, \cdot \rangle}{|w|^3} \Bigr\rangle,
\end{split}
\end{equation}
where for $(*)$ we used the fact that $D_{w}(w/|w|) = \frac{|w|^2 - w \langle w, \cdot \rangle}{|w|^3}$. It remains to notice that by Subsection \ref{sec:LyaexmomLyafun} the families of functions $
(\psi_{p} )_{p\in [-p_*, p_*]}$ and $(D_{w/|w|}\psi_{p} )_{p\in [-p_*, p_*]}$ are uniformly bounded.
\end{proof}

For each $q\in \mathbb R$ let 
 \begin{align*}
 C_q(T\mathcal M) = \Big\{f\in C((T_x\mathcal M\setminus \{0\})_{x\in \mathcal M}):\lim_{|w|\to 0}\sup_{x\in\mathcal M}&|f|(x, w) |w|^q =0 \\
 \text{ and} \;\;&\lim_{|w|\to \infty} |f|(x, w)= 0 \Big\},
 \end{align*}
which is endowed with the norm
 $$
 \|f\|_{C_q(T\mathcal M)}:= \sup_{x\in \mathcal M, w\in T_x\mathcal M\setminus \{0\}}|f(x, w)| (|w|^q\wedge1).
 $$ 
For the proof of Theorem \ref{thm:actionL(2)kappaonVkappa} we will need the following proposition.

\begin{proposition}\label{prop:Tlkappahpkappa=das}
$(TP_t)_{t\geq 0}$ generates a $C_0$-semigroup on  $C_{p+3}(T\mathcal M)$ and
its generator $T\mathcal L$ of $(TP_t)_{t\geq 0}$ satisfies in $C_{p+3}(T\mathcal M)$
\[
T\mathcal L h_{p} = -\Lambda(p)\chi(|w|) f_{p}(x, w),\;\;\; x\in \mathcal M,\;\;\; w\in T_x \mathcal M,\;\;\; |w|\leq \delta_0/8. 
\]
\end{proposition}

For the proof we will need the formula for $D\phi_t w$ which can be deduced from \eqref{eq:flowforkappa=0dsa} and has the following form.
\begin{equation}\label{eq:dsrqdoreDphikappa}
\begin{split}
D \phi_tw &=w + \sum_{k\geq 1} \int_0^{t} \langle D \sigma_k(x_s), D\phi_s w \rangle \ud W^k_t \\
&\quad\quad\quad\quad\quad\quad\quad\quad\quad\quad+ \frac {1}{2}[\langle D \sigma_k(x_{\cdot}), D\phi_{\cdot} w \rangle, W^k]_t,\;\;\; t\geq 0.
\end{split}
\end{equation}

\begin{proof}[Proof of Proposition \ref{prop:Tlkappahpkappa=das}]
 First note that
 similarly to Lemma \ref{lem:P(2)tisstrocononCp} one has that $TP_t$ generates a $C_0$-semigroup on  $C_{p+3}(T\mathcal M)$, where it only remains to show that for any $f\in C^{\infty}(T\mathcal M)$ with a compact support (i.p.\ there exists a constant $\delta>0$ such that $f(x, w)=0$ for any $(x, w)\in T\mathcal M$ with $|w|>\delta$) one has that 
 {
 $$
 \lim_{|w|\to \infty}\sup_{x\in\mathcal M}|TP_tf(x, w)|= 0.
 $$ 
 To this end fix $x\in \mathcal M$ any $w\in T_x \mathcal M$ with $|w|>2\delta$. Then by Markov's inequality 
 \begin{align*}
 |TP_tf(x, w)|  &= | \mathbb E_{(x, w)} f(x_t, D \phi_t(x)w)| \\
 &\leq  \|f\|_{\infty} \mathbb P(\|D \phi_t(x)\|\leq {\delta}/{|w|})\\
 &=  \|f\|_{\infty} \mathbb P(\|D \phi_t(x)\|^{-1}\geq |w|/{\delta})\\
 &\leq |w|^{-1}\delta \|f\|_{\infty} \mathbb E\|D \phi_t(x)\|^{-1},
 \end{align*}
 }
 which vanishes uniformly as $|w|\to\infty$ by Lemma \ref{lem:phitkat64wppacontinxandkapps}

 Next, by \eqref{eq:defoffpkappasd} we have that $f_p\in C_{p+3}(T \mathcal M)$ and that for any $x\in  \mathcal M$ and $w\in T_x \mathcal M\setminus \{0\}$ by the definition of $\psi_p$ (see Subsection \ref{sec:LyaexmomLyafun})
 \begin{align*}
   TP_t f_{p}(x, w) &=  \mathbb E_{(x, w)} f_{p}(x_t, D \phi_tw)=\mathbb E_{(x, w)} |D \phi_tw|^{-p} \psi_{p}\bigl(x_t, \tfrac{D \phi_tw}{|D \phi_tw|}\bigr)\\
   &=|w|^{-p}\mathbb E_{(x, v)}|D \phi_tv|^{-p}  \psi_{p}(x_t, v_t) = |w|^{-p}\hat P^p_{t}\psi_p\\
  & = |w|^{-p} e^{-\Lambda(p)t} \psi_{p}(x, v) = e^{-\Lambda(p)t} f_{p}(x, w),
\end{align*}
where $v:= w/|w|$ and $v_t:= w_t/|w_t|$.
Hence $f_{p}$ is in the domain of $T\mathcal L$ and $T\mathcal Lf_{p} = -\Lambda(p)f_{p}$.
Let us show that $h_p$ is in the domain of $T\mathcal L$. For each $n\geq 1$ let $\eta_n:\mathbb R_+ \to \mathbb R_+$ be nondecreasing such that $\|\eta_n'\|_{\infty} \leq Cn$ and $\|\eta_n''\|_{\infty} \leq Cn^2$ for some fixed constant $C>0$ and such that $\eta_n|_{[0, 1/n]}=0$ and $\eta_n|_{[2/n, \infty)}=1$. Set 
$$
h^n(x, w) := h_p(x, w)\eta_n(|w|),\;\;\; x\in \mathcal M,\;\;\; w\in T_x \mathcal M,\;\;\; n\geq 1.
$$
Then each of $h^n$ is $C^2(T\mathcal M)$ with a compact support  by \eqref{eq:defoffpkappasd}, \eqref{eq:defofhpcmiw}, and Proposition \ref{prop:maintextpsipexinsandinrwc}. Hence, $(h^n)_{n\geq 1}$ are in the domain of $T\mathcal L$. Moreover, we have $h^n \to h_p$, and arguing analogously as for \eqref{eq:L(2)Vp-L(2)fnvaoimceaCp+3} below using Lemma \ref{lem:estimforhpkappaandderivac}, $T\mathcal L h^n \to T\mathcal L h_p$ in $C_{p+3}(T\mathcal M)$. Since $T$ is a closed operator, this implies that  $h_p$ is in the domain of $T\mathcal L$ as well.

Therefore the following holds in $C_{p+3}(T \mathcal M)$
\begin{align*}
TP_t h_{p} (x, w) &= TP_t f_{p} \chi(|\cdot|) (x, w)= TP_t f_{p} (\chi(|\cdot|) - \chi(|w|)) (x, w) + \chi(|w|) TP_t f_{p} \\
&=  TP_t f_{p} (\chi(|\cdot|) - \chi(|w|)) (x, w)\\
&\quad\quad\quad+\chi(|w|) e^{-\Lambda(p)t} f_{p}(x, w),\; x\in  \mathcal M,\; w\in T_x  \mathcal M.
\end{align*}
Thus it is sufficient to show that 
$$
\lim_{t\to 0} \frac{1}{t}TP_t f_{p} (\chi(|\cdot|) - \chi(|w|)) (x, w) = 0
$$
for any $x\in \mathcal M$ and $w\in T_x\mathcal M$ with $|w|<\delta_0/8$. To this end notice that
\begin{equation}\label{eq:1/tTpka[ppatd(xi||-xi|w|)}
 \begin{split}
  & \frac{1}{t}TP_t f_{p} (\chi(|\cdot|) - \chi(|w|)) (x, w) \\
 &\quad\quad\quad\quad = \frac{1}{t}\mathbb E_{(x, w)} |D\phi_tw|^{-p} \psi_{p}\Bigl(x_t, \frac{ D\phi_tw}{ |D\phi_tw|}\Bigr)\bigl(\chi(|D\phi_tw|) - \chi(|w|)\bigr),
 \end{split}
\end{equation}
which vanishes uniformly in  $w\in T_x\mathcal M$ with $|w|<\delta_0/8$ as in this case
\begin{align*}
 \frac{1}{t}\mathbb E_{(x, w)} &|D\phi_t(x)w|^{-p} \Bigl|\psi_{p}\Bigl(x_t, \frac{ D\phi_t(x)w}{ |D\phi_tw|}\Bigr)\Bigr|\bigl|\chi(|D\phi_t(x)w|) - \chi(|w|)\bigr|\\
 &\leq t^{-1} \|\psi_p\|_{\infty} |w|^{-p} \mathbb E_x \bigl\|(D\phi_t(x))^{-1}\bigr\|^p\bigl|\chi(|D\phi_t(x)w|) - \chi(|w|)\bigr|\\
 &\lesssim_{\chi} t^{-1} \|\psi_p\|_{\infty} |w|^{-p} \mathbb E_x \bigl\|(D\phi_t(x))^{-1}\bigr\|^p\mathbf 1_{\|D\phi_t(x)\| \geq \delta_0 |w|^{-1}/4}\\
 &\leq   \|\psi_p\|_{\infty} \Bigl(\mathbb E_x \bigl\|(D\phi_t(x))^{-1}\bigr\|^{2p}\Bigr)^{1/2}\bigl( t^{-2}  |w|^{-2p}\mathbb P(\|D\phi_t(x)\| \geq \delta_0 |w|^{-1}/4)\bigr)^{1/2},
\end{align*}
so it suffices to notice that $\mathbb E_x \|(D\phi_t(x))^{-1}\|^{2p}$ is uniformly bounded for $t$ small by Lemma \ref{lem:phitkat64wppacontinxandkapps} and that as $\delta_0 |w|^{-1}/4>2$ (so $\delta_0 |w|^{-1}/4- 1 \geq  \delta_0 |w|^{-1}/8$) thanks to \eqref{eq:EcowemkDphikappatx-Dphikappa't'x'}
\begin{align*}
& t^{-2}  |w|^{-2p}\mathbb P(\|D\phi_t(x)\| \geq \delta_0 |w|^{-1}/4) \\
 &\quad\quad\quad\quad\quad\quad\quad\quad\quad\leq  t^{-2}  |w|^{-2p}\mathbb P(\|D\phi_t(x)-I_{T_x\mathcal M}\| \geq \delta_0 |w|^{-1}/8)\\
 &\quad\quad\quad\quad\quad\quad\quad\quad\quad \lesssim  t^{-2}  |w|^{4+2p-2p} \mathbb E \|D\phi_t(x)-I_{T_x\mathcal M}\|^{4+2p} \lesssim t^{p}  |w|^{4}, 
\end{align*}
which vanishes as $t\to 0$ uniformly in $|w|<\delta_0/8$.
\end{proof}

Finally, let us prove the desired statement. For any $(x, y)\in \mathcal D^c$ with $d(x, y)<\delta_0$ and $w=w(x, y)\in T_x\mathcal M$ defined by Remark \ref{rem:PhiPhi-1onDcdleta} we set 
{
\begin{equation}\label{eq:Sigmaforexpboundsa}
\begin{split}
\Sigma &:= \frac{1}{2}\sum_{k\geq 1} \left \langle D^2 w(x, y), \left (\binom{\sigma_k(x)}{\sigma_k(y)},\binom{\sigma_k(x)}{\sigma_k(y)}\right)\right \rangle \\
&\quad\quad\quad\quad+ \left \langle D w(x, y),  \binom{\langle D\sigma_k(x), \sigma_k(x)\rangle}{\langle D\sigma_k(y), \sigma_k(y)\rangle}  \right \rangle\\
&\quad\quad\quad\quad-\langle D^2\sigma_k(x), (\sigma_k(x), w)\rangle  - \bigl \langle D\sigma_k(x),\langle D\sigma_k(x), w\rangle\bigr\rangle,
 \end{split}
\end{equation}
\begin{equation}\label{eq:Sigmaprikqforexpboundsa}
\begin{split}
 \Sigma' &:=\sum_{k\geq 1}\left \langle Dw(x, y), \binom{\sigma_k(x)}{ \sigma_k(y)}\right \rangle \otimes \left \langle Dw(x, y), \binom{\sigma_k(x)}{ \sigma_k(y)}\right \rangle\\
&\quad\quad\quad\quad- \langle D \sigma_k(x), w \rangle\otimes\langle D \sigma_k(x), w \rangle,
 \end{split}
\end{equation}
\begin{equation}\label{eq:formforSigmad'''das}
 \Sigma '' := \sum_{k\geq 1}\sigma_k(x)\otimes \left (\left \langle Dw(x, y), \binom{\sigma_k(x)}{ \sigma_k(y)}\right \rangle-\langle D \sigma_k(x),w \rangle\right).
\end{equation}

Notice that $\Sigma$, $\Sigma'$, and $\Sigma''$ are well defined and continuous in $(x, y)\in \mathcal D^c_{\delta_0}$. This follows from the fact that $(x, y)\mapsto w(x, y)$ is $C^{\infty}$ on $\mathcal D^c_{\delta_0}$ (see Remark \ref{rem:PhiPhi-1onDcdleta}) and from the fact that by \eqref{eq:nescondonsigmaintermsofsums} the sequences 
$$
(  \sigma_k(x), \sigma_k(y) )_{k\geq 1},\;\;( \langle D\sigma_k(x), \sigma_k(x) \rangle)_{k\geq 1},\;\;(  D\sigma_k(x),D \sigma_k(y) )_{k\geq 1},
$$
$$
(\langle D^2\sigma_k(x),(\sigma_k(x), \cdot) \rangle)_{k\geq 1}, \;\; \text{and}\;\; \left(\left\langle D\sigma_k(x), \langle D\sigma_k(x), \cdot\rangle \right\rangle\right)_{k\geq 1}
$$ 
are summable for any $x,y\in \mathcal M$ and the corresponding sums are continuous in $x,y\in \mathcal M$.
}

\begin{theorem}\label{thm:actionL(2)kappaonVkappa}
 $\widetilde V_p$ is in the domain of $L^{(2)}$ as of an unbounded operator on $C_{p+3}(\mathcal D^c)$ and for any $(x, y) \in \mathcal D^c$ with $d(x, y)<\delta_0/8$ 
 \begin{equation*}\label{eq:mathcalL(2)kappaVkappaeqwfp}
 \begin{split}
    &L^{(2)} \widetilde V_p (x, y)= -\Lambda(p)\chi(|w(x, y)|) f_{p}(x, w(x, y)) + D_{ w} h_{p} (x, w(x, y)) \Sigma \\
    &\quad\quad\quad\quad\quad+ D^2_{w,w}h_{p} (x, w(x, y)) \Sigma' + D^2_{x,w}h_{p} (x, w(x, y)) \Sigma'',
 \end{split}
 \end{equation*}
 where $\Sigma$, $\Sigma'$, and $\Sigma ''$ are defined by \eqref{eq:Sigmaforexpboundsa}, \eqref{eq:Sigmaprikqforexpboundsa}, and \eqref{eq:formforSigmad'''das}.
\end{theorem}

\begin{proof}
By Proposition \ref{prop:Tlkappahpkappa=das} it is sufficient to show that $\widetilde V_p$ is in the domain of $L^{(2)}$ and that
\begin{equation}\label{eq:L(2)kappa-Tlkappacsfewvp}
 \begin{split}
   &L^{(2)} \widetilde V_p(x, y) - T\mathcal L h_{p}(x, w(x, y))= D_{ w(x, y)} h_{p} (x, w(x, y)) \Sigma\\
 &\quad\quad\quad\quad+ D^2_{w}h_{p} (x, w(x, y)) \Sigma' + D^2_{x,w}h_{p} (x, w(x, y)) \Sigma'',
 \end{split}
\end{equation}
 holds true pointwise for $d(x, y)<\delta_0/8$.

Let us first show that $\widetilde V_p$ is in the domain of $L^{(2)}$. It is enough to prove that $L^{(2)}\widetilde V_p \in C_{p+3}(\mathcal D^c)$ and that there exists a sequence $(f_n)_{n\geq 1}$ of $C^2$-functions on $\mathcal D^c$ with compact supports such that both 
\begin{equation}\label{eq:Vp-fnbvainCp+3}
\|\widetilde V_p - f_n\|_{C_{p+3}(\mathcal D^c)}\to 0,\;\;\; n\to \infty
\end{equation}
 and
 \begin{equation}\label{eq:L(2)Vp-L(2)fnvaoimceaCp+3}
 \|L^{(2)}\widetilde V_p - L^{(2)}f_n\|_{C_{p+3}(\mathcal D^c)}\to 0,\;\;\;n\to \infty
 \end{equation}
 hold true.

Let us start by noticing that as $(x, y)\mapsto \Phi^{-1}(x, y) = (x, (w(x, y))$ is a $C^{\infty}$ function on $\mathcal D^c_{\delta_0}$ (see Remark \ref{rem:PhiPhi-1onDcdleta}), by \eqref{eq:defofVkapapfasdcoi} and Lemma \ref{lem:estimforhpkappaandderivac} for any $(x,y)\in \mathcal D^c_{\delta_0}$ we have that
\begin{equation}\label{eq:D2DVpleqD2Dhpcowem}
\begin{split}
&\|D^2\widetilde V_p(x, y)\| + \|D\widetilde V_p(x, y)\| + |\widetilde V_p(x, y)| \\
&\quad\quad\quad\quad\quad\lesssim \|D^2h_p(x, w(x, y))\| + \|Dh_p(x, w(x, y))\| + |h_p(x, w(x, y))|\\
&\quad\quad\quad\quad\quad\lesssim d(x, y)^{-p-2},
\end{split}
\end{equation}
so $L^{(2)}\widetilde V_p \in C_{p+3}(\mathcal D^c)$ as $\widetilde V_p \equiv 0$ outside $\mathcal D^c_{\delta_0}$.
Next, {set $(\eta_n)_{n\geq 1}$ to be as in the proof of Proposition \ref{prop:Tlkappahpkappa=das}}. For all $(x, y)\in \mathcal D^c$ set $f_n(x, y):= \widetilde V_p(x, y)\eta_n(d(x, y))$. Let us show \eqref{eq:Vp-fnbvainCp+3} and \eqref{eq:L(2)Vp-L(2)fnvaoimceaCp+3}. \eqref{eq:Vp-fnbvainCp+3} follows directly from the inequality 
\begin{align*}
|\widetilde V_p(x, y) - f_n(x, y)| &\leq |\widetilde V_p(x, y)|\mathbf 1_{[0, 2/n]}(d(x, y))\\
& \lesssim \mathbf 1_{[0, 2/n]}(d(x, y)) d(x, y)^{-p}\leq \frac 2n d(x, y)^{-p-1},\;\;\; (x, y)\in \mathcal D^c.
\end{align*}

In order to show \eqref{eq:L(2)Vp-L(2)fnvaoimceaCp+3} it is enough to notice that for any $(x, y)\in \mathcal D^c$
\begin{align*}
|L^{(2)}(\widetilde V_p&-f_n)(x, y)|\lesssim \|D^2(\widetilde V_p-f_n)(x, y)\| +  \|D(\widetilde V_p-f_n)(x, y)\| +  |(\widetilde V_p-f_n)(x, y)|\\
&\leq \bigl( \|D^2\widetilde V_p(x, y)\||1-\eta_n(d(x, y))| \\
&\quad+ \|D\widetilde V_p(x, y)\|(|1-\eta_n(d(x, y))| + \|D\eta_n(d(x, y))\|) \\
&\quad+ |\widetilde V_p(x, y)|(|1-\eta_n(d(x, y))| + \|D\eta_n(d(x, y))\| + \|D^2\eta_n(d(x, y))\|)\bigr)\\
&\stackrel{(*)}\lesssim \mathbf 1_{[0, 2/n]}(d(x, y))\left( d(x, y)^{-p-2} + n d(x, y)^{-p-1} + {n^2}d(x, y)^{-p} \right)\\
&\lesssim \mathbf 1_{[0, 2/n]}(d(x, y))d(x, y)^{-p-2} \leq \frac{1}{\sqrt n} d(x, y)^{-p-\tfrac {5}{2}},
\end{align*}
where $(*)$ follows from Lemma \ref{lem:estimforhpkappaandderivac}, \eqref{eq:D2DVpleqD2Dhpcowem}, upper bounds for $\eta$, $\eta_n'$, and $|\eta_n''|$, and the fact that both $(x, y)\mapsto d(x, y)$ and $(x, y)\mapsto w(x, y)$ are $C^{\infty}$ on $\mathcal D^c_{\delta_0}$ and $V_p \equiv 0$ outside $\mathcal D^c_{\delta_0}$. The latter inequality implies \eqref{eq:L(2)Vp-L(2)fnvaoimceaCp+3}.

Now we need to show \eqref{eq:L(2)kappa-Tlkappacsfewvp} pointwise. Fix $(x, y)\in \mathcal D^c$ such that $d(x, y)\leq \delta_0/8$ and set $w := w(x, y)\in T_x\mathcal M$ as in Remark \ref{rem:PhiPhi-1onDcdleta}. Then  \eqref{eq:L(2)tildeVp-TLhpfwe} holds true, so we only need to compute $ \lim_{t\to 0}t^{-1}\mathbb E_{(x, w)} (h_{p} (x_t, w_t) -h_{p} (x_t, D \phi_t(x)w))$, where for each $t\geq 0$ we set $w_t = w(x_t, y_t)\in T_{x_t}\mathcal M$ if $d(x_t, y_t)<\delta_0$ and $w_t = 0$ if $d(x_t, y_t)\geq\delta_0$.
Let a stopping time $\tau$ be defined by 
$$
\tau:= \inf\{t\geq 0:d(x_t, x) + d(y_t, y) \geq \eps\},
$$
where $\eps := {|w|}/{2}$, so that $t\mapsto w_t$ is continuous a.s.\ in $t\in [0, \tau]$ and keeps distance $|w|/2$ from zero. Let $\Phi$ be defined by \eqref{eq:expmapPhi}. Then $t\mapsto (x_t, y_t)$ satisfies the following SDE due to the chain rule (see e.g.\ \cite[pp.\ 339--342]{Kal})
\begin{equation}\label{eq:dxtwtequationicweo}
\begin{split}
\dd \binom{x_t}{ w_t} = \dd \Phi^{-1}\binom{x_t}{y_t}  &=  D\Phi^{-1}(x_t, y_t)\circ \dd \binom{x_t}{y_t}   \\
&= \sum_{k\geq 1} \left\langle (D\Phi)^{-1}(x_t, w_t), \binom{\sigma_k}{\sigma_k}\bigl(\Phi(x_t, w_t)\bigr)  \right\rangle\circ \dd W^k_t.
\end{split}
\end{equation}
Therefore, by It\^o's formula \cite[Theorem 26.7]{Kal}
(recall that by Proposition \ref{prop:Lambda(p)defandpsiprequxwqxwq} we have that $\psi_{p}\in C^{2}(S\mathcal M)$, and  hence $(x, v)\mapsto h_{p}(x, v)$ is $C^2$ on $\{(x, v)\in T\mathcal M:x\in \mathcal M, v\in T_x\mathcal M, |w|/2\leq |v|\leq 2|w|\}$) for any $t\geq 0$ we have that thanks to \eqref{eq:flowforkappa=0dsa}, \eqref{eq:dsrqdoreDphikappa}, and \eqref{eq:dxtwtequationicweo}
\begin{equation}\label{eq:eqforABCScawP,ce}
\begin{split}
h_{p} (x_{t\wedge\tau}, D \phi_{t\wedge\tau}w) &= \sum_{k\geq 1}\int_0^{t\wedge\tau} A^k_s \ud W^k_s + \int_0^{t\wedge\tau} B_s \ud s,\\
h_{p} (x_{t\wedge\tau}, w_{t\wedge\tau}) &= \sum_{k\geq 1}\int_0^{t\wedge\tau} C^k_s \ud W^k_s + \int_0^{t\wedge\tau} D_s \ud s,
\end{split}
\end{equation}
where $A^k, B, C^k,D:\mathbb R_+ \times \Omega \to \mathbb R$ are continuous processes such that the stochastic integrals $ \sum_{k\geq 1}\int_0^{t\wedge\tau} A^k_s \ud W^k_s$ and $ \sum_{k\geq 1}\int_0^{t\wedge\tau} C^k_s \ud W^k_s$ exist for any $t\geq 0$ thanks to 
\cite{DPZ92,MP80} as by the condition {\textnormal {\bf (C)}}, \eqref{eq:vsoEdxtyt-pCtdxy-p}, the fact that $\Phi(x_t, w_t) = (x_t, y_t)$ and that $\Phi$ is a $C^{\infty}$-diffeomorphim, and Corollary \ref{cor:supofDphiontsciintegrs} we have that
\begin{align*}
 &\mathbb E\sum_{k\geq 1}\int_0^{t\wedge\tau} |A^k_s|^2 +| C^k_s|^2 \ud s\\
 &\quad\leq \mathbb E  \sum_{k\geq 1}\int_0^{t\wedge\tau}\big( \| D_x h_{p} (x_s, w_s)\|^2 + \| D_w h_{p} (x_s, w_s)\|^2\big) \\
 &\quad\quad \quad\quad\quad\quad \quad\quad \quad\quad \quad\quad \quad\quad \quad\quad  \cdot\big\|D\Phi^{-1}(x_s, w_s)\big\|^2\big(|\sigma_k(x_s)|^2 + |\sigma_k(y_s)|^2\big) \\
 &\quad\quad+ \big | \langle D_x h_{p} (x_s, D \phi_sw) ,\sigma(x_s)\rangle\big |^2 + \big | \langle D_w  h_{p} (x_s,  D \phi_sw), \langle D \sigma_k(x_s), D\phi_s w \rangle \rangle\big |^2 \ud  s\\
 &\quad\lesssim_{p,\psi_{p,\kappa},\|\Phi^{-1}\|}  \mathbb E  \sum_{k\geq 1}\int_0^{t\wedge\tau}\bigl( |w_s|^{-2p}+|w_s|^{-2p-2} + |w|^{-2p}\|(D \phi_s)^{-1}\|^{2p}\bigr)\|\sigma_k\|_{\infty}^2\\
 &\quad\quad + |w|^{-2p-2}\|(D \phi_s)^{-1}\|^{2p+2} \|D \sigma_k\|^2_{\infty} \|D\phi_s \|^2 \ud  s <\infty,
\end{align*}
where by $\|\Phi^{-1}\|$ we denote $\|D\Phi^{-1}\|_{C(\mathcal D_{\delta_0/2})} + \|D^2\Phi^{-1}\|_{C(\mathcal D_{\delta_0/2})}$ with $\mathcal D_{\delta_0/2}:= \{(x, y)\in \mathcal M\times \mathcal M:d(x, y)\leq \delta_0/2\}$ and where $\|D_x h_{p}(x, w)\| \lesssim_{p,\psi_{p}} |w|^{-p}$ and $\|D_w h_{p}(x, w)\| \lesssim_{p,\psi_{p}} |w|^{-p-1}$ by Lemma \ref{lem:estimforhpkappaandderivac}. For a similar reason $\mathbb P$-a.s.\ for any $0\leq s\leq t\wedge\tau$
\begin{equation}\label{eq:intBsDsoimcew}
\begin{split}
 |B_s| +&| D_s| \lesssim_{p,\psi_{p,\kappa},\|\Phi^{-1}\|} \\
   &\sum_{k\geq 1}\bigl( |w|^{-p} + |w|^{-p-2}\bigr)
\Bigl(\| \sigma_k \|_{\infty}^2 +\| D\sigma_k \|_{\infty}^2 + \| \sigma_k \|_{\infty}\|D^2\sigma_k\|_{\infty}\Bigr),
\end{split}
\end{equation}
so integrals $\int_0^{t\wedge \tau}B_s\ud s$ and $\int_0^{t\wedge \tau}D_s\ud s$ exist $\mathbb P$-a.s.
 Further, note that
\begin{equation}\label{eq:B0losslikmfe}
\begin{split}
B_0&= \frac{1}{2}\sum_{k\geq 1}\Bigl\langle D_x h_{p} (x,w), \langle D\sigma_k(x), \sigma_k(x)\rangle\Bigr\rangle\\
&\quad+ \Bigl\langle D_w  h_{p} (x, w), \langle D^2\sigma_k(x), (\sigma_k(x), w)\rangle  + \langle D\sigma_k(x),\langle D\sigma_k(x), w\rangle\Bigr\rangle  \\
&\quad+  D^2_{x,x} h_p(x, w) (\sigma_k(x),\sigma_k(x))\\
&\quad+ D^2_{w,w} h_p(x, w) (\langle D \sigma_k(x), w \rangle,\langle D \sigma_k(x), w \rangle)\\
& \quad+ 2D^2_{x,w} h_p(x, w) (\sigma_k(x),\langle D \sigma_k(x),w \rangle)
\end{split}
\end{equation}
and 
\begin{align}\label{eq:D0lofwmelcomxa}
D_0&=\frac{1}{2}\sum_{k\geq 1}\left \langle Dh_p(x, w), \left \langle D^2 \Phi^{-1}(x, y), \left (\binom{\sigma_k(x)}{\sigma_k(y)},\binom{\sigma_k(x)}{\sigma_k(y)}\right)\right \rangle\right \rangle\nonumber\\
&+\left \langle Dh_p(x, w),\left \langle D \Phi^{-1}(x, y), \binom{\langle D\sigma_k(x), \sigma_k(x)\rangle}{\langle D\sigma_k(y), \sigma_k(y)\rangle}\right \rangle\right \rangle\nonumber\\
&+\left \langle D^2h_p(x, w),\left(\left \langle D \Phi^{-1}(x, y), \binom{\sigma_k(x)}{ \sigma_k(y)}\right \rangle,\left \langle D \Phi^{-1}(x, y), \binom{\sigma_k(x)}{ \sigma_k(y)}\right \rangle\right )\right \rangle\nonumber\\
&=\frac{1}{2}\sum_{k\geq 1}\Biggl \langle D_wh_p(x, w), \left \langle D^2 w(x, y), \left (\binom{\sigma_k(x)}{\sigma_k(y)},\binom{\sigma_k(x)}{\sigma_k(y)}\right)\right \rangle \\
&\quad\quad\quad\quad\quad\quad\quad\quad\quad\quad\quad\quad\quad\quad\quad+ \left \langle D w(x, y),  \binom{\langle D\sigma_k(x), \sigma_k(x)\rangle}{\langle D\sigma_k(y), \sigma_k(y)\rangle}  \right \rangle\Biggr \rangle\nonumber\\
&+\bigl \langle D_xh_p(x, w),\langle D\sigma_k(x), \sigma_k(x)\rangle\bigr\rangle+  \bigl \langle D^2_{x,x} h_p(x, w), (\sigma_k(x),\sigma_k(x))\bigr \rangle\nonumber\\
&+ \left \langle D^2_{w,w} h_p(x, w),\left( \left \langle Dw(x, y), \binom{\sigma_k(x)}{ \sigma_k(y)}\right \rangle,\left \langle Dw(x, y), \binom{\sigma_k(x)}{ \sigma_k(y)}\right \rangle\right)\right \rangle\nonumber\\
& + 2\left \langle D^2_{x,w} h_p(x, w),\left (\sigma_k(x),\left \langle Dw(x, y), \binom{\sigma_k(x)}{ \sigma_k(y)}\right \rangle\right)\right \rangle\nonumber
\end{align}
(we leave these technical but elementary calculations to the reader, the notations $D^2_{x,x}$, $D^2_{x, w}$, and $D^2_{w,w}$ are provided in Section \ref{sec:prelim}). 

By Lemma \ref{lem:phitkat64wppacontinxandkapps} we have that for any $t>0$ and $q>1$
\[
 \mathbb E \sup_{0\leq s\leq t} (d(x_s, x) + d( y_s, y))^{2q} \lesssim_q t^{q/2},
\]
so by Markov's inequality for any $C>0$
\begin{equation*}\label{eq:wkappatnotfarfromwwithtq102}
\begin{split}
\mathbb P(\tau<t) = \mathbb P\Bigl\{\sup_{0\leq s\leq t} (d(x_s, x) + d( y_s, y))^{2q}\geq \eps^{2q}\Bigr\} \lesssim_q \eps^{-2q}t^{q/2},
\end{split}
\end{equation*}
and by choosing $q>4$ with exploiting of H\"older's inequality, \eqref{eq:vsoEdxtyt-pCtdxy-p}, and Corollary \ref{cor:supofDphiontsciintegrs} we get 
\begin{equation}\label{eq:Extau<tbddbyttq/2-1oidmew}
\begin{split}
 \mathbb E_{(x, w)}&\mathbf 1_{\tau <t}\frac{ h_{p} (x_t, w_t) -h_{p} (x_t, D \phi_tw) }{t} \lesssim \mathbb E_{(x, w)} \frac{\mathbf 1_{\tau <t}}{t}\Bigl(|w_t|^{-p} + \bigl|(D\phi_t)^{-1}\bigr|^p\Bigr)\\
 &\lesssim \frac{\mathbb P(\tau<t)^{\frac{1}{2}} }{t} \Bigl(\mathbb E_{(x, w)} \Bigl(|w_t|^{-p} + \bigl|(D\phi_t)^{-1}\bigr|^p\Bigr)^{2}\Bigr)^{\frac{1}{2}}\lesssim_{\eps, q} t^{\frac{q}{2} -1},
\end{split}
\end{equation}
so the left-hand side of the latter inequality vanishes as $t\to 0$. 
Consequently, by \eqref{eq:eqforABCScawP,ce}, \eqref{eq:B0losslikmfe}, \eqref{eq:D0lofwmelcomxa}, \eqref{eq:Extau<tbddbyttq/2-1oidmew}, and the fact that $B$ and $D$ are a.s.\ continuous uniformly bounded on $[0, t \wedge \tau]$ by \eqref{eq:intBsDsoimcew}  the following holds true
\begin{align*}
  \lim_{t\to 0}\mathbb E_{(x, w)}& \frac{ h_{p} (x_t, w_t) -h_{p} (x_t, D \phi_tw) }{t}\\
  &=  \lim_{t\to 0}\mathbb E_{(x, w)} \frac{ h_{p} (x_{t\wedge \tau}, w_{t\wedge \tau}) -h_{p} (x_{t\wedge \tau}, D \phi_{t\wedge \tau}w) }{t} \\
  &=  \lim_{t\to 0}\mathbb E_{(x, w)}\frac{\int_0^{t\wedge \tau}D_s-B_s\ud s}{t} = D_0-B_0\\
 &=  \frac{1}{2} \sum_{k\geq 1}\Biggl \langle D_wh_p(x, w), \left \langle D^2 w(x, y), \left (\binom{\sigma_k(x)}{\sigma_k(y)},\binom{\sigma_k(x)}{\sigma_k(y)}\right)\right \rangle \\
&\quad\quad\quad\quad+ \left \langle D w(x, y),  \binom{\langle D\sigma_k(x), \sigma_k(x)\rangle}{\langle D\sigma_k(y), \sigma_k(y)\rangle}  \right \rangle\\
&\quad\quad\quad\quad-\langle D^2\sigma_k(x), (\sigma_k(x), w)\rangle  - \bigl \langle D\sigma_k(x),\langle D\sigma_k(x), w\rangle\bigr\rangle\Biggr \rangle\nonumber\\
&+ \Biggl \langle D^2_{w,w} h_p(x, w),\left( \left \langle Dw(x, y), \binom{\sigma_k(x)}{ \sigma_k(y)}\right \rangle,\left \langle Dw(x, y), \binom{\sigma_k(x)}{ \sigma_k(y)}\right \rangle\right)\\
&\quad\quad\quad\quad- (\langle D \sigma_k(x), w \rangle,\langle D \sigma_k(x), w \rangle)\Biggr \rangle\nonumber\\
& + 2\left \langle D^2_{x,w} h_p(x, w),\left (\sigma_k(x),\left \langle Dw(x, y), \binom{\sigma_k(x)}{ \sigma_k(y)}\right \rangle-\langle D \sigma_k(x),w \rangle\right)\right \rangle,
\end{align*}
so Theorem \ref{thm:actionL(2)kappaonVkappa} follows.
\end{proof}

Let us prove the final lemma before obtaining the drift condition \eqref{eq:driftcondinpropow}.

\begin{lemma}\label{lem:theisCinL(2)kappaVkappa}
Let $\Sigma$, $\Sigma'$, and $\Sigma''$ be defined by \eqref{eq:Sigmaforexpboundsa}, \eqref{eq:Sigmaprikqforexpboundsa}, and \eqref{eq:formforSigmad'''das}, and let $p_0<\beta/2$ with $\beta$ being as in the condition {\textnormal {\bf (C)} (see Section \ref{sec:intro})}. Then for any $p\in [0, p_0]$ there exists a constant $c>0$ such that for any $(x, y)\in \mathcal D^c$
\begin{align*}
\big\||D_{ w} h_{p} (x, w(x, y)) \Sigma | + 
 |D_{ w,w}^2 h_{p} (x, w(x, y))\Sigma'|
 &+ |D_{x, w}^2 h_{p} (x, w(x, y))\Sigma'' |\big\|_{\infty}<c.
\end{align*}
\end{lemma}

\begin{proof}
The fact that both $|D_{ w} h_{p} (x, w(x, y)) \Sigma |$ and $|D_{ w,w}^2 h_{p} (x, w(x, y))\Sigma'|$ are uniformly bounded follows thanks to the condition {\textnormal {\bf (C)}} and Lemma \ref{lem:estimforhpkappaandderivac}.
Uniform boundedness of $|D_{x, w}^2 h_{p} (x, w(x, y))\Sigma'' |$ holds due to Lemma \ref{lem:estimforhpkappaandderivac} and the fact that $\|\Sigma''\| \lesssim |w|^{1+\beta/2}$ as
\begin{multline*}
 \left\|\sum_{k\geq 1} \sigma_k(x)\otimes \left (\left \langle Dw(x, y), \binom{\sigma_k(x)}{ \sigma_k(y)}\right \rangle-\langle D \sigma_k(x),w \rangle\right)\right\|_{T_x\mathcal M \otimes T_x\mathcal M} \\
 \leq \Bigl(\sum_{k\geq 1} \|\sigma_k\|^2\Bigr)^{1/2} \Biggl(\sum_{k\geq 1} \left\|\left \langle Dw(x, y),\binom{\sigma_k(x)}{ \sigma_k(y)}\right \rangle-\langle D \sigma_k(x),w \rangle\right\|^2\Biggr)^{1/2}\\
  = O(|w|^{1+\beta/2}),
\end{multline*}
where the latter follows from the condition {\textnormal {\bf (C)}}.
\end{proof}

\begin{proof}[Proof of Proposition \ref{prop:exiofVtildeoimc}]
The proposition follows directly from Theorem \ref{thm:actionL(2)kappaonVkappa} and Lemma \ref{lem:theisCinL(2)kappaVkappa}, { where $p^*$ is set to be $p_*\wedge \tfrac {\beta}{4}$}.
\end{proof}

\subsection{Ergodicity of the two-point motion}\label{subsec:ergoftwopmot}

 Throughout this section we will be mainly working with the probability space $(\overline {\Omega}, \overline {\mathcal F}, \overline {\mathbb P})$ (see Section \ref{sec:prelim}).

{Fix any $p\in(0, p^*]$}. From Proposition \ref{prop:exiofVtildeoimc} and the fact that $\widetilde V_p$ is in the domain of $L^{(2)}$ we conclude that for any $t\geq 0$
\begin{equation}\label{eq:eLambdaoplakVpdkam=-erq23d}
\begin{split}
e^{\Lambda(p)t} P^{(2)}_t\widetilde V_{p} - \widetilde V_{p} &= \int_0^t e^{\Lambda(p)s} P^{(2)}_s(\Lambda(p) + L^{(2)} ) \widetilde V_{p} \ud s\\
& \leq \int_0^t e^{\Lambda(p)s} P^{(2)}_s c_{p}  \ud s  = c_{p} \int_0^t e^{\Lambda(p)s}   \ud s\\
& =  c_{p} \frac{e^{\Lambda(p)t}-1}{\Lambda(p)} \leq c_{p} \frac{e^{\Lambda(p)t}}{\Lambda(p)}.
\end{split}
\end{equation}
Therefore $P^{(2)}_t \widetilde V_p \leq e^{-\Lambda(p)t} \widetilde V_p + \widetilde C_p$, where $\widetilde C_p := \tfrac{c_{p}}{\Lambda(p)}$, and in particular, if we consider 
$$
V_p(x, y) := d(x, y)^{-p},\;\;\; (x, y)\in \mathcal D^c,
$$
then by \eqref{eq:widetildVpeqaimxoimxq} and thanks to the fact that $P^{(2)}_t$ is nonnegative we get that 
\begin{equation}\label{eq:P(2)tVpbasdpomK2Rr3[p}
 P^{(2)}_t  V_p \leq K^2e^{-\Lambda(p)t} V_p + \widetilde C_p K.
\end{equation}
Let us show that an analogue of \eqref{eq:P(2)tVpbasdpomK2Rr3[p} holds for $P^{(2), \kappa}$ with constants independent of $\kappa$ for $\kappa$ small enough.

\begin{proposition}\label{prop:P(2)kappasminopckascbyVpVpdla}
 For any $t>0$ there exist $\kappa_0>0$ and $C_{p, t}>0$ such that for any $\kappa\in[0, \kappa_0]$ one has that
 \begin{equation}\label{eq:P(2)kappaosalmostwhawened}
    P^{(2), \kappa}_t  V_p \leq 2K^2e^{-\Lambda(p)t} V_p + C_{p, t}.
 \end{equation}
\end{proposition}

\begin{proof}
 By \eqref{eq:P(2)kappadefinition} and the definition of $V_p$ we have that for any $(x, y)\in \mathcal D^c$ 
 \begin{align*}
   | (P^{(2), \kappa}_t  - P^{(2)}_t ) V_p(x, y) |&= |\overline{\mathbb E}_{x, y} V_p(x^{\kappa}_t, y^{\kappa}_t) -V_p(x_t, y_t)|\\
   & = |\overline{\mathbb E}_{x, y} d(x^{\kappa}_t, y^{\kappa}_t)^{-p} -d(x_t, y_t)^{-p}|\\
   &\lesssim_{p} \overline{\mathbb E}_{x, y} \bigl|d(x^{\kappa}_t, y^{\kappa}_t) -d(x_t, y_t)\bigr|  \bigl(d(x^{\kappa}_t, y^{\kappa}_t)^{-p-1} +d(x_t, y_t)^{-p-1}\bigr)\\
   &\lesssim \Bigl(\overline{\mathbb E}_{x, y}\bigl(d(x^{\kappa}_t, y^{\kappa}_t) -d(x_t, y_t)\bigr)^2\Bigr)^{\frac 12}\\
   &\;\;\;\; \cdot \Bigl(\overline{\mathbb E}_{x, y}d(x^{\kappa}_t,y^{\kappa}_t)^{-2p-2} +d(x_t, y_t)^{-2p-2}\Bigr)^{\frac 12}.
 \end{align*}
Thanks to \eqref{eq:phikappaconeinxx'} 
\[
 \Bigl(\overline{\mathbb E}_{x, y}d(x^{\kappa}_t,y^{\kappa}_t)^{-2p-2} +d(x_t, y_t)^{-2p-2}\Bigr)^{\frac 12} \lesssim_{p} d(x, y)^{-p-1}.
\]
Fix $\kappa_0>0$ to be chosen at the end of the proof. Let us bound $\overline{\mathbb E}_{x, y}\bigl(d(x^{\kappa}_t, y^{\kappa}_t) -d(x_t, y_t)\bigr)^2$. As by Lemma \ref{lem:phitkat64wppacontinxandkapps} $(\phi^{\mu^2}(x))_{x\in \mathcal M, \mu \in[0, \sqrt{\kappa_0}]}$ is a stochastic flow of $C^1$-diffeomorphisms on $\mathcal M \times [0, \sqrt{\kappa_0}]$, and as $\mathcal M \times [0, \sqrt{\kappa_0}]$ is a compact, there exists $\eps>0$ small enough such that with probability at least $1-\delta_{\eps}$ for any $x, y\in \mathcal M$ with $d(x, y)\leq \eps$ and for any $\kappa\in[0, \kappa_0]$ one has that 
\begin{equation}\label{eq:dxkappaykappt-dxtytopkld0}
|d(x^{\kappa}_t, y^{\kappa}_t) -d(x_t, y_t)|\leq (\|D\phi^{\kappa}_t - D\phi_t\|_{\infty}+ \eps^{1/4}) d(x, y),
\end{equation}
where $\delta_{\eps}\to 0$ as $ \eps\to 0$ and $\kappa_0\to 0$. To this end first fix $\kappa_0>0$ and let $\overline{\Omega}_1\subset \overline{\Omega}$ be the set of all $\overline\omega\in \overline{\Omega}$ such that for any $x, y\in \mathcal M$ with $d(x, y)\leq \eps$ one has that $(x^{\kappa}_t(\overline \omega))_{\kappa\in[0,\kappa_0]}$ and $(y^{\kappa}_t(\overline \omega))_{\kappa\in[0,\kappa_0]}$ are all in the same local chart. Assume that $\eps$ is so small that any $x, y\in \mathcal M$ with $d(x, y)\leq \eps$ are in the same local chart as well. Then by \eqref{eq:dphikappatxphikaa't'x'bqd} one can show that $\delta'_{\eps}:= 1-\mathbb P(\overline{\Omega}_1)\to 0$ as $\eps\to 0$ and $\kappa_0\to 0$.

Next, for any $\kappa\in[0, \kappa_0]$ by \eqref{eq:dphikappatxphikaa't'x'bqd} and by the fact that $(x, y)\mapsto d(x, y)$ is $C^{\infty}$ we have that for some universal $C_{\mathcal M}$ 
\begin{equation}\label{eq:bdnrssofd(xxsas)and|dasp|}
\bigl|d(x^{\kappa}_t,y^{\kappa}_t) - |y^{\kappa}_t - x^{\kappa}_t|\bigr| \leq C_{\mathcal M} d(x^{\kappa}_t,y^{\kappa}_t)^2 \leq C_{\mathcal M} \widetilde C d(x,y)^{3/2},
\end{equation}
where $\widetilde C$ is as in  \eqref{eq:dphikappatxphikaa't'x'bqd}. Let $\overline{\Omega}_2:=\overline{\Omega}_1\cap \{C_{\mathcal M} \widetilde C\eps^{1/2} \leq \eps^{1/4}/2\}$ and let $\delta_{\eps}:=1- \mathbb P(\overline{\Omega}_2)$. Then it remains to notice that $\delta_{\eps}\to 0$ as $\eps \to 0$ and $\kappa_0\to 0$ and that
\begin{equation}\label{eq:bendsm|cwom|-|decw|}
\begin{split}
\bigl||y^{\kappa}_t &- x^{\kappa}_t| - |y_t - x_t|\bigr| \leq |y^{\kappa}_t - x^{\kappa}_t - (y_t - x_t)| \\
&= \Bigl| \int_{0}^{d(x, y)} \langle D\phi^{\kappa}_t(\zeta(a)), \zeta'(a)\rangle \ud a- \int_{0}^{d(x, y)} \langle D\phi_t(\zeta(a)), \zeta'(a)\rangle \ud a\Bigr|\\
&\leq \|D\phi^{\kappa}_t - D\phi_t\|_{\infty}d(x, y),
\end{split}
\end{equation}
where $\zeta:[0, d(x, y)]\to \mathcal M$ is a $C^{\infty}$ geodesic mapping within one chart such that $\zeta(0)=x$ and $\zeta(d(x, y))=y$ with $|\zeta'(a)|=1$, so \eqref{eq:dxkappaykappt-dxtytopkld0} follows from \eqref{eq:bdnrssofd(xxsas)and|dasp|} and \eqref{eq:bendsm|cwom|-|decw|}.

Therefore thanks to \eqref{eq:dxkappaykappt-dxtytopkld0}
\begin{align*}
 \overline{\mathbb E}_{x, y}\bigl(d(x^{\kappa}_t, y^{\kappa}_t) &-d(x_t, y_t)\bigr)^2\\
 &\stackrel{(i)} \lesssim (\sqrt{\delta_{\eps}} + \eps^{1/4}+ \mathbf 1_{[\eps, \infty)}(d(x, y))) d(x, y)^2 +  \overline{\mathbb E} \|D\phi^{\kappa}_t - D\phi_t\|_{\infty}^2 d(x, y)^2\\
 & = d(x, y)^2 (\sqrt{\delta_{\eps}}+ \eps^{1/4} + \mathbf 1_{[\eps, \infty)}(d(x, y)) +  \overline{\mathbb E} \|D\phi^{\kappa}_t - D\phi_t\|_{\infty}^2)\\
 &\stackrel{(ii)} \lesssim  d(x, y)^2 (\sqrt{\delta_{\eps}} + \eps^{1/4}+ \mathbf 1_{[\eps, \infty)}(d(x, y)) +  \kappa_0^{\beta/2}),
\end{align*}
where $(i)$ holds by \eqref{eq:phikappaconeinxx'} and H\"older's inequality and $(ii)$ follows from \eqref{eq:Dphikappa-Dphikappa'bddbykpkaf[w}. Now, via choosing optimal $\eps$ and $\kappa_0$ small enough one can obtain that
\[
  | (P^{(2), \kappa}_t  - P^{(2)}_t ) V_p(x, y) | \leq 
  \begin{cases}
   K^2 e^{-\Lambda(p)t}d(x, y)^{-p},\;\;\;&d(x, y)\leq \eps,\\
   C_{p, t},\;\;\;&d(x, y)> \eps,
  \end{cases}
\]
for some $C_{p, t}>0$, so \eqref{eq:P(2)kappaosalmostwhawened} holds true.
\end{proof}

Later one needs $C_{p, t}$ being independent of $t>0$. This can be done via the following corollary.

\begin{corollary}\label{cor:existnoft*dxc}
 Fix $t^*>0$ so that $2K^2e^{-\Lambda(p)t^*}<1$ and $\kappa_0$ as in Proposition \ref{prop:P(2)kappasminopckascbyVpVpdla} corresponding to $t^*$. Then there exists $C_p>0$ such that for any $n\geq 1$ and $\kappa\in[0, \kappa_0]$
 \begin{equation}\label{eq:P(2)kappaindof}
    P^{(2), \kappa}_{nt^*}  V_p \leq (2K^2e^{-\Lambda(p)t^*})^n V_p + C_{p},
 \end{equation}
 where $C_p := \tfrac{C_{p, t^*}}{1-2K^2e^{-\Lambda(p)t^*}}$.
\end{corollary}

\begin{proof}
 It is sufficient to notice that by induction and the fact that $P^{(2), \kappa}_{nt^*}  = (P^{(2), \kappa}_{t^*})^n$
 \[
  P^{(2), \kappa}_{nt^*}  V_p \leq (2K^2e^{-\Lambda(p)t^*})^n V_p + C_{p, t^*}\bigl(1 + 2K^2e^{-\Lambda(p)t^*} + \ldots + (2K^2e^{-\Lambda(p)t^*})^{n-1}\bigr),
 \]
so the desired follows.
\end{proof}

Thanks to Harris' Theorem \ref{thm:QuaeHarwqthnms} we can conclude the following.

\begin{proposition}\label{prop:harridsthmforsVkappa}
Fix $p^*>0$ as in Proposition \ref{prop:exiofVtildeoimc}.
Then for any $0<p<p^*$ there exist $t_0, \kappa_0,C,\hat{\alpha}>0$ such that for any $\kappa\in [0, \kappa_0]$, for any $\psi: \mathcal D^c \to \mathbb R$ with $\int_{\mathcal M \times \mathcal M} \psi(x, y) \ud \mu( x)\ud \mu( y) = 0$, and for any $(x, y)\in \mathcal D^c$ we have that
 \begin{equation}\label{eq:P(2)kappapsiifsdsbddbyCVk||||CVk}
  |P^{(2), \kappa}_{nt_0}\psi(x, y)| \leq C e^{-\hat{\alpha} nt_0} d(x, y)^{-p} \|\psi\|_{\infty},\;\;\; n\geq 1.
 \end{equation}
\end{proposition}

\begin{proof}
Let us start by noticing that $\mu\otimes \mu$ is an invariant measure of $(x^{\kappa}, y^{\kappa})$ on $\mathcal D^c$ 
 as $t\mapsto (x^{\kappa}, y^{\kappa})$ satisfies the following SDE
defined on $\mathcal D^c$ by
\[
  \dd\binom{x_t^{\kappa}}{y_t^{\kappa}} = \sum_{k\geq 1} \binom{\sigma_k(x_t^{\kappa})}{\sigma_k(y_t^{\kappa})} \circ \ud W_{t}^k 
  + \sqrt{2\kappa} \sum_{m=1}^n \binom{\chi_m(x_t^{\kappa})}{\chi_m(y_t^{\kappa})} \circ \ud \widetilde W_t^m,
\]
which contains only divergence-free vector fields and hence preserves $\mu\otimes\mu$ a.s.\ by Remark \ref{rem:phikappameasupres}. 
Hence in order to show \eqref{eq:P(2)kappapsiifsdsbddbyCVk||||CVk} one needs only to check the conditions of Theorem \ref{thm:QuaeHarwqthnms} (in this case the unique invariant measure would be $\mu\otimes \mu$), i.e.\ we need to show that for $\kappa_0$ being from Proposition \ref{prop:P(2)kappasminopckascbyVpVpdla} and for $t_0=t^*$ as in Corollary \ref{cor:existnoft*dxc} for any $\kappa\in[0, \kappa_0]$ $\mathcal P= P_{t^*}^{(2), \kappa}$ satisfies the conditions $(i)$ and $(ii)$ of Theorem \ref{thm:QuaeHarwqthnms} with constants independent of $\kappa\in[0,\kappa_0]$.

 The condition $(i)$ is satisfied by Corollary \ref{cor:existnoft*dxc} with constants $\gamma := 2K^2e^{-\Lambda(p)t^*}$ and $C:=C_p$. Let us now check the condition $(ii)$, i.e.\ that there exist $\eta>0$, $z_{*}\in \mathcal D^c$, and $\eps>0$ such that

\begin{equation}\label{eq:transprobabforP(2)kappafromztoballinbssfromdsa}
\inf_{z\in \{V_{p} \leq R\}} P_{t^*}^{(2), \kappa} \mathbf 1_{A'} (z) >\eta (\mu\otimes \mu)(A'),
\end{equation}
for any $A'\subset B_{\eps}(z_*)$. { Choose $R>1$ to be such that 
$$
\{V_p<R\} = \{(x, y)\in \mathcal D^c:d(x, y)>R^{-1/p}\}
$$ 
is nonempty.}
Then the condition follows from Harnack-type inequalities presented in Proposition \ref{prop:appendxHarnackmancewell}, since for any $\kappa \geq 0$ the transition probability function $p^{\kappa}:(0, \infty)\times \mathcal D^c \times \mathcal D^c$ corresponding to $(x^{\kappa}, y^{\kappa})$ satisfies the following equation for any $u\in C_0^{\infty}(\mathbb R_+ \times \mathcal D^c)$
\[
\int_{\mathbb R_+ \times \mathcal D^c} \left[\frac{\partial u(t, (x, y))}{\partial t} +L^{(2), \kappa}u(t, (x, y))\right] p^{\kappa}(t, \cdot, (x, y))\ud t\ud \mu (x) \ud \mu (y)=0,
\]
where 
\[
L^{(2), \kappa}u:= \left(\frac 12  \sum_{k\geq 1} \binom{\sigma_k}{\sigma_k}^2 +\kappa\sum_{m=1}^n \binom{\chi_m}{\chi_m}^2  \right) u,
\]
see \eqref{eq:prelimsigma2xfcq}. Note that the second- and first-order coefficients, as well as linear term of $L^{(2), \kappa}$ are continuous in $\kappa$. Moreover, by the condition {\textnormal {\bf (C)}} the derivatives of the second- and first-order coefficients and by the condition {\textnormal {\bf (A)}} the ellipticity constant on any given compact $K\subset \mathcal D^c$ are continuous in $\kappa$ as well. Therefore one can apply Proposition \ref{prop:appendxHarnackmancewell} for $p^{\kappa}$ with setting the corresponding compact $K$ to be $\{V_p\leq R\}$ so that due to the continuity in $\kappa$ one is able to assume that both $a_K$ and $b_{K}$ are independent on $\kappa\in [0, \kappa_0]$, hence the corresponding Harnack constant $c_{t^*, K, a_K, b_K}$ is independent of  $\kappa\in [0, \kappa_0]$ and thus \eqref{eq:transprobabforP(2)kappafromztoballinbssfromdsa} holds true.
\end{proof}

\subsection{Proof of Theorem \ref{thm:H-sconvergensofHsinvartext}}\label{subsec:proofofThmmain}

We now present the proof of Theorem \ref{thm:H-sconvergensofHsinvartext}. 
To this end it is sufficient to show the following proposition, connecting ergodicity of the two-point motion and mixing.
\begin{proposition}\label{prop:twopointergimplmixing}
 Assume that the condition {\textnormal {\bf (C)}} is satisfied and that in addition
 there exist $t_0, \kappa_0,C,\hat{\alpha}>0$ such that for any $\kappa\in [0, \kappa_0]$, for any $f\in L^{
 \infty}(\mathcal M)$ with $\int_{\mathcal M }f(x) \ud \mu( x) = 0$ we have that
 \begin{equation}\label{eq:P(2)kappapsforphitimesphi}
  \int_{\mathcal M\times \mathcal M}\bigr|\mathbb E f(\phi^{\kappa}_{nt_0}(x))f(\phi^{\kappa}_{nt_0}(y))\bigl|\ud \mu( x) \ud \mu( y)  \leq C e^{-\hat{\alpha} nt_0}  \|f\|_{\infty}^2,\;\;\; n\geq 1.
 \end{equation}
 Let $(u_t)_{t\geq 0}$ be the solution of \eqref{eq:mainstochdiffeqwithkappa} with $u$ being  mean-zero in $H^s(\mathcal M)\cap H^1(\mathcal M)$ for some $s\in(0, 1+\beta/2)$. Then there exist $\kappa_0>0$ and $\gamma_0>0$ independent of $s$ and $u$ such that for any $\kappa \in [0, \kappa_0]$ and $\gamma\in(0, \gamma_0)$ there exists $D_{\kappa,\gamma}:\Omega \to [1, \infty)$ so that 
 \begin{equation}\label{eq:H-sHsineqsfaprop}
  \|u_t\|_{H^{-s}} \leq D_{\kappa,\gamma}e^{-\gamma s t} \|u\|_{H^s},\;\;\; t> 0.
 \end{equation}
 
 Moreover, $\sup_{\kappa\in[0, \kappa_0]}\mathbb E |D_{\kappa,\gamma}|^p<\infty$ for any $1\leq p<\tfrac{9d \gamma_0}{2\gamma s}$.
\end{proposition}

Note that \eqref{eq:P(2)kappapsforphitimesphi} is a direct consequence of \eqref{eq:P(2)kappapsiifsdsbddbyCVk||||CVk} as for any $0<p<d$ one has $\int_{\mathcal M}\int_{\mathcal M} d(x, y)^{-p}\ud \mu(x)\ud \mu(y)<\infty$.

The proof of Proposition \ref{prop:twopointergimplmixing} is inspired by Section 2.1 and Section 7 from \cite{BBP-S}. In the first step of the proof, we restrict to integer times, relying on their countability. This step exploits in an essential way the ergodicity of the two-point motion \eqref{eq:P(2)kappapsforphitimesphi} 
and allows to conclude for large orders $s$. Using an interpolation argument, in the second step, this is extended to general order $s$. In the third step, the result is extended to all times, by means of local-in-time regularity preservation by the stochastic flow. Together with moments bounds, derived in step five, the proof is finished in the final step.

\textit{Step 1: Integer times and large $s$}

 First, similarly to \cite[Section 2.1]{BBP-S} we prove the proposition for integer times $n$. 
 To this end we notice that for any $n\geq 1$, for any zero mean $f,g\in C(\mathcal M)$, and for any fixed $0<\hat{\gamma}< \hat{\alpha}/2$ 
 (here we set for simplicity that $t_0$ 
 equals 1)
\begin{equation}\label{eq:Pthatfgphikappanletwqe-alpha2gammahads}
\begin{split}
\overline {\mathbb P} \Bigl\{\Big| \int_{\mathcal M} f(x) g(\phi^{\kappa}_n(x)) \ud &\mu(x) \Big| >e^{-\hat{\gamma}n} \Bigr\} \leq e^{2\hat{\gamma}n} \mathbb E \Bigl| \int_{\mathcal M} f(x) g(\phi^{\kappa}_n(x)) \ud \mu(x) \Bigr|^2\\
&= e^{2\hat{\gamma}n} \int_{\mathcal M \times \mathcal M} \tilde f(x, y) P^{(2)}_n \tilde g(x, y) \ud \mu(x) \ud \mu(y)\\
&\stackrel{(*)} \lesssim e^{-( \hat{\alpha} - 2\hat{\gamma})n} \|f\|_{\infty}^2 \|g\|_{\infty}^2,
\end{split}
\end{equation}
where $\tilde f(x, y) = f(x)f(y)$, $\tilde g(x, y) = g(x)g(y)$, and where $(*)$ follows from \eqref{eq:P(2)kappapsforphitimesphi}. By using the fact that the embedding $H^s \hookrightarrow L^{\infty}$ is continuous for $s$ big enough we can conclude that 
$$
\overline {\mathbb P} \Bigl\{\Big| \int_{\mathcal M} f(x) g(\phi^{\kappa}_n(x)) \ud\mu(x) \Big| >e^{-\hat{\gamma}n} \Bigr\} \lesssim e^{-( \hat{\alpha} - 2\hat{\gamma})n} \|f\|_{H^s}^2 \|g\|_{H^s}^2,
$$ 
so by Borell-Cantelli lemma we have that
$$
\Bigl| \int_{\mathcal M} f(x) g(\phi^{\kappa}_n(x,\overline\omega)) \ud \mu(x)\Bigr | \lesssim_{\overline\omega, f, g} e^{-\hat{\gamma}n}
$$ 
for a.e.\ $\overline \omega\in \overline \Omega$.
Unfortunately, this logic does not seem to lead to \eqref{eq:H-sHsineqsfa} as there we need a constant independent of $f\in H^s$ (recall that 
$$
\|g\|_{H^{-s}} = \sup_{0\neq f\in H^s}\int_{\mathcal M} fg/\|f\|_{H^s}
$$
for any $g\in H^{-s}$). So we will proceed via an orthonormal basis of $H^s$. Fix any $s_0>2d$ and $f, g\in H^{s_0}$. Let $(\hat{e}_k)_{k\geq 1}$ be an orthonormal basis of mean zero functions on $\mathcal M$  consisting of all eigenvectors of $-\Delta$ and let $(\lambda_k)_{k\geq 1}$ be the corresponding eigenvalues (see Section \ref{sec:prelim}). Without loss of generality let $(\lambda_k)_{k\geq 1}$ be nondecreasing. Then by Weyl's law (see e.g.\ \cite{Pes19}) we have that 
\begin{equation}\label{eq:weylslawdwq}
\lambda_k \lesssim_{\mathcal M} k^{2/d},\;\;\;k\geq 1;
\end{equation}
 we also know that $\|\hat{e}_k\|_{\infty} \lesssim_{\mathcal M} \lambda_k^{d/4}$ (see e.g.\ \cite{So88,Do01}). Therefore we can conclude that
\begin{equation}\label{eq:ekinfk12da54}
\|\hat{e}_k\|_{\infty} \lesssim_{\mathcal M}  k^{1/2},\;\;\; k\geq 1.
\end{equation}
Moreover, for each $k, k' \geq 1$ analogously to \eqref{eq:Pthatfgphikappanletwqe-alpha2gammahads} by \eqref{eq:ekinfk12da54} we have that
\begin{multline*}
\overline {\mathbb P} \Bigl\{\Big| \int_{\mathcal M} \hat{e}_k(x) \hat{e}_{k'}(\phi^{\kappa}_n(x)) \ud \mu(x) \Big| >e^{-\hat{\gamma}n} (kk')^{\frac{s_0}{d}-1}\Bigr\} \\
\lesssim e^{-( \hat{\alpha} - 2\hat{\gamma})n} (kk')^{2-\frac{2s_0}{d}} \|\hat{e}_k\|_{\infty}^2 \|\hat{e}_{k'}\|_{\infty}^2
 \lesssim e^{-( \hat{\alpha} - 2\hat{\gamma})n} (kk')^{3-\frac{2s_0}{d}},
\end{multline*}
so by Borell-Cantelli lemma and the fact that $\sum_{n, k, k'\geq 1} e^{-( \hat{\alpha} - 2\hat{\gamma})n} (kk')^{3-\frac{2s_0}{d}}<\infty$ there exists a.s.\ finite $\widehat D:\overline\Omega \to \mathbb R_+$ defined by 
\begin{equation}\label{eq:widehatDnewfpw,xw}
\widehat D := 1\vee \sup_{n, k, k'\geq 1} \frac{| \int_{\mathcal M} \hat{e}_k(x) \hat{e}_{k'}(\phi^{\kappa}_n(x)) \ud \mu(x) | e^{\hat{\gamma}n}}{(kk')^{\frac{s_0}{d}-1}},
\end{equation}
so that $| \int_{\mathcal M} \hat{e}_k(x) \hat{e}_{k'}(\phi^{\kappa}_n(x)) \ud \mu(x) | \leq \widehat D  e^{-\hat{\gamma}n }( kk')^{\frac{s_0}{d}-1}$ a.s. In particular, if $f$ and $g$ have the following expansions 
$$
f = \sum_{k\geq 1} f_k \hat e_k\;\;\; \text{and}\;\;\;g = \sum_{k\geq 1} g_k \hat e_k,
$$
then we get that a.s.\ for any $n\geq 1$
\begin{multline*}
\Bigl| \int_{\mathcal M} f(x) g(\phi^{\kappa}_n(x))\ud \mu(x) \Bigr| \leq \sum_{k, k' \geq 1} |f_k||g_{k'}| \Big| \int_{\mathcal M} \hat{e}_k(x) \hat{e}_{k'}(\phi^{\kappa}_n(x)) \ud \mu(x)\Big |\\
\leq \widehat{D} e^{-\hat{\gamma}n }  \sum_{k\geq 1} k^{\frac{s_0}{d}-1}|f_k| \sum_{k'\geq 1}k'^{\frac{s_0}{d}-1}|g_{k'}| \lesssim \widehat D e^{-\hat{\gamma}n} \|f\|_{H^{s_0}} \|g\|_{H^{s_0}},
\end{multline*}
where for the last estimate we used that by the Cauchy-Schwartz inequality, \eqref{eq:defofHspcq}, and \eqref{eq:weylslawdwq}
\[
 \sum_{k\geq 1} k^{\frac{s_0}{d}-1}|f_k|  \leq \|( k^{\frac{s_0}{d}}|f_k|)_{k\geq 1}\|_{\ell^2} \|(k^{-1})_{k\geq 1}\|_{\ell^2} \eqsim \|f\|_{H^s}.
\]
Thus \eqref{eq:H-sHsineqsfa} follows for integer $t=n$ and any fixed $s_0 >2d$.

\textit{Step 2: Integer times and arbitrary $s$}

In order to move to a general $s>0$ notice that $S:g\mapsto S(g):=  g(\phi^{\kappa}_n(\cdot))$ is a bounded linear operator from $H^{s_0}$ to $H^{-s_0}$ with the norm bounded from above by $\widehat D e^{-\hat{\gamma}n}$ a.s.; at the same time this is an isometry in $L^2$ by the fact that $\phi^{\kappa}_t$ is $\mu$-measure preserving thanks to Remark \ref{rem:phikappameasupres}. Hence thanks to the Riesz-Thorin interpolation argument (see e.g.\ \cite[Section C]{HNVW1}) we can extend \eqref{eq:H-sHsineqsfa} to any $0< s \leq  2d$ (and hence to any $s>0$) still with $t$ being integer so that
$$
\|S\|_{\mathcal L(H^s, H^{-s})} \leq \|S\|_{\mathcal L(H^{s_0}, H^{-s_0})}^{\frac{s}{s_0}}\|S\|_{\mathcal L(L^2)}^{1-\frac{s}{s_0}} = \|S\|_{\mathcal L(H^{s_0}, H^{-s_0})}^{\frac{s}{s_0}},
$$ 
in particular
\begin{equation}\label{eq:fgintegekviaHscaoim}
\Bigl| \int_{\mathcal M} f(x) g(\phi^{\kappa}_n(x)) \ud \mu(x) \Bigr|\lesssim_{s} \widehat D^{\frac{s}{s_0}} e^{-\hat{\gamma}'n} \|f\|_{H^{s}} \|g\|_{H^{s}},
\end{equation}
where $\hat{\gamma}' = \frac{s\hat{\gamma}}{s_0}$ (one can assume $s_0=3d$ for simplicity).

\textit{Step 3: Arbitrary times}

Fix $s\in (1,1+\beta/2)$, the case $s\in(0, 1]$ can be done analogously. Let us now treat general $t\geq 0$. This can be done similarly to \cite[Section 7.2]{BBP-S}: it is enough to notice that for a.e.\ $\overline\omega \in \overline\Omega$
\begin{equation}\label{eq:fgineintintegtorealvcda}
\Bigl | \int_{\mathcal M} f(x) g(\phi^{\kappa}_t(x, \overline\omega)) \ud \mu(x) \Bigr| \lesssim_{s} \widehat{D}(\overline\omega)^{\frac{s}{s_0}}e^{-\hat{\gamma}'n} \|f\|_{H^s} \|g(\phi^{\kappa}_{\hat t}(\cdot, \theta_n\overline \omega))\|_{H^s},
\end{equation}
where $n\in \mathbb N$ and $\hat t\in[0, 1)$ are such that $t = n + \hat t$ and where $\theta_n$ is as defined in Section \ref{sec:prelim}, see \eqref{eq:thetataudef}. Then thanks to \eqref{eq:Sobolevequivalentnormisd} for any $\hat t\in[0, 1]$ we can bound for a.e.\ $\overline\omega\in \overline\Omega$ (we omit $\overline\omega$ for the simplicity)
\begin{equation}\label{eq:bigeqforSobSlnors}
\begin{split}
&\|g(\phi^{\kappa}_{\hat t})\|_{H^s}  \eqsim_{\mathcal M, (U_{\alpha})_{\alpha \in I}}\|g\|_{L^2} \\
&\quad+\Bigl( \iint_{\mathcal D^c} \frac{|\langle Dg(x^{\kappa}_{\hat t}), D\phi^{\kappa}_{\hat t}(x)\rangle-\langle Dg(y^{\kappa}_{\hat t}),D\phi^{\kappa}_{\hat t}(y)\rangle |^2}{d(x, y) ^{2\{s\}+d}} \ud \mu\otimes\mu(x, y) \Bigr)^{\frac{1}{2}}
\end{split}
\end{equation}
where $\{s\}:= s-1$ and where $g$ can be assumed $C^{\infty}$ by an approximation argument. The right hand side of \eqref{eq:bigeqforSobSlnors} can be bounded by
\begin{align*}
\|g\|_{L^2}&+  \Bigl( \iint_{\mathcal D^c} \frac{|Dg(x^{\kappa}_{\hat t})-Dg(y^{\kappa}_{\hat t})|^2\|D\phi^{\kappa}_{\hat t}\|_{L^{\infty}}^2 }{d(x, y)^{2\{s\}+d}}  \ud \mu\otimes\mu(x, y) \Bigr)^{\frac{1}{2}}\\
 &+\Bigl( \iint_{\mathcal D^c} \frac{ \|D\phi^{\kappa}_{\hat t}(x)-D\phi^{\kappa}_{\hat t}(y)\|^2|Dg(x^{\kappa}_{\hat t})|^2}{d(x, y)^{2\{s\}+d}}  \ud \mu\otimes\mu(x, y) \Bigr)^{\frac{1}{2}}.
\end{align*}
In order to bound the expression above on the one hand notice that by \eqref{eq:Dphikappa-Dphikappa'bddbykpkaf[w} $\overline {\mathbb P}$-a.s.
\begin{align*}
&\iint_{\mathcal D^c} \frac{ \|D\phi^{\kappa}_{\hat t}(x)-D\phi^{\kappa}_{\hat t}(y)\|^2|Dg(x^{\kappa}_{\hat t})|^2}{d(x, y)^{2\{s\}+d}}  \ud \mu\otimes\mu(x, y)\\
&\quad\quad\quad\quad \leq \widetilde C^2 \int_{\mathcal M}|Dg(x^{\kappa}_{\hat t})|^2  \int_{\mathcal M}d(x, y)^{\beta-2\{s\}-d} \ud \mu(y) \ud \mu(x) \\
&\quad\quad\quad\quad \lesssim_{\mathcal M, \beta}\widetilde C^2 \int_{\mathcal M}|Dg(x^{\kappa}_{\hat t})|^2   \ud \mu(x)\lesssim_{\mathcal M}\widetilde C^2 \|g\|_{H^s}^2,
\end{align*}
where we use the fact that the map $x\mapsto x^{\kappa}_{\hat t}$ preserves $\mu$ a.s.\ by Remark \ref{rem:phikappameasupres} and where $\widetilde C:\overline\Omega\to \mathbb R_+$ is an integrable random variable independent of $\hat t\in[0, 1]$ and $\kappa\in[0, \kappa_0]$ (we assume that $T=1$ in \eqref{eq:Dphikappa-Dphikappa'bddbykpkaf[w}, see Lemma \ref{lem:phitkat64wppacontinxandkapps}), and on the other hand
\begin{align*}
&\Bigl( \iint_{\mathcal D^c} \frac{|Dg(x^{\kappa}_{\hat t})-Dg(y^{\kappa}_{\hat t})|^2\|D\phi^{\kappa}_{\hat t}\|_{L^{\infty}}^2 }{d(x, y)^{2\{s\}+d}}  \ud \mu\otimes\mu(x, y) \Bigr)^{\frac{1}{2}}\\
&\quad\quad\quad=  \|D\phi^{\kappa}_{\hat t}\|_{L^{\infty}} \Bigl( \iint_{\mathcal D^c} \frac{|Dg(x^{\kappa}_{\hat t})-Dg(y^{\kappa}_{\hat t})|^2 }{d(x^{\kappa}_{\hat t},y^{\kappa}_{\hat t})^{2\{s\}+d}} \frac{d(x^{\kappa}_{\hat t},y^{\kappa}_{\hat t})^{2\{s\}+d}}{d(x,y)^{2\{s\}+d}}  \ud \mu\otimes\mu(x, y) \Bigr)^{\frac{1}{2}} \\
&\quad\quad\quad\lesssim \|D\phi^{\kappa}_{\hat t}\|_{L^{\infty}}^{s + d/2+1} \Bigl( \iint_{\mathcal D^c} \frac{|Dg(x^{\kappa}_{\hat t})-Dg(y^{\kappa}_{\hat t})|^2}{d(x^{\kappa}_{\hat t},y^{\kappa}_{\hat t})^{2\{s\}+d}}\ud \mu\otimes\mu(x, y) \Bigr)^{\frac{1}{2}}\\
&\quad\quad\quad\stackrel{(*)}=  \|D\phi^{\kappa}_{\hat t}\|_{L^{\infty}}^{s + d/2+1} \Bigl( \iint_{\mathcal D^c} \frac{|Dg(x)-Dg(y)|^2}{d(x,y)^{2\{s\}+d}}\ud \mu\otimes\mu(x, y) \Bigr)^{\frac{1}{2}}\\
&\quad\quad\quad\eqsim_{\mathcal M, s }\|D\phi^{\kappa}_{\hat t}(\overline\omega)\|_{L^{\infty}}^{s + d/2+1}\|g\|_{H^s},
\end{align*}
where $(*)$ follows from the fact that the mapping $(x, y)\mapsto (x^{\kappa}_{\hat t}, y^{\kappa}_{\hat t})$ preserves $\mu\otimes \mu$ thanks to Remark \ref{rem:phikappameasupres} (see the proof of Proposition \ref{prop:harridsthmforsVkappa}).
Therefore
\begin{equation}\label{eq:frinn+thttothatecmaxGamma}
 \sup_{\hat t\in[0, 1]}\|g(\phi^{\kappa}_{\hat t})(\cdot, \overline\omega)\|_{H^s} \lesssim \Gamma(\overline\omega) \|g\|_{H^s},\;\;\;\overline \omega \in \overline\Omega,
\end{equation}
 with
\begin{equation}\label{eq:defofGgfdammaomegdasdgbvp}
 \Gamma(\overline\omega):= \sup_{\kappa\in[0, \kappa_0]}\sup_{t\in [0, 1]} \|D\phi^{\kappa}_t(\overline\omega)\|_{L^{\infty}(\mathcal M)}^{s + d/2+1} + \widetilde C(\overline\omega),\;\;\; \overline\omega \in \overline\Omega,
\end{equation}
for which $L^{p}$-boundedness for any $1\leq p<\infty$ follows from Corollary \ref{cor:supofDphiontsciintegrs} and Lemma \ref{lem:phitkat64wppacontinxandkapps}.  By setting
\begin{equation}\label{eq:formforCo12dw}
C(\overline \omega) := \sum_{n\geq 1} {\Gamma(\theta_n \overline \omega)}e^{-\eps n} ,\;\;\;\overline \omega \in \overline \Omega,
\end{equation}
we have that
$ \Gamma(\theta_n \cdot) \leq C(\cdot) e^{\eps n}$ for some $C:\overline\Omega \to \mathbb R_+$,
so \eqref{eq:fgineintintegtorealvcda} and \eqref{eq:frinn+thttothatecmaxGamma} yield
\begin{equation}\label{eq:hammaisgammahatprmnepasfdq}
\Bigl | \int_{\mathcal M} f(x) g(\phi^{\kappa}_t(x, \overline\omega)) \ud \mu(x) \Bigr| \lesssim_{s}\widehat D(\overline\omega)^{\frac{s}{s_0}} C(\overline\omega) e^{-(\hat{\gamma}'-\eps)t} \|f\|_{H^s} \|g\|_{H^s},\; \overline\omega \in \overline\Omega.
\end{equation}

\textit{Step 4: Moment bounds}

In order to finish the proof we will need the following lemmas concerning $L^p$-moments of $C$ and $\widehat{D}$.
Let $\Sigma = (\sigma_k)_{k\geq 1}$.

\begin{lemma}\label{lem:ChasallLpnporemcfas}
 For any $1\leq p<\infty$ we have that $\mathbb E C^p\lesssim_{\Sigma, \kappa_0, s,p} \tfrac {1}{\eps^p}$.
\end{lemma}

\begin{proof}
By \eqref{eq:formforCo12dw}, the fact that the distribution of $\Gamma(\theta_n \cdot)$ does not depend on $n\geq 0$ (see the definition \eqref{eq:thetataudef} of $\theta_n$), and a triangle inequality 
 \[
  (\mathbb E C^p)^{1/p} \leq \sum_{n\geq 1}\|\Gamma(\theta_n\cdot)\|_{L^p(\Omega)}e^{-\eps n} = \sum_{n\geq 1}\|\Gamma\|_{L^p(\Omega)}e^{-\eps n}\stackrel{(*)}\lesssim_{\Sigma, \kappa_0, s,p} \sum_{n\geq 1} e^{-\eps n} \leq 1/\eps,
 \]
where $(*)$ follows from \eqref{eq:defofGgfdammaomegdasdgbvp},  Lemma \ref{lem:phitkat64wppacontinxandkapps}, and Corollary \ref{cor:supofDphiontsciintegrs}.
\end{proof}

\begin{lemma}\label{lem:hatDhasmanyLpnofnmwioc}
Choose $\hat{\gamma}< \hat{\alpha}/2$ and $s_0\geq 3d$. Let $1\leq p< \frac{\hat{\alpha}}{2\hat{\gamma}}$ and $\kappa\in[0, \kappa_0]$. Then $\mathbb E |\widehat D|^p\lesssim_{\Sigma, \kappa_0} 1+\frac{1}{\left(\tfrac{ps_0}{d}-p-\tfrac32\right)^2\left(\tfrac{\hat\alpha}{2}-p\hat \gamma\right)}<\infty$.
\end{lemma}

\begin{proof}
Due to \eqref{eq:widehatDnewfpw,xw} we have that
\begin{align*}
\mathbb E |\widehat D|^p &\leq 1+ \sum_{n, k, k'\geq 1}\frac{ \mathbb E| \int_{\mathcal M} \hat{e}_k(x) \hat{e}_{k'}(\phi^{\kappa}_n(x)) \ud \mu(x) |^p e^{p\hat{\gamma}n}}{(kk')^{p(\frac{s_0}{d}-1)}}\\
&\stackrel{(i)}\leq 1+ \sum_{n, k, k'\geq 1}\frac{ \mathbb E| \int_{\mathcal M} \hat{e}_k(x) \hat{e}_{k'}(\phi^{\kappa}_n(x)) \ud \mu(x) | e^{p\hat{\gamma}n}}{(kk')^{p(\frac{s_0}{d}-1)}}\\
&\leq 1+ \sum_{n, k, k'\geq 1}\frac{ (\mathbb E| \int_{\mathcal M} \hat{e}_k(x) \hat{e}_{k'}(\phi^{\kappa}_n(x)) \ud \mu(x) |^2)^{\frac 12} e^{p\hat{\gamma}n}}{(kk')^{p(\frac{s_0}{d}-1)}}\\
&\stackrel{(ii)}\lesssim_{\Sigma, \kappa_0}  1+ \sum_{n, k, k'\geq 1}\frac{ e^{p\hat{\gamma}n-\frac{\hat{\alpha}}{2}n} }{(kk')^{\frac{ps_0}{d}-p-\frac 12}} =1+\left(\sum_{k\geq 1}\frac {1}{k^{\frac{ps_0}{d}-p-\frac 12}}\right)^2\sum_{n\geq 1}e^{p\hat{\gamma}n-\frac{\hat{\alpha}}{2}n}\\
&\lesssim 1+\frac{1}{(\tfrac{ps_0}{d}-p-\tfrac32)^2\left(\tfrac{\hat\alpha}{2}-p\hat \gamma\right)} \stackrel{(iii)}<\infty,
\end{align*}
where $(i)$ holds as $| \int_{\mathcal M} \hat{e}_k(x) \hat{e}_{k'}(\phi^{\kappa}_n(x)) \ud \mu(x) | \leq 1$ a.s.\
 by the Cauchy-Schwarz inequality, $(ii)$ follows from \eqref{eq:ekinfk12da54} and the second half of \eqref{eq:Pthatfgphikappanletwqe-alpha2gammahads}, and where $(iii)$ holds as $\tfrac{ps_0}{d}-p-\tfrac 12>1$ (since $p\geq 1$ and $s_0\geq 3d$) and as $p\hat{\gamma} < \tfrac{\hat{\alpha}}{2}$.
\end{proof}

\textit{Step 5: Conclusion}

 Now we are ready to prove \eqref{eq:H-sHsineqsfa} in full generality. To this end notice that thanks to \eqref{eq:Utviaphikaplxw} and the fact that $\phi^{\kappa}_t$ is $\overline{\mathbb P}$-a.s.\ measure-preserving for a.e.\ $\omega\in \Omega$ by Remark \ref{rem:phikappameasupres} we have that (we omit $\omega$ for simplicity)
\begin{align*}
\|u_t\|_{H^{-s}} &:= \sup_{g\in H^s, g\neq 0}\frac{\int_{\mathcal M} u_t(x) g(x) \ud \mu(x)}{\|g\|_{H^s}}\\
& = \sup_{g\in H^s, g\neq 0}\frac{\mathbb E_{\widetilde W}\int_{\mathcal M} u\bigl((\phi^{\kappa}_t)^{-1}(x)\bigr) g(x) \ud \mu(x)}{\|g\|_{H^s}}\\
&= \sup_{g\in H^s, g\neq 0}\frac{\mathbb E_{\widetilde W}\int_{\mathcal M} u(x) g(\phi^{\kappa}_t(x)) \ud \mu(x)}{\|g\|_{H^s}}\\
&\stackrel{(*)}\lesssim_s \mathbb E_{\widetilde W}\Bigl[\widehat D^{\frac{s}{s_0}} C e^{-(\hat{\gamma}'-\eps)t} \|u\|_{H^s} \Bigr] = \mathbb E_{\widetilde W}\bigl[\widehat D^{\frac{s}{s_0}} C\bigr]  e^{-(\hat{\gamma}'-\eps)t} \|u\|_{H^s},
\end{align*}
where $(*)$ follows from \eqref{eq:hammaisgammahatprmnepasfdq} and where the latter conditional expectation is finite due to Lemma \ref{lem:ChasallLpnporemcfas} and \ref{lem:hatDhasmanyLpnofnmwioc} and H\"older's inequality. Therefore \eqref{eq:H-sHsineqsfa} holds true with 
\begin{equation}\label{eq:Dkappaexpasform}
D_{\kappa,\gamma}(\omega)\eqsim_s \mathbb E_{\widetilde W}\widehat {D}(\omega, \cdot)^{\frac{s}{s_0}}C(\omega, \cdot),\;\;\;\omega\in \Omega.
\end{equation}
It remains to estimate $L^p$-moments of $D_{\kappa,\gamma}$. Set $\gamma_0 =  \tfrac{\hat{\alpha}}{4s_0}$ and fix any $\gamma\in (0, \gamma_0)$. Let $\hat{\gamma}=\tfrac{\hat{\alpha}\gamma}{3\gamma_0}=\tfrac{4s_0\gamma}{3}$. Then, we have that $\hat{\gamma}'=\tfrac {s\hat{\gamma}}{s_0}=\tfrac{4}{3}\gamma s$, so $\gamma s=\hat{\gamma}'-\eps$ for $\eps = \gamma s/3$. Thus, \eqref{eq:H-sHsineqsfa} holds and by H\"older's inequality, the fact that $\|\mathbb E_{\widetilde W} f\|_{L^p(\Omega)} \leq \|f\|_{L^p(\overline{\Omega})}$ for any $f\in L^p(\overline \Omega)$, and Lemma \ref{lem:ChasallLpnporemcfas} and \ref{lem:hatDhasmanyLpnofnmwioc} for any choice of $s_0\geq 3d$ 
we have that 
$\mathbb E D_{\kappa,\gamma}^p$ is finite if $\frac{s_0}{2s}\leq p< \tfrac{s_0\hat{\alpha}}{4s\hat{\gamma}}=\tfrac{3s_0 \gamma_0}{4\gamma s}$ (in particular, if $p\in[\frac{s_0}{2s},\tfrac{9d \gamma_0}{4\gamma s})$) so that by \eqref{eq:Dkappaexpasform}
\begin{align*}
 \mathbb E D_{\kappa,\gamma}^p\lesssim_{s_0} \bigl(\mathbb E \widehat {D}^{\frac{2ps}{s_0}}\bigr)^{\frac 12} \bigl(\mathbb E C^{2p}\bigr)^{\frac 12}&\lesssim_{\Sigma, \kappa_0, s, p} \frac{1}{(\tfrac{2ps}{d}-\tfrac{2ps}{s_0}-\tfrac32)\sqrt{\tfrac{\hat\alpha}{2}-\frac{2ps}{s_0}\hat \gamma}}\frac{1}{\eps^p},
\end{align*}
so $\sup_{\kappa\in[0, \kappa_0]}\mathbb E D_{\kappa, \gamma}^p<\infty$. The same can be shown for $0<p<\frac{s_0}{2s}$ by H\"older's inequality and the fact that $D_{\kappa, \gamma} = D_{\kappa, \gamma} \cdot 1$.
This concludes the proof of Theorem~\ref{thm:H-sconvergensofHsinvartext}.

\begin{remark}
Notice that if the set of $(\sigma_k)$ is finite (as e.g.\ in Section \ref{subsec:Hormtwopint}), then similarly to \eqref{eq:fgintegekviaHscaoim} the inequality \eqref{eq:H-sHsineqsfa} can be proven for any $s>0$. In this case we have $C^{\infty}$ local characteristics, and hence thanks to \cite{KSF90} a version of Lemma \ref{lem:phitkat64wppacontinxandkapps} holds true for $D^n\phi$ for any $n\geq 1$, so the proof above can be extended to any $s>0$.
\end{remark}

\begin{remark}
 Note that the proof above is still valid if one uses the following weaker version of \eqref{eq:P(2)kappapsforphitimesphi}
 $$
 \int_{\mathcal M\times \mathcal M}\bigr|\mathbb E \hat e_k(\phi^{\kappa}_{nt_0}(x))\hat e_k(\phi^{\kappa}_{nt_0}(y))\bigl|\ud \mu( x) \ud \mu( y)  \leq C k^{b} e^{-\hat{\alpha} nt_0},\;\;\; k\geq 1,
 $$
 for some fixed $b\geq 0$.
\end{remark}

\section{Stablilization by transport noise}\label{sec:Proofofmainthm}

In this section we prove Theorem \ref{thm:proofofCadpconj}.  We start with the following corollary of Theorem \ref{thm:H-sconvergensofHsinvartext}.

\begin{corollary}\label{cor:L2normocnverg}
  Let $(\sigma_k)_{k\geq 1}$ satisfy the conditions {\textnormal {\bf (A)--(C)}} and let $(u_t)_{t\geq 0}$ be the solution of \eqref{eq:mainstochdiffeqwithkappa} with $u$ being  mean-zero in $L^2(\mathcal M)$. Then there exist $\kappa_0>0$ and $\gamma_0>0$ independent of $u$ such that for any $\kappa \in (0, \kappa_0]$ and $\gamma\in(0, \gamma_0)$  there exists $D'_{\kappa,\gamma}:\Omega \to [1, \infty)$ such that a.s.
 \begin{equation}\label{eq:strongL2estforUkappa}
   \|u_t\|_{L^2} \leq D'_{\kappa,\gamma} \kappa^{-1} e^{-\gamma t} \|u\|_{L^2},\;\;\; t> 0.
 \end{equation}
 Moreover, $\sup_{\kappa\in[0, \kappa_0]}\mathbb E |D'_{\kappa,\gamma}|^p<\infty$ for any $1\leq p< \frac{9d\gamma_0}{4\gamma}$.
\end{corollary}

\begin{proof}
It follows from \eqref{eq:mainstochdiffeqwithkappa} that for any $t>0$ a.s.\
\begin{equation}\label{eq:uttimesequafforut}
\begin{split}
\|u_t\|_{L^2}^2&-\|u\|_{L^2}^2 +\sum_{k\geq 1} \int_0^t \int_{\mathcal M}\langle  u_s(x)\sigma_k(x), \nabla u_s(x)\rangle_{x}  \ud \mu(x) \circ \ud W^k_s \\
 &\quad\quad= \kappa   \int_0^t \int_{\mathcal M} u_s T u_s \ud \mu(x) \ud s = -\kappa \sum_{m=1}^n\int_0^t \|\chi_m u_s\|_{L^2(\mathcal M)}^2\ud s,
\end{split}
\end{equation}
where the integrals $ \int_{\mathcal M} \langle u_s(x)\sigma_k(x), \nabla u_s(x)\rangle_{x} \ud \mu(x)$ are well-defined as $u_s$ is an element of $H^1(\mathcal M)$ for any $s>0$ a.s.\ (see \cite[Theorem 4.2.4]{LR15}), so $\nabla u_s\in L^2(\mathcal M;T\mathcal M)$ a.s., and as $u_s(x)\sigma_k\in L^2(\mathcal M;T\mathcal M)$ for any $k\geq 1$.
The fact that $\circ \ud W^k$-integrals of \eqref{eq:uttimesequafforut} exist and are summable follows from the condition {\textnormal {\bf (C)}} {and the fact that by \cite[Theorem 4.2.4]{LR15} $u\in L^2(\Omega \times [0, T]; H^1(\mathcal M))\cap L^2(\Omega;C_{loc}(\mathbb R_+; L^2(\mathcal M))$.}

As the coefficients $\sigma_k$'s are divergence-free, by approximating $\nabla u_s$ by smooth functions in $L^2(\mathcal M;T\mathcal M)$ thanks to the divergence theorem on manifolds (see \cite[Theorem 16.32]{Le13} and \cite[Theorem 1]{Step16}; here we use that $u_s\nabla u_s = \frac 12 \nabla u_s^2$)
\[
  \sum_{k\geq 1}\int_0^t\int_{\mathcal M}\Bigl\langle  \sigma_k(x), \frac{1}2\nabla u_s^2(x)\Bigr\rangle_{x}  \ud \mu(x) \circ \ud W^k_s   = 0,
\]
so a.s.\
\begin{equation}\label{eq:u2-ut2isintofnablaut2}
  \|u\|^2_{L^2} -\|u_t\|^2_{L^2}  = \kappa \sum_{m=1}^n\int_0^t \|\chi_m u_s\|_{L^2}^2\ud s,
\end{equation}
hence by Lemma \ref{lem:Tstriellverycoopchimdfowe} $\frac{\ud \|u_t\|^2_{L^2}}{\ud t} = -\kappa \sum_{m=1}^n \|\chi_m u_t\|_{L^2}^2 \eqsim -\kappa\|\nabla u_t\|_{L^2}^2$. In particular, for $D_{\kappa,\gamma}$ and $\gamma$ being as in Theorem \ref{thm:H-sconvergensofHsinvartext} analogously to \cite[Lemma 7.1]{BBP-S1911a} we have that a.s.\ (here we use that $\int_{\mathcal M} u_t=0$ so $\|u_t\|_{H^1} \eqsim \|\nabla u_t\|_{L^2}$)
\[
\frac{\ud \|u_t\|^2_{L^2}}{\ud t} \lesssim -\kappa\|\nabla u_t\|_{L^2}^2 \leq -\kappa \frac{\|u_t\|_{L^2}^4}{\|u_t\|_{H^{-1}}^2} \leq -\kappa  e^{2\gamma t} \frac{\|u_t\|_{L^2}^4}{D_{\kappa,\gamma}^2\|u\|_{H^1}^2},
\]
thus a.s.\
\[
\frac{\ud \frac{1}{\|u_t\|^2_{L^2}}}{\ud t} = -\frac{1}{\|u_t\|^4_{L^2}}\frac{\ud \|u_t\|^2_{L^2}}{\ud t}  \gtrsim \kappa  e^{2\gamma t} \frac{1}{D_{\kappa,\gamma}^2\|u\|_{H^1}^2},
\]
consequently a.s.\
\[
\frac{1}{\|u_t\|^2_{L^2}} \geq\frac{1}{\|u_t\|^2_{L^2}} - \frac{1}{\|u\|^2_{L^2}} \gtrsim\kappa ( e^{2\gamma t} -1)\frac{1}{\gamma D_{\kappa,\gamma}^2\|u\|_{H^1}^2},
\]
which in turn implies
\begin{equation}\label{eq:L2throuwH1opfemw}
 \|u_t\|^2_{L^2} \lesssim \gamma D_{\kappa,\gamma}^2 \kappa^{-1}(e^{2\gamma t}-1)^{-1}\|u\|^2_{H^1},\;\;\; t\geq 0.
\end{equation}

\medskip

Let us now show that there exists an $\mathbb F$-stopping time $\tau\in[0, 1]$ such that $\|\nabla u_{\tau}\|_{L^2}\leq C$ for some universal constant $C>0$. To this end define $\tau$ by
\[
 \tau:=\inf\Bigl\{t\geq 0:\sum_{m=1}^n \|\chi_m u_t\|_{L^2}^2 \leq \|u\|_{L^2}^2/\kappa\Bigr \}.
\]
Note that $\tau\leq 1$ by \eqref{eq:u2-ut2isintofnablaut2}.
Moreover, $\tau$ is a stopping time as $f\mapsto \sum_{m=1}^n\|\chi_m f\|_{L^2}^2 $ is a Borel (possibly, infinite) function on $ L^2$ (it can be defined via the basis $(\hat e_k)_{k\geq 1}$, see Section \ref{sec:prelim}). Thanks to the definition of $\tau$ for a.e.\ $\omega\in \Omega$ there exists a sequence $(t_i(\omega))_{i\geq 1}$ such that $t_i(\omega)\searrow \tau(\omega)$ as $i\to \infty$ and 
$$
\sum_{m=1}^n \|\chi_m u_{t_i(\omega)}\|_{L^2}^2 \leq \|u\|_{L^2}^2/\kappa,\;\;\; i\geq 1. 
$$
In particular, by Lemma \ref{lem:Tstriellverycoopchimdfowe} there exists $C>0$ such that
\begin{equation}\label{eq:u1H1bddbyL2atsadparb}
 \frac 1C \|\nabla u_{t_i}(\omega)\|_{L^2} \leq  \Bigl(\sum_{m=1}^n \|\chi_m u_{t_i(\omega)}\|_{L^2}^2\Bigr)^{\frac 12} \leq  \frac{\|u\|_{L^2}}{\sqrt{\kappa}},\;\;\; i\geq 1.
\end{equation}
As $t\mapsto u_t(\omega)$ is a continuous $L^2$-valued function in $t\geq 0$ (see \cite[Theorem 4.2.4]{LR15}) and as $f\mapsto \|f\|_{H^1}$ is lower-semicontinuous as a function on $L^2(\mathcal M)$,  
we have that by \eqref{eq:u1H1bddbyL2atsadparb} for a.e.\ $\omega \in \Omega$
\begin{equation}\label{eq:u1H1bddbyL2atsadpfdsarb}
\| u_{\tau(\omega)}(\omega)\|_{H^1} \leq \limsup_{i\to \infty} \| u_{t_i(\omega)}(\omega)\|_{H^1} \leq C\frac{\|u\|_{L^2}}{\sqrt{\kappa}}.
\end{equation}

Finally let us show inequality \eqref{eq:strongL2estforUkappa}.
For any $f\in L^1(\overline \Omega)$ denote by $\mathbb E_1f$ and $\mathbb E_2f$ the following conditional expectations:
\[
\mathbb E_1f:= \mathbb E (f|(W^k)_{k\geq 1},(\widetilde W^m_t)_{t\in[0, r], m=1,\ldots, n}),
\]
\[
 \mathbb E_2f:= \mathbb E (f|(W^k)_{k\geq 1},(\widetilde W^m_{t+\tau} -\widetilde W^m_{\tau} )_{t\geq 0, m=1,\ldots, n}).
\]
Then $\mathbb E_{\widetilde W} = \mathbb E_1\mathbb E_2$ since 
$(\widetilde W^m_t)_{t\in[0, \tau(\omega)], m=1,\ldots, n}$ and $(\widetilde W^m_{t+\tau(\omega)} -\widetilde W^m_{\tau(\omega)} )_{t\geq 0, m=1,\ldots, n}$ are independent
for a.e.\ $\omega\in\Omega$ and since for a.e.\ $\omega\in \Omega$ $\mathbb P$-a.s.\
\[
\mathbb E_1f(\omega, \cdot) = \mathbb E(f(\omega, \cdot)|(\widetilde W^m_t)_{t\in[0, \tau(\omega)], m=1,\ldots, n}), 
\]
\[ 
 \mathbb E_2f(\omega, \cdot) = \mathbb E(f(\omega, \cdot)|(\widetilde W^m_{t+\tau(\omega)} -\widetilde W^m_{\tau(\omega)} )_{t\geq 0, m=1,\ldots, n})
\]
(these conditional expectations are well-defined as by Fubini's theorem $f(\omega, \cdot)\in L^1(\widetilde \Omega)$ for a.e.\ $\omega\in \Omega$; the latter equalities can be shown first for step functions and then extended to the whole $L^1(\overline \Omega)$ by an approximation argument).

Note that as $\tau$ is an $\mathbb F$-stopping time, Corollary \ref{cor:Fstopingtimwlidtstochdlow} yields that $(\phi^{\kappa,\tau}_t)_{t\geq 0} := (\phi^{\kappa}_{t+\tau}(\phi^{\kappa}_{\tau})^{-1})_{t\geq 0}$ is a stochastic flow of homeomorphisms independent of $(W^k_t)_{k\geq 1, t\in[0, \tau]}$ and $(\widetilde W^m_t)_{m= 1, \ldots,n, t\in[0, \tau]}$. Therefore by \eqref{eq:Utviaphikaplxw}, Corollary \ref{cor:Fstopingtimwlidtstochdlow}, and  Remark \ref{rem:stoppingrimaforformfout} for any $t\geq 0$ $\mathbb P$-a.s.\
\begin{align*}
u_{t+\tau} &= \mathbb E_{\widetilde W}u((\phi^{\kappa}_{t+\tau})^{-1}\cdot)= \mathbb E_1\mathbb E_2u\bigl((\phi^{\kappa}_{\tau})^{-1}(\phi^{\kappa}_{t+\tau}(\phi^{\kappa}_{\tau})^{-1})^{-1}\cdot\bigr) \\
&= \mathbb E_1u_{\tau}\bigl((\phi^{\kappa,\tau}_t)^{-1}\cdot\bigr)=\mathbb E_{\widetilde W}u_{\tau}\bigl((\phi^{\kappa,\tau}_t)^{-1}\cdot\bigr).
\end{align*}
Next, as $\phi^{\kappa,\tau}$ and $\phi^{\kappa}$ are equidistributed by Corollary \ref{cor:Fstopingtimwlidtstochdlow}, \eqref{eq:L2throuwH1opfemw} and \eqref{eq:u1H1bddbyL2atsadpfdsarb} yield that for any $t\geq0$ $\mathbb P$-a.s.\ 
\begin{align*}
\|u_{t+\tau}\|_{L^2}^2 &=\bigl \|\mathbb E_{\widetilde W}u_{\tau}\bigl((\phi^{\kappa,\tau}_t)^{-1}\cdot\bigr)\bigr\|_{L^2}^2\\
&\lesssim \gamma D_{\kappa,\gamma,\tau}^2 \kappa^{-1}(e^{2\gamma t}-1)^{-1}\|u_{\tau}\|_{H^1}^2 \lesssim  \gamma D_{\kappa,\gamma,\tau}^2 \kappa^{-2}e^{-2\gamma t}\|u\|_{L^2}^2,
\end{align*}
where $D_{\kappa,\gamma,\tau}$ is defined analogously to $D_{\kappa,\gamma}$ but for $\phi^{\kappa,\tau}$ instead of $\phi^{\kappa}$. Thus for any $t\geq 1$ $\mathbb P$-a.s.\
\[
\|u_{t}\|_{L^2} = \|u_{t-\tau+\tau}\|_{L^2} \lesssim\gamma D_{\kappa,\gamma,\tau}^2 \kappa^{-2}e^{-2\gamma (t-\tau)}\|u\|_{L^2}^2 \lesssim \gamma D_{\kappa,\gamma,\tau}^2 \kappa^{-2}e^{-2\gamma t}\|u\|_{L^2}^2,
\]
where the latter holds true as $\tau\leq 1$ a.s.

Therefore \eqref{eq:strongL2estforUkappa} follows with
\begin{equation}\label{eq:Dkappa'howlookliads}
 D'_{\kappa,\gamma}(\omega)\eqsim \sqrt{\gamma} C^2 D_{\kappa,\gamma,\tau}(\omega),\;\;\; \omega \in \Omega.
\end{equation}

Let us turn to the moments of $D'_{\kappa,\gamma}$. To this end note that by Theorem \ref{thm:H-sconvergensofHsinvartext},   \eqref{eq:Dkappa'howlookliads}, the fact that $(W^k_t)_{t\in[0, \tau], k\geq 1}$ is independent of $(W^k_{\cdot + \tau} - W^k_{\tau})_{k\geq 1}$, and the fact that by \eqref{eq:Dkappaexpasform} and Corollary \ref{cor:Fstopingtimwlidtstochdlow} $D_{\kappa,\gamma, \tau}$ depends only on $(W^k_{\cdot + \tau} - W^k_{\tau})_{k\geq 1}$ and has the same distribution as $D_{\kappa,\gamma}$, we have that for $p\in (0, \frac{9d\gamma_0}{4\gamma})\subset (0, \frac{3s_0\hat{\alpha}}{4\hat{\gamma}})$ (recall that $s_0\geq 3d$ and that here we assume that $s=1$)
 \begin{equation}\label{eq:Dikappaintergaoxq}
 \begin{split}
 \mathbb E |D'_{\kappa,\gamma}|^p &\eqsim \gamma^{\frac{p}{2}}C^{2p}\mathbb E |D_{\kappa,\gamma,\tau}|^p \\
 &= \gamma^{\frac{p}{2}}C^{2p} \mathbb E \mathbb E\Bigl(  |D_{\kappa,\gamma,\tau}|^p \Big| (W^k_t)_{t\in[0, \tau], k\geq 1}\Bigr)=  \gamma^{\frac{p}{2}}C^{2p}\mathbb E |D_{\kappa,\gamma}|^p<\infty.
 \end{split}
 \end{equation}
\end{proof}

Let us now show Theorem \ref{thm:proofofCadpconj}. For $C=0$ Theorem \ref{thm:proofofCadpconj} follows from Corollary \ref{cor:L2normocnverg} by applying a rescaling argument. Namely, { let $A_0 := 1/\sqrt{\kappa_0}$. Then for any $A\geq A_0$ and} for $\kappa:= 1/A^2$ we have that for the solutions $(u^{A}_t)_{t \geq 0}$ of \eqref{eq:mainstochdiffeq} and $(u_{t})_{t\geq 0}$ of \eqref{eq:introkappaisnmwq} the processes $(u^{A}_t)_{t \geq 0}$ and $(u_{A^2t})_{t\geq 0}$ are equidistributed, so \eqref{eq:INTROmainthmwithCiqo} is a direct consequence of \eqref{eq:strongL2estforUkappa} with $D^{A,\gamma}(\omega) := D_{1/A^2,\gamma}'(\omega(A^2\cdot))$ (recall that we set $\Omega = C(\mathbb R_+;\mathbb R^\infty)$, see Section \ref{sec:prelim}). { The integrability of $D^{A,\gamma}$ then follows from \eqref{eq:Dikappaintergaoxq}.}

\smallskip

The case $C>0$ follows from the case $C=0$ by applying Theorem \ref{thm:proofofCadpconj} to $v_t := e^{-Ct}u_t$ which in turn satisfies \eqref{eq:mainstochdiffeq} with $C=0$.
 
\begin{remark}\label{rem:wherewhatisneeded}
 This remark provides an overview of the steps of the proof, highlighting where the assumptions are used:
 
 \begin{enumerate}[(i)]
  \item For exponential mixing \eqref{eq:INTROH-sHsineqsfa} and enhanced dissipation \eqref{eq:INTROmainthmwithCiqo} it is sufficient to assume the condition {\textnormal {\bf (C)}} together with the exponential ergodicity of the two-point motion  \eqref{eq:P(2)kappapsforphitimesphi} (see Subsection \ref{subsec:proofofThmmain} and the present Section \ref{sec:Proofofmainthm}).
  \item The two-point exponential ergodicity \eqref{eq:P(2)kappapsforphitimesphi} follows from on Harris' Theorem \ref{thm:QuaeHarwqthnms}, assuming $\kappa$-independent Harnack inequalities, and the existence of a Lyapunov function for the two-point motion (see proof of Proposition \ref{prop:harridsthmforsVkappa}). 
  \item In order to show the existence of a density and $\kappa$-independent Harnack inequalities for the two-point motion we exploit the corresponding non-degeneracy provided by the ellipticity condition {\textnormal {\bf (A)}} (see Proposition \ref{prop:appendxHarnackmancewell} and the proof of Proposition \ref{prop:harridsthmforsVkappa}).
  \item The existence of a Lyapunov function for the two-point motion is shown by using the positivity of the Lyapunov exponent (see Proposition \ref{prop:Lyapconstexistsandpos}), the spectral gap of the normalized tangent flow, and sufficient regularity of the eigenfunction of the normalized tangent flow (see Subsection \ref{subsec:constrofVkappa}).     
  \begin{enumerate}[(a)]
		  \item The Lyapunov exponent is shown to be positive due to the existence of densities of the two-point motion and the normalized tangent flow and the criteria given by \cite[Theorem 6.8]{Bax89} (see Proposition \ref{prop:Lyapconstexistsandpos}).
		  \item The spectral gap of the normalized tangent flow can be shown via the corresponding Harnack inequalities and Harris' ergodic Theorem \ref{thm:QuaeHarwqthnms} (see Lemma \ref{lem:Pt0hasspectgap} and the results thereafter).
		  \item Finally, the existence of a density and the Harnack inequalities for the normalized tangent flow, and the sufficient regularity of the corresponding eigenfunction 
           in turn follow from the non-degeneracy which in our case is given by the ellipticity condition {\textnormal {\bf (B)}} (see Proposition \ref{prop:Lambda(p)defandpsiprequxwqxwq} and Proposition \ref{prop:appendxHarnackmancewell}). 
   \end{enumerate}
 \end{enumerate}
\end{remark}

\section{Regular Kraichnan model on $\mathbb T^d$}\label{sec:Kraichnan}

The {\em Kraichnan model} (a.k.a.\ the {\em RDT model} for ``Rapid Decorrelation in Time'') was introduced by Kraichnan in \cite{Kr68} and by Kazantsev in \cite{Kaz68} in the context of turbulence in fluids. We also refer to \cite{MK99} for further details. The classical Kraichnan model defines an {\em isotropic flows} on $\mathbb R^d$ (or $\mathbb S^d$), while the Kraichnan model on the $d$-dimensional torus $\mathcal M = \mathbb T^d$ was introduced in  \cite[Section 3]{ChDG} as follows. For any $x, y\in \mathbb T^d$ we set, for $i, j=1, \ldots, d$,
\begin{equation}\label{eq:genformofDxyfa}
\begin{split}
D(x, y)(e_i, e_j) &= \sum_{k\geq 1} \langle \sigma_k(x), e_i\rangle \langle \sigma_k(y), e_j\rangle  \\
& = \sum_{z\in \mathbb Z^d_0}\Bigl[ (1-\wp) \delta_{ij} - (1-\wp d) \frac{z_i z_j}{|z|^2}\Bigr] e^{i z \cdot (x-y)}\hat d(|z|),
\end{split}
\end{equation}
where the parameter $\wp$ measures the compressibility of the flow (see e.g.\ \cite[Section 2]{GV00}). Since we are considering incompressible flows, corresponding to $\sigma_k$ being divergence-free (see e.g.\ \cite[\S 10]{LL596}), we have $\wp=0$. 
 {Following \cite[pp.\ 343, 426, and 432]{MK99} the coefficient $\hat d(|z|)$ is chosen as $\hat d(|z|) \eqsim \frac{1}{|z|^{d+\alpha}}$ for some $\alpha>2$.}

In terms of the coefficients $\sigma_k$ this choice correlation function $D(x,y)$ corresponds to choosing $(\sigma_k)_{k\geq 1}$ as $(\hat e_{z}^{\ell})_{z\in\mathbb Z^d_0,  1 \leq \ell \leq d-1}$, where for any $x\in \mathbb T^d$ and for any pair $(z, -z)$ from $\mathbb Z^d_0$ with $z$ lexicographically dominating $-z$
\begin{equation}\label{eq:sigmakforisotonTdalpha>2}
\begin{split}
\hat e_{z}^{\ell}(x) &= a^{\ell}_z\frac{1}{\sqrt{2}|z|^{\frac{d+\alpha}{2}}}\cos(z\cdot x),\;\;\; x\in \mathbb T^d,\\
 \hat e_{-z}^{\ell}(x) &= a^{\ell}_z\frac{1}{\sqrt{2} |z|^{\frac{d+\alpha}{2}}}\sin(z\cdot x),\;\;\; x\in \mathbb T^d,
\end{split}
\end{equation}
with $(a^{\ell}_z)_{ 1 \leq \ell \leq d-1} \subset \mathbb R^d$ being an orthonormal basis of the orthogonal complement to $z$. Indeed, in this case for any $x, y\in \mathbb T^d$ we have
\begin{align*}
&D(x, y)(e_i, e_j) = \sum_{k\geq 1} \langle \sigma_k(x), e_i\rangle \langle \sigma_k(y), e_j\rangle  \\
&\quad\quad\quad=\frac{1}{4}\sum_{z\in \mathbb Z^d_0} \frac{1}{|z|^{d+\alpha}} \bigl[\cos(z\cdot x)\cos(z\cdot y) + \sin(z\cdot y)\sin(z\cdot y) \bigr] \sum_{\ell = 1}^{d-1} \langle a^{\ell}_z, e_i\rangle\langle a^{\ell}_z, e_j\rangle\\
&\quad\quad\quad \stackrel{(*)}=\frac{1}{4}\sum_{z\in \mathbb Z^d_0} \frac{1}{|z|^{d+\alpha}}  \cos\bigl(z\cdot (x-y)\bigr)\Bigl[  \langle e_i, e_j\rangle -  \frac{ \langle z, e_i\rangle \langle z, e_j\rangle}{|z|^2} \Bigr] \\
&\quad\quad\quad =  \sum_{z\in \mathbb Z^d_0}\Bigl[  \delta_{ij} - \frac{z_i z_j}{|z|^2}\Bigr] e^{i z \cdot (x-y)}\frac{1}{8|z|^{d+\alpha}},
\end{align*}
where $(*)$ follows from the fact that $(a^1_z, \ldots, a^{d-1}_z, z/|z|)$ forms an orthonormal basis of $\mathbb R^d$. Thus, we obtain \eqref{eq:genformofDxyfa} with $\hat d(|z|) = \frac{1}{8|z|^{d+\alpha}}$ as required.

We further notice that,
\begin{equation}\label{eq:isotcaseDsigmakcoinswithsigmak}
\langle D \sigma_k, \sigma_k \rangle = 0,\;\;\;\;\; k\geq 1,
\end{equation} 
since 
\begin{equation}\label{eq:ezellDezellleqcam0cz}
\begin{split}
\langle D\hat e_{z}^{\ell}, \hat e_{z}^{\ell} \rangle &= a^{\ell}_z\langle a^{\ell}_z, z\rangle\frac{1}{2 |z|^{d+\alpha}}\sin(z\cdot x)\cos(z\cdot x) = 0,\\
\langle D\hat e_{-z}^{\ell}, \hat e_{-z}^{\ell} \rangle &=- a^{\ell}_z\langle a^{\ell}_z, z\rangle\frac{1}{2 |z|^{d+\alpha}}\sin(z\cdot x)\cos(z\cdot x) = 0.
\end{split}
\end{equation}
This implies that the Stratonovich SDE \eqref{eq:mainstochdiffeq} coincides with the It\^o one by  \cite[Theorem V.26]{Prot}. Note that \eqref{eq:isotcaseDsigmakcoinswithsigmak} holds for general isotropic flows on $\mathbb R^d$ and $\mathbb S^d$, see \cite[(3.7)]{BH86} and compare Remark \ref{rmk:Kraichnan-isotropy}.

\begin{remark}[(An-)isotropy in the Kraichnan model]\label{rmk:Kraichnan-isotropy}
In $\mathbb R^d$ (or $\mathbb S^d$)  the Kraichnan model is known to be governed by {\em isotropic flows}, i.e.\ stochastic flows defined analogously \eqref{eq:flowforkappa=0dsa} on $\mathbb R^d$ (resp.\ $\mathbb S^d$) with $D(x, y) := \sum_{k\geq 1}\sigma_k(x)\otimes\sigma_k(y)$ depending only on $|x-y|$, see, for example \cite[Section 4.2.2]{MK99} and \cite{LJR02,BH86,Yag57,Rai99,Yag87I}. For the Kraichnan model over $\mathbb T^d$ (the {\em box case}) the phenomenon of anisotropy appears (see \cite{ChDG,CFG,LJ84}). Indeed, any isotropic flow on the torus is necessarily trivial: If $D(x, y)$ depends only on $|x-y|$ then it is invariant under rotations. Now applying rotations with center $\{0\} \in \mathbb T^d$ any point on the torus can be moved arbitrarily close to $\{0\}$. Hence, continuity of  $D(x, y)$ implies that $D$ has to be a constant.
\end{remark}

\begin{theorem}
  Consider the Kraichnan model on the $d$-dimensional torus with spatial decorrelation parameter $\alpha>2$, that is, the stochastic transport equation
  \begin{equation}\label{eq:kraichnan}
   \ud u_t + A\sum_{k\geq 1} \langle \sigma_k, \nabla u_t\rangle_{T\mathcal M} \circ \ud W^k_t  = 0,
  \end{equation}   
  with $\sigma_k$ chosen as in \eqref{eq:sigmakforisotonTdalpha>2}. Then, \eqref{eq:kraichnan} is  exponentially mixing, in the sense of Theorem \ref{thm:proofofCadpconj}.
   Moreover, the viscous Kraichnan model   \begin{equation}\label{eq:kraichnan_viscous}
      \ud u_t + A\sum_{k\geq 1} \langle \sigma_k, \nabla u_t\rangle_{T\mathcal M} \circ \ud W^k_t  = \Delta u_t \ud t,\
     \end{equation}   
   satisfies enhanced diffusion, in the sense of Theorem \ref{thm:H-sconvergensofHsinvartext}.
\end{theorem}
\begin{proof}
The proof consists in verifying the assumptions of Theorem \ref{thm:proofofCadpconj} for the Kraichnan model, that is, the conditions {\textnormal {\bf (A)--(C)}} from Section \ref{sec:intro}.

{\textnormal {\bf (A)}:} Fix $x \neq y \in \mathbb T^d$. Fix $z\in \mathbb Z^d_0$ lexicographically dominating $-z$ such that
\begin{equation}\label{eq:goodzforellirpek(B)two-dklpw}
  \langle z, x-y\rangle \neq 2\pi n,\;\;\; \forall n\in \mathbb Z.
\end{equation}
Note that in this case for any $\ell$
\[
 \left(
\begin{array}{c}
\hat{e}_z^{\ell}(x)\\
\hat{e}_z^{\ell}(y)\\
\end{array}
\right)
=
 \tfrac{1}{\sqrt 2 |z|^{\frac{d+\alpha}{2}}}\left(
\begin{array}{c}
a^{\ell}_z \cos(z\cdot x)\\
a^{\ell}_z \cos(z\cdot y)\\
\end{array}
\right)
\;\; \text{and}\;\;
 \left(
\begin{array}{c}
\hat{e}_{-z}^{\ell}(x)\\
\hat{e}_{-z}^{\ell}(y)\\
\end{array}
\right)
=
 \tfrac{1}{\sqrt 2 |z|^{\frac{d+\alpha}{2}}}\left(
\begin{array}{c}
a^{\ell}_z \sin(z\cdot x)\\
a^{\ell}_z \sin(z\cdot y)\\
\end{array}
\right)
\]
generate $(a a^{\ell}_z, b a^{\ell}_z)$ for any $a, b\in \mathbb R$ as then $(\sin(z\cdot x),\sin(z\cdot y))$ and $(\cos(z\cdot x),\cos(z\cdot y))$ are not collinear. Since $\ell$ was arbitrary, we are able to generate $z^{\perp}\times z^{\perp}\subset \mathbb R^{2d}$ (where $z^{\perp}$ is the orthogonal complement to $z$ in $\mathbb R^d$). It remains to notice that such $z$ satisfying \eqref{eq:goodzforellirpek(B)two-dklpw} span the whole $\mathbb R^d$ as the set 
$$
\Bigl\{\frac{z}{|z|}:z\in \mathbb Z^d_0 \;\;\; \textnormal{such that \eqref{eq:goodzforellirpek(B)two-dklpw} is satisfied} \Bigr\}\subset\mathbb R^d
$$ 
is dense in the unit sphere of $\mathbb R^d$ (if $z$ does not satisfy \eqref{eq:goodzforellirpek(B)two-dklpw}, then a small perturbation of $Cz$ for big $C\in \mathbb N$ does). Therefore $(\sigma_k(x), \sigma_k(y))_{k\geq 1}$ span $\mathbb R^{2d}$ and hence {\textnormal {\bf (A)}} is satisfied.

{\textnormal {\bf (B)}:} To this end let us actually write down what are $\tilde{\sigma}_k$'s. By \eqref{eq:equfortildephikappa} we have that $(\tilde{\sigma}_k)_{k\geq 1}$ coincides with $(\tilde e_{z}^{\ell})_{z\in\mathbb Z^d_0,  1 \leq \ell \leq d-1}$, where for any $x\in \mathbb T^d$, for any $v\in \mathbb S^{d-1}$, and for any pair $(z, -z)$ from $\mathbb Z^d_0$ with $z$ lexicographically dominating $-z$
\begin{equation*}\label{eq:sigwrmakforisotonTdtwalpha>2}
\begin{split}
\tilde e_{z}^{\ell}(x,v) &= \frac{1}{\sqrt{2}|z|^{\frac{d+\alpha}{2}}}\sin(z\cdot x)\langle z, v\rangle\Bigl[ a^{\ell}_z - v\langle v, a^{\ell}_z\rangle \Bigr],\;\;\; x\in \mathbb T^d,\;\;v\in \mathbb S^{d-1},\\
 \tilde e_{-z}^{\ell}(x,v) &=\frac{1}{\sqrt{2} |z|^{\frac{d+\alpha}{2}}}\cos(z\cdot x)\langle z, v\rangle\Bigl[ v\langle v, a^{\ell}_z\rangle  - a^{\ell}_z \Bigr],\;\;\; x\in \mathbb T^d,\;\;v\in \mathbb S^{d-1}.
\end{split}
\end{equation*}
Fix $x\in \mathbb T^d$ and $v\in \mathbb S^{d-1}$.
We need to show that for any $u\in \mathbb R^d$ and $w\in v^{\perp}$ satisfying $|u|^2 + |w|^2=1$ there exists $z\in \mathbb Z^d_0$ and $\ell\in\{1, \ldots, d-1\}$ such that $\langle e_{z}^{\ell}(x), u\rangle + \langle \tilde e_{z}^{\ell}(x,v), w\rangle \neq 0$. This follows from the fact that if $z$ dominates lexicographically $-z$, then
\[
\langle e_{z}^{\ell}(x), u\rangle + \langle \tilde e_{z}^{\ell}(x,v), w\rangle = \tfrac{1}{\sqrt{2} |z|^{\frac{d+\alpha}{2}}}\bigl[ \cos(z\cdot x)\langle a^{\ell}_z, u\rangle + \sin(z\cdot x)\langle z, v\rangle \langle a^{\ell}_z, w\rangle \bigr],
\]
\[
\langle e_{-z}^{\ell}(x), u\rangle + \langle \tilde e_{z}^{\ell}(x,v), w\rangle = \tfrac{1}{\sqrt{2} |z|^{\frac{d+\alpha}{2}}}\bigl[ \sin(z\cdot x) \langle a^{\ell}_z, u\rangle-\cos(z\cdot x)\langle z, v\rangle \langle a^{\ell}_z, w\rangle \bigr].
\]
Linear combinations of these numbers generate both $\langle a^{\ell}_z, u\rangle$ and $\langle z, v\rangle \langle a^{\ell}_z, w\rangle$. So if $u\neq 0$, then it is sufficient to choose some $z$ so that $\langle a^{\ell}_z, u\rangle \neq 0$. If $u=0$, then $w\neq 0$, and we can choose $z$  so that $ \langle a^{\ell}_z, w\rangle \neq 0$ and $ \langle z, v\rangle \neq 0$.

{\bf (C)}: First note that by \eqref{eq:isotcaseDsigmakcoinswithsigmak} we have that $\sum_{k\geq 1}\langle D \sigma_k, \sigma_k \rangle = 0 $, so this function is in $C^{1, \beta}$ for some $\beta\in (0,1]$. Next notice that
\[
\|\hat e_{z}^{\ell}\|_{\infty} \eqsim \frac{1}{|z|^{\frac{d+\alpha}2}},\;\; \;\;\|D\hat e_{z}^{\ell}\|_{\infty} \eqsim \frac{1}{|z|^{\frac{d+\alpha}2-1}},\;\; \;\;\|D^2\hat e_{z}^{\ell}\|_{\infty} \eqsim \frac{1}{|z|^{\frac{d+\alpha}2-2}},
\]
hence
\begin{equation}\label{eq:Dsifma2issummableforalpha>2}
\begin{split}
\sum_{k\geq 1}\|\sigma_k\|_{\infty}^2 +  \|D \sigma_k\|_{\infty}^2 &+ \|D \sigma_k\|_{\infty} \|\sigma_k\|_{\infty}+ \|D^2 \sigma_k\|_{\infty} \|\sigma_k\|_{\infty}\\
&\quad\quad\quad\eqsim_d \sum_{z\in\mathbb Z^d_0} \frac{1}{|z|^{d+\alpha-2}} \eqsim_d \sum_{n=1}^{\infty} \sum_{z\in\mathbb Z^d_0, |z|_{\infty}=n} \frac{1}{n^{d+\alpha-2}}\\
&\quad\quad\quad \eqsim_d \sum_{n=1}^{\infty} \frac{n^{d-1}}{n^{d+\alpha-2}}  = \sum_{n=1}^{\infty} \frac{1}{n^{\alpha-1}} <\infty,
\end{split}
\end{equation}
where $|z|_{\infty}$ is the standard sup-norm in $\mathbb Z^d$. Therefore \eqref{eq:nescondonsigmaintermsofsums} holds true as well. Finally, for any $x, y\in \mathbb T^d$ and any $u, v\in \mathbb R^d$ we have that
\begin{equation}\label{eq:DsigmaDsigmainCbetafroKraichsd}
\begin{split}
 &\sum_{k\geq 1}  D \sigma_k(x)\otimes D \sigma_k(y) (u, v) \\
 &\quad\quad\quad\quad= \frac{1}{4}\sum_{z\in \mathbb Z_0^d} \frac{1}{|z|^{d+\alpha }} \sum_{\ell=1} ^{d-1} a_z^{\ell} \otimes a_z^{\ell}\cos\bigl(z\cdot(x-y) \bigr)\langle z, u\rangle \langle z, v\rangle,
 \end{split}
\end{equation}
 so $x, y \mapsto \sum_{k\geq 1}  D \sigma_k(x)\otimes D \sigma_k(y)$ is $C^{\beta}$ for $\beta<\alpha-2$ (since $\cos a - \cos b \leq |a-b|^{\beta}$ for any $0\leq \beta \leq 1$). Similar calculations lead to the last part of the condition {\textnormal {\bf (C)}}.
 Indeed, 
 \begin{align*}
  &\sum_{k\geq 1}\|(\sigma_k(y) - \sigma_k(x)) \otimes (\sigma_k(y) - \sigma_k(x)) - \langle D \sigma_k(x), y-x \rangle \otimes  \langle D \sigma_k(x), y-x \rangle\|\\
 &\quad\quad\quad = \sum_{z\in \mathbb Z^d_0} \frac{1}{2|z|^{d+\alpha}} \sum_{\ell=1}^{d-1} \|a^{\ell}_z \otimes a^{\ell}_z\| \bigl( 2-2\cos(z\cdot(y-x)) - \langle z, y-x\rangle^2 \bigr)\\
 &\quad\quad\quad = \sum_{z\in \mathbb Z^d_0} \frac{d-1}{2|z|^{d+\alpha}} \bigl( 2-2\cos(z\cdot(y-x)) - \langle z, y-x\rangle^2 \bigr);
 \end{align*}
the latter is of the order $O(|w|^{\alpha-\eps})$ for any $0<\eps<\alpha-2$ as it is summable for such $\eps$ and as $2-2\cos(a)-a^2 = O(|a|^{\alpha-\eps})$ for any $a$. \eqref{eq:qvofsigma-goan-Dfopsaxw)wfq} holds for the similar reason.
 
 Thus {\textnormal {\bf (C)}} follows and consequently, Theorem \ref{thm:proofofCadpconj} holds true.
\end{proof}

 \begin{remark}\label{rem:alpha=4/3inreal}
 According to \cite[pp.\ 427 and 436]{MK99} the Kolmogorov spectrum of turbulence corresponds to $\alpha = 4/3$, which in turn leads to coefficients $\sigma_k$ which do not have enough regularity in order to generate a stochastic flow and, therefore, the approach of this paper cannot be applied.
 \end{remark}

\begin{remark}\label{rem:whyalpha>2needed}
The reader might question why do we consider such a range of $\alpha$, why cannot we go beyond $\alpha = 2$?  If we consider $\alpha\leq 2$, we are no longer guaranteed that $\phi$ creates a flow and that in particular $D\phi$ is well defined. Indeed, in this case analogously to \eqref{eq:Dsifma2issummableforalpha>2} we get that the local characteristics of $\phi$ are no longer differentiable but only H\"older continuous, so we can not apply \cite[Section 4.5 and Theorem 4.6.5]{KSF90}.
\end{remark}

\section{H\"ormander conditions}\label{subsec:Hormtwopint}

In this section we show that, in the case of smooth diffusion coefficients, the ellipticity conditions $\textbf{(A)}$ for the two-point motion and $\textbf{(B)}$ for the normalized tangent flow can be relaxed to the H\"ormander condition, as it was also done in the classical work \cite{BS88} by Baxendale and Stroock (see also \cite{Bax91,DKK04}). 
\begin{enumerate}
 \item[\bf (A')] There is a $K>0$ such that  $\sigma_k \equiv 0$ for all $k> K$ and that for all $x_1, x_2\in \mathcal M$ such that $x_1\neq x_2$ we have that
\[
{\rm Lie}\left(\binom{\sigma_1}{ \sigma_1}, \ldots ,\binom{\sigma_K}{\sigma_K}\right)\binom{x_1}{x_2}  = T_{x_1}\mathcal M\times T_{x_2}\mathcal M.
\]
 \item[\bf (B')] There is a $K>0$ such that  $\sigma_k \equiv 0$ for all $k> K$ and that for all $x\in \mathcal M$ and $v\in S_x\mathcal M$ we have that
\[
{\rm Lie}\left(\binom{\sigma_1}{\tilde \sigma_1}, \ldots ,\binom{\sigma_K}{\tilde \sigma_K}\right)\binom{x}{v}  = T_x\mathcal M \times T_v(S_x\mathcal M).
\]
\end{enumerate}

\begin{theorem}\label{thm:hoermander}
 Let $(\sigma_k)_{k\geq 1}$ satisfy the conditions {\textnormal {\bf (A'),(B'),(C)}}. Then the results of  Theorem \ref{thm:proofofCadpconj} and Theorem \ref{thm:H-sconvergensofHsinvar} hold.
\end{theorem}
\begin{proof}
{ First, one can impose the condition {\textnormal {\bf (B')}} instead of {\textnormal {\bf (B)}} in Subsection \ref{sec:LyaexmomLyafun} as it was done e.g.\ in \cite{BS88}. Next, Proposition \ref{prop:harridsthmforsVkappa} with  the conditions {\textnormal {\bf (A')}} and  {\textnormal {\bf (B')}} instead of {\textnormal {\bf (A)}} and {\textnormal {\bf (B)}} was shown on \cite[p.\ 9]{DKK04}.}
The proof then follows the lines of the proofs of Theorem \ref{thm:H-sconvergensofHsinvartext} and Theorem \ref{thm:proofofCadpconj}. Since the sequence $(\sigma_k)$ is finite and each $\sigma_k$ is a $C^{\infty}$ function, the flow $\phi$ in the proof of Lemma \ref{lem:phitkat64wppacontinxandkapps} has smooth local characteristics. Hence, in particular, similarly to the proof of Lemma \ref{lem:phitkat64wppacontinxandkapps} it follows that for any $n\geq 1$ the $n$-th derivative $D^{(n)}\phi^{\kappa}$ exists and has all moments finite, bounded by a constant independent of $\kappa\in[0, \kappa_0]$ (but depending on $n$). Therefore, the analogues of Proposition \ref{prop:P(2)kappasminopckascbyVpVpdla} and Corollary \ref{cor:existnoft*dxc} follow. 
\end{proof}

A notable variation of Kraichnan model was presented by Baxendale and Rozovskii in \cite[pp.\  57--58]{BR92}. We next show that the results of the present section are applicable to this example.

\begin{theorem}\label{ex:BRMHDmodel}
Let $d=2$ and for $x=(x^1, x^2)\in \mathbb T^2$
\[
\sigma_1(x) = \binom{0}{\sin x^1},\;\;\;\sigma_2(x) = \binom{0}{\cos x^1},\;\;\;\sigma_3(x) = \binom{\sin x^2}{0},\;\;\;\sigma_4(x) = \binom{\cos x^2}{0}.
\]
Then, the stochastic flow is diffusion enhancing in the sense of Theorem \ref{thm:proofofCadpconj} and exponentially mixing in the sense of Theorem \ref{thm:H-sconvergensofHsinvar}.
\end{theorem}

\begin{proof}
In this case  the condition {\textnormal {\bf (B')}} is satisfied by \cite[p.\ 58]{BR92} and {\textnormal {\bf (C)}} is obviously true, but condition {\textnormal {\bf (A)}} does not hold. Indeed, for any fixed $x, y\in \mathbb T^2$ with $x=(x^1, x^2)$ and $y=(y^1, y^2)$ one has that
\begin{equation}\label{eq:sigm1sigm2twopcao}
\binom{\sigma_1(x)}{\sigma_1(y)} = \begin{pmatrix} 0 \\ \sin x^1 \\ 0\\ \sin y^1 \end{pmatrix},\;\; 
\binom{\sigma_2(x)}{\sigma_2(y)} = \begin{pmatrix} 0 \\ \cos x^1 \\ 0\\ \cos y^1 \end{pmatrix},
\end{equation}

\begin{equation}\label{eq:sigm3sigm4twopcao}
\binom{\sigma_3(x)}{\sigma_3(y)} = \begin{pmatrix} \sin x^2 \\ 0\\ \sin y^2\\0 \end{pmatrix},\;\; 
\binom{\sigma_4(x)}{\sigma_4(y)} = \begin{pmatrix}  \cos x^2 \\ 0\\ \cos y^2\\ 0 \end{pmatrix},
\end{equation}
which span $\mathbb R^4$ if and only if
\begin{equation}\label{eq:badxydq,BaxRoz}
x^1\ne y^1+\pi k\;\; \text{or}\;\; x^2 \ne y^2+\pi k\;\; \text{for all}\;\; k\in \{0, 1\}.
\end{equation}
Indeed, in this case both \eqref{eq:sigm1sigm2twopcao} and \eqref{eq:sigm3sigm4twopcao} span two-dimensional subspaces of $\mathbb R^4$ which are orthogonal to each other.
Hence, the strict ellipticity condition {\textnormal {\bf (A)}} for the two point motion is not satisfied. 

However, we next show that {\textnormal {\bf (A')}} can be verified. To this end fix $x$ and $y$ so that \eqref{eq:badxydq,BaxRoz} is satisfied and let us exemplarily compute one of the two-point Lie brackets
\begin{align*}
\left [ \binom{\sigma_1(x)}{\sigma_1(y)}, \binom{\sigma_3(x)}{\sigma_3(y)}\right ] &= \left \langle D \binom{\sigma_1(x)}{\sigma_1(y)}, \binom{\sigma_3(x)}{\sigma_3(y)} \right \rangle - \left \langle D \binom{\sigma_3(x)}{\sigma_3(y)}, \binom{\sigma_1(x)}{\sigma_1(y)} \right \rangle\\
&=  \begin{pmatrix}  \sin x^1\cos x^2 \\ -\cos x^1 \sin x^2\\ \sin y^1 \cos y^2\\ -\cos y^1 \sin y^2 \end{pmatrix}.
\end{align*}
In a similar way, the first-order Lie brackets generate the vectors
\[
\begin{pmatrix} \cos x^1 \cos x^2 \\ \sin x^1 \sin x^2\\ \cos y^1 \cos y^2\\ \sin y^1 \sin y^2 \end{pmatrix},\;\; \begin{pmatrix}  \sin x^1 \sin x^2\\\cos x^1 \cos x^2 \\ \sin y^1 \sin y^2\\ \cos y^1 \cos y^2 \end{pmatrix}, \;\; \text{and} \;\;  \begin{pmatrix} \cos x^1 \sin x^2\\ - \sin x^1\cos x^2 \\ \cos y^1 \sin y^2\\-\sin y^1 \cos y^2\end{pmatrix}.
\]
Together with \eqref{eq:sigm1sigm2twopcao} and \eqref{eq:sigm3sigm4twopcao} these vectors span $\mathbb R^4$ unless $x^1=y^1$ and $x^2=y^2$, i.e.\ unless $x=y$.
Therefore, the desired result follows from Theorem \ref{thm:hoermander}.
\end{proof}

\section{Acknowledgements} 
The authors acknowledge support by the Max Planck Society through the Max Planck Research Group
\textit{Stochastic partial differential equations}. This work was funded by the Deutsche Forschungsgemeinschaft (DFG, German Research Foundation) -- SFB 1283/2 2021 -- 317210226.

\appendix

\section{Heat kernel regularity and positivity on manifolds}

The goal of this section is to show technical statements as Harnack inequalities and H\"older continuity for the probability kernels corresponding to SDEs on manifolds with non-smooth coefficients (e.g.\ as in \eqref{eq:flowforkappa=0dsa}), which are manifold versions of those presented in \cite{BKRSh}. We will write our statements in the tensor form, i.e.\ in the form independent of local coordinates (we refer the reader to \cite[Chapter 4]{SpVI99}). For example, we assume that the coefficients $a^{ij}$ and $b^i$ depend on local coordinates in the way described on \cite[p.\ 120]{SpVI99} and are in fact $(2,0)$- and $(1, 0)$-tensor fields (i.e.\ a bilinear form on $(T\mathcal M)^* \times (T\mathcal M)^*$ and a tangent vector fields respectively).

Throughout the section we assume that $\mathcal M$ is a {general} Riemannian manifold with or without boundary. In particular, by writing $\phi\in C^{\infty}_0(\mathcal M)$ we mean that $\phi$ is $C^{\infty}$-smooth with a compact support lying in the interior of $\mathcal M$.

In the paper we need density existence, positivity, and regularity for fundamental solutions of certain PDEs, which can be reached via the following propositions. We start with the following one requiring only H\"older continuity from the coefficients.

\begin{proposition}\label{prop:appendtkernelHoldcoegg}
Let $\mathcal M$ be a $C^{\infty}$-smooth $d$-dimensional Riemannian manifold, let $\mu$ be the corresponding volume measure, $\alpha\in(0,1]$, $(a^{i j})_{i, j=1}^d$ be a symmetric positive-definite $C^{\alpha}$-continuous $(2, 0)$-tensor field, $b$ be a tangent $C^{\alpha}$-continuous vector field, $c\in C(\mathcal M)$,  and let $\nu$ be a signed locally finite Borel measure on $(0, \infty)\times \mathcal M$ satisfying
\begin{equation}\label{eq:appenexistskernelHoldcoeff}
\int_{\mathbb R_+\times\mathcal M} \partial_t\phi + a^{ij} \frac{\partial^2\phi}{ \partial x^i\partial x^j} + b^i \frac{\partial \phi}{\partial x^i} + c\phi \ud \nu = 0,\;\;\; \phi\in C^{\infty}_0(\mathbb R_+ \times \mathcal M).
\end{equation}
Then $\nu$ has a density $\rho$ with respect to $\lambda_{(0, \infty)}\otimes \mu$ on $(0, \infty)\times \mathcal M$.
\end{proposition}

\begin{proof}
The proposition follows from localizing \eqref{eq:appenexistskernelHoldcoeff} and from \cite[Theorem 6.3.1]{BKRSh}.
\end{proof}

We also need continuity for the kernel $\rho$ given more regularity of $a^{ij}$ which can be shown as follows.

\begin{proposition}\label{prop:appendcontkernel}
Let $\mathcal M$ be a $C^{\infty}$-smooth $d$-dimensional Riemannian manifold, let $\mu$ be the corresponding volume measure, $\alpha\in(0,1]$, $(a^{i j})_{i, j=1}^d$ be a symmetric positive-definite $C^{1+\alpha}$-continuous $(2, 0)$-tensor field, $b$ be a tangent $C^{\alpha}$-continuous vector field, $c\in C(\mathcal M)$,  and let $\nu$ be a signed locally finite measure on $(0, \infty)\times \mathcal M$ satisfying
\begin{equation}\label{eq:appenforPhatpeqo}
\int_{\mathbb R_+\times\mathcal M} \partial_t\phi + \frac {\partial}{\partial x^i} a^{ij}\frac{\partial\phi}{ \partial x^j} + b^i \frac{\partial \phi}{\partial x^i} + c\phi \ud \nu = 0,\;\;\; \phi\in C^{\infty}_0(\mathbb R_+ \times \mathcal M).
\end{equation}
Then $\nu$ has a density $\rho$ with respect to $\lambda_{(0, \infty)}\otimes \mu$ on $(0, \infty)\times \mathcal M$ which is locally H\"older continuous.
\end{proposition}

{

We assume that $a^{ij}$ is $C^{1+\alpha}$-continuous as in \cite[Section 6.4]{BKRSh} applied below one needs $a^{ij}$ to be locally $W^{p,1}$ which holds in our case.

}

\begin{proof}[Proof of Proposition \ref{prop:appendcontkernel}]
First we can localize \eqref{eq:appenforPhatpeqo}, i.e.\ choose a local chart $U\subset \mathcal M$ such that $\overline U\subset \mathcal M$, open set $V\subset \mathbb R^d$, and $C^{\infty}$ local coordinates $(x^1, \ldots,x^d)=i_U:U\to V$ so that
\begin{equation*}
\int_{\mathbb R_+\times V} \partial_t\phi + \frac {\partial}{\partial x^i} a^{ij}\frac{\partial \phi}{ \partial x^j} + b^i \frac{\partial \phi}{\partial x^i} + c\phi \ud \nu(t, i_U^{-1} \cdot) = 0,\;\;\; \phi\in C^{\infty}_0(\mathbb R_+ \times V).
\end{equation*}
Then as $\overline U\subset \mathcal M$ and as $a^{ij}$ is continuous, $A=(a^{ij})_{i, j=1}^d$ is strictly elliptic on $V$ together with its inverse $A^{-1}$, so by \cite[Corollary 6.4.3]{BKRSh} we have that $\nu(\cdot, i_U^{-1}\cdot)$ has a density $\rho$ on $(0, +\infty)\times V$ which is locally H\"older continuous.
\end{proof}

\begin{remark}\label{rem:appenbinsidep3d2xim}
{Let $b$ be $C^{1+\alpha}$. Then} the proof is similar if one considers  
\begin{equation*}
\int_{\mathbb R_+\times\mathcal M} \partial_t\phi + \frac {\partial}{\partial x^i} a^{ij}\frac{\partial\phi}{ \partial x^j} +\frac{\partial  b^i  \phi}{\partial x^i} + c\phi \ud \nu = 0,\;\;\; \phi\in C^{\infty}_0(\mathbb R_+ \times \mathcal M).
\end{equation*}
instead of \eqref{eq:appenforPhatpeqo} with the same assumptions on $a^{ij}$ and $c$.
\end{remark}

Now we are ready to prove Harnack-type inequalities for Fokker-Planck equations on manifolds.  Here we assume that $\rho_{\mathcal M}  = \sqrt{\det g_{ij}}$ is the density of $\mu$ in local coordinates $(x^1, \ldots,x^d)$ (see e.g.\ \cite[Chapter 9]{SpVI99}) which is  positive and $C^{\infty}$ as $\mathcal M$ is $C^{\infty}$.

\begin{proposition}\label{prop:appendxHarnackmancewell}
Let $\mathcal M$ be a $C^{\infty}$-smooth $d$-dimensional connected bounded Riemannian manifold, let $\mu$ be the corresponding volume measure, $\alpha\in(0, 1]$, $(a^{i j})_{i, j=1}^d$ be a positive-definite $C^{1+\alpha}$-continuous $(2, 0)$-tensor field, $b$ be a tangent $C^{1+\alpha}$-continuous vector field, $c\in C^{\alpha}(\mathcal M)$. Assume that there exists a function $p:(0, \infty)\times \mathcal M\times \mathcal M \to \mathbb R_+$ such that for any $y\in \mathcal M$ the measure $\nu=p(\cdot,y, \cdot)\ud \lambda_{\mathbb R_+}\dd \mu$ is Borel
and it is a solution of the following equation
\begin{equation}\label{eq:FPKeqcpew}
\int_{\mathbb R_+\times\mathcal M} \partial_t\phi + \frac{\partial}{\partial x^i} a^{ij}\frac{\partial\phi}{\partial x^j} + b^i \frac{\partial \phi}{\partial x^i} + c\phi \ud \nu = 0,\;\;\; \phi\in C^{\infty}_0(\mathbb R_+ \times \mathcal M).
\end{equation}
Assume that 
\begin{equation}\label{eq:PtforapeendaBoic}
P_t\psi(y):= \int_{\mathcal M} \psi(x)p(t, y, x)\ud \mu(x),\;\;\; \psi \in C(\mathcal M), \;\;\; y\in \mathcal M,\;\;\; t\geq 0,
\end{equation}
defines a $C_{0}$-semigroup on $C(\mathcal M)$ with a generator $L_{a,b,c}$ defined for $\psi\in C^{\infty}_0(\mathcal M)$ by
\[
L_{a,b,c}\psi : =  \frac{\partial}{\partial x^i} a^{ij}\frac{\partial \psi}{\partial x^j} + \frac{\partial \ln \rho_{\mathcal M} }{ x^i} \frac{\partial  a^{ij} \psi }{\partial x^j} + \frac{a^{ij}\psi }{\rho_{\mathcal M} }\frac{\partial^2 \rho_{\mathcal M} }{\partial x^i\partial x^j}  - \frac{\partial b^i \psi }{\partial x^i} -b^i\psi \frac{\partial \ln \rho_{\mathcal M} }{\partial x^i} + c\psi.
\]

Then $p(\cdot,\cdot,\cdot)$ is positive locally H\"older continuous, and for any compact $K\subset \mathcal M$ and $t>0$ there exists $C_{t,K,a_K,b_K}>0$ such that $p(t,y,x)\geq C_{t,K,a_Kb_K}$ for any $x, y\in K$, where $a_K>0$ is such that $1/a_K\leq a^{ij}(x)v_iv_j\leq a_K$ for any $x\in K$ and any vector $v_i\in (T_x\mathcal M)^*$ from the cotangent space satisfying $|v_i|^2=g^{ij}(x) v_iv_j=1$ and where $b_K:= \sup_{x\in K}|b^i(x)| + \|Db^i(x)\| + \|Da^{ij}(x)\|$.
\end{proposition}

\begin{proof}
We start with proving that $p$ is locally H\"older continuous in all variables. The fact that $(t, x)\mapsto p(t, \cdot, x)$ is locally H\"older continuous follows from Proposition \ref{prop:appendcontkernel}. Let us show that $(t, y)\mapsto p(t, y, \cdot)$ is locally H\"older continuous. It is sufficient to notice that by \eqref{eq:FPKeqcpew} for any $y\in \mathcal M$ $(x,t)\mapsto p(t,y, x)$ is a fundamental solution of the following PDE on $\mathbb R_+ \times \mathcal M$
$$
\begin{cases}
\partial_t u = L_{a,b,c}u ,\\ 
u(0,\cdot) = u_0\in C^{\infty}_0(\mathcal M).
\end{cases}
$$
Thus by a standard trick for any $x\in \mathcal M$ the map $(y, t)\mapsto p(t,y, x)$ is a weak fundamental solution for the adjoint equation $\partial_t u = L_{a,b,c}^*u$, {as for any $f, u_0\in C_0^{\infty}$ and $\zeta\in C_0^{\infty}(\mathbb R_+)$ we have that
\begin{multline*}
\frac {\partial}{\partial t}\langle f\zeta(t), P^*_t u_0 \rangle_{L^2(\mu)}  = \frac {\partial}{\partial t}\langle P_t f\zeta(t), u_0 \rangle_{L^2(\mu)} \\
= \langle P_t L_{a, b,c} f\zeta(t) + P_t  f\zeta'(t), u_0 \rangle_{L^2(\mu)} =\langle L_{a, b,c} f\zeta(t) +  f\zeta'(t), P_t^*u_0 \rangle_{L^2(\mu)},
\end{multline*}
so $u(t, \cdot) = P_t^*u_0(\cdot)$, $t\geq 0$, solves $\partial_t u = L_{a,b,c}^*u$ weakly, where $P_t^*$ is defined by 
\begin{equation*}
P^*_t\psi(x):= \int_{\mathcal M} \psi(y)p(t, y, x)\ud \mu(y),\;\;\; \psi \in C(\mathcal M), \;\;\; x\in \mathcal M,\;\;\; t\geq 0,
\end{equation*}
 thanks to \eqref{eq:PtforapeendaBoic}.} Therefore by Proposition \ref{prop:appendcontkernel} and Remark \ref{rem:appenbinsidep3d2xim} $(y, t)\mapsto p(t,y, x)$ is locally H\"older continuous.

Let us now turn to the last part of the proposition concerning Harnack-type inequalities. Without loss of generality let $K$ be connected and let $ t<1$. As $K$ is a compact, there exist local charts $U_1, \ldots, U_N\subset \mathcal M$ such that $K\subset \cup_{n=1}^N U_n$. Without loss of generality we may assume that these charts are balls, i.e.\ the corresponding maps map $U_n$'s into balls in $\mathbb R^d$. Fix $x,y\in K$. First assume that $x,y\in U_n$ for some $n=1, \ldots, N$. Fix any $z\in U_n$. Then by \cite[Theorem 8.1.3]{BKRSh} (and as $U_n$ is a ball in local coordinates) for any $0<s<t<1$ there exists a constant $C_1$ depending only on $s$, $t$, $a_K$, and $b_K$ such that 
\begin{equation}\label{eq:appenHarnf3}
p(t,x,y)\geq C_1 p(s,x,z).
\end{equation} 
In particular, $p(t,x,y)\geq \tfrac{C_1}{\mu(U_n)} \int_{U_n}p(s,x,z)\ud \mu(z)$. 

Let us estimate the latter term. Let the semigroup $(P_t)_{t\geq 0}$ be defined by \eqref{eq:PtforapeendaBoic} with the generator $L_{a, b}$. Fix some nonnegative $\psi\in C^{\infty}_0(U_n)$ so that $\psi (x)=1$ and $\psi_{\infty}\leq 1$. Then
\begin{multline*}
 \int_{U_n}p(s,x,z)\ud \mu(z) = P_s \mathbf 1_{U_n}(x) \geq P_s \psi(x)  = 1 + \int_0^s \partial_r P_r \psi(x)\ud r\\
  = 1+\int_0^s P_r L_{a, b}\psi(x)\ud r \geq 1+ \min_{y\in U_n} L_{a, b}\psi(y)\int_0^s P_r \mathbf 1_{\mathcal M}(x) \ud r\\
 \geq 1+ s(0\wedge\min_{y\in U_n} L_{a, b}\psi(y))\max_{y\in \mathcal M, 0\leq r\leq 1} P_r \mathbf 1_{\mathcal M}(y)  = 1- C_2s,
\end{multline*}
where $C_2\geq 0$ depends only on $\psi$, $a_K$, and $b_K$, $\max_{y\in \mathcal M, 0\leq r\leq 1} P_r \mathbf 1_{\mathcal M}(y)<\infty$ as the expression under the maximum is continuous in both $r$ and $y$ and as maximum is taken over a compact set, where we use that $P_tf\geq 0$ for $f\geq 0$ due to \eqref{eq:PtforapeendaBoic}, and where $P_r \psi(x)\to \psi(x)=1$ as $r\to 0$. Therefore $p(t,x,y)\geq \tfrac{C_1}{\mu(U_n)} (1 - C_2s)$. By choosing $s$ close enough to $0$ we get the desired estimate. By minimizing $C_1$ and $C_2$ over all charts we make these estimates independent of $n=1, \ldots, N$.

If $x$ and $y$ do not belong to one $U_n$, it is possible to construct a chain $x_0, \ldots, x_N\in K$ such that $x_n$ and $x_{n+1}$ are in the same local chart for any $n=0, \ldots N-1$ and such that $x_0=x$ and $x_N=y$. Then first by the considerations above we obtain that $p(t/N, x_0, x_1)>C'_{t/N,K,a_K,b_K}$ and by \eqref{eq:appenHarnf3} we get that $p(t(n+1)/N, x, x_{n+1})>C'_{t/N,K,a_K,b_K}p(tn/N, x, x_n)$ for any $n=1, \ldots, N-1$ for some universal constant $C'_{t/N,K,a_K,b_K}>0$ depending only on $t$, $K$, $a_K$, and $b_K$. A simple computation terminates the proof.
\end{proof}

\bibliographystyle{plain}

\end{document}